\documentclass{birkjour} 

\newcommand*\lesubsection[1]{\vskip4ex\noindent{\normalsize\bf#1}\addcontentsline{toc}{lesubsection}%
                            {\numberline{}\quad#1}\nopagebreak\vspace{1ex}\newline\nopagebreak}

 \newcommand*{\bal}{\begin{aligned}}
 \newcommand*{\eal}{\end{aligned}}
 \newcommand*{\uti}[1]{(#1)\space}
 \newcommand*{\qa}{,\qquad}
 \newcommand*{\qb}{,\quad}
 \newcommand*{\mf}[1]{\boldsymbol{#1}} 
 \newcommand*{\ci}{\mathaccent"7017 }     
 \newcommand*{\wheq}{\mathrel{\wh{=}}}
 \newcommand*{\hb}[1]{\hbox{$#1$}} 
 \newcommand*{\sdot}{\!\cdot\!}
 \newcommand*{\sco}{\kern2pt\colon\kern2pt}
 \newcommand*{\sn}{\kern1pt|\kern1pt}
 \newcommand*{\bsn}{\kern1pt\big|\kern1pt}
 \newcommand*{\ssm}{\!\setminus\!}
 \newcommand*{\esdot}{\hbox{$[{}\sdot{}]$}}
 \newcommand*{\eea}[1]{\hbox{$[\![#1]\!]$}}
 \newcommand*{\vsdot}{\hbox{$\vert\sdot\vert$}}
 \newcommand*{\Vsdot}{\hbox{$\Vert\sdot\Vert$}}
 \newcommand*{\hh}[1]{{\textbf{#1}}} 
 \newcommand*{\npb}{\postdisplaypenalty=10000}
 \newcommand*{\pe}{\hbox{$[{}\sdot{},{}\sdot{}]$}}
 \newcommand*{\pr}{\hbox{$(\cdot,\cdot)$}}
 \newcommand*{\prsn}{\hbox{$(\cdot\sn\cdot)$}}
 \newcommand*{\pw}{\hbox{$\dl{}\sdot{},{}\sdot{}\dr$}}
 \newcommand*{\mfA}{{\mf{A}}} 
 \newcommand*{\mfE}{{\mf{E}}} 
 \newcommand*{\mfF}{{\mf{F}}} 
 \newcommand*{\mfd}{{\mf{d}}} 
 \newcommand*{\mfu}{{\mf{u}}} 
 \newcommand*{\mfv}{{\mf{v}}} 
 \newcommand*{\mfnu}{{\mf{\nu}}} 
 \newcommand*{\mfom}{{\mf{\om}}} 
 \newcommand*{\mfpl}{{\mf{\pl}}} 
 \newcommand*{\bka}{{\bar\ka}} 
 \newcommand*{\bs}{{\bar s}} 
 \newcommand*{\dlim}[2]{\substack{{#1}\\{#2}}}
 \newcommand*{\bid}[2]{\Bigl({{#1}\atop{#2}}\Bigr)} 
 \newcommand*{\wh}{\widehat}
 
 \newsavebox{\Prel}
 \sbox{\Prel}{\begin{picture}(2,6)(0,0)\put(1,2.5){\circle*{2}}\end{picture}}
 \newcommand*{\btdot}{\mathrel{\usebox{\Prel}}}
 \newsavebox{\iPrel}
 \sbox{\iPrel}{\begin{picture}(2.4,6)(0,0)\put(1.2,1.5){\circle*{1.2}}\end{picture}}
 \newcommand*{\ibtdot}{\mathrel{\usebox{\iPrel}}}     
 \newsavebox{\Ptop}
 \sbox{\Ptop}{\begin{picture}(2,2)(-1,-1)\put(0,0){\circle*{2}}\end{picture}}
 \newcommand*{\thdot}[3]{\mbox{\slap{\kern#3ex\raisebox{#2ex}{\usebox{\Ptop}}}{$#1$}}}
 \newsavebox{\iPtop}
 \sbox{\iPtop}{\begin{picture}(2,2)(0,0)\put(0,0){\circle*{1.2}}\end{picture}}
 \newcommand*{\ithdot}[3]{\mbox{\slap{\kern#3ex\raisebox{#2ex}{\usebox{\iPtop}}}{$#1$}}}
 \newcommand*{\thBN}{\rlap{\thdot{\BN}{1.9}{.6}}{\ph{\BN}}}
 \newcommand*{\thsZ}{\thdot{\sZ}{1.9}{.45}}
 \newcommand*{\ph}{\phantom}
 \newcommand*{\vph}{\vphantom}
 \newcommand*{\cip}{\ci{\vph{a}}\kern.2ex}                        
 \newcommand*{\dotp}{\thdot{\vph{a}}{1.3}{0}\;}  
 \newcommand*{\idotp}{\ithdot{\vph{a}}{1}{0}\;}  
 \newcommand*{\ihati}{{\textstyle\hat{\raisebox{.5ex}{${\scriptstyle\imath}$}}}}
 \newcommand*{\iihati}{{\textstyle\hat{\raisebox{.8ex}{${\scriptscriptstyle\imath}$}}}}
 \newcommand*{\hr}{\hookrightarrow}
 \newcommand*{\ra}{\rightarrow}
 \newcommand*{\dl}{\langle}
 \newcommand*{\dr}{\rangle}
 \newcommand*{\sdh}{\stackrel{d}{\hookrightarrow}}
 \newcommand*{\doteqeta}{\underset{\mbox{\raisebox{1ex}{$\scriptstyle\eta$}}}\doteq}
 \newcommand*{\simeta}{\underset{\mbox{\raisebox{1ex}{$\scriptstyle\eta$}}}\sim}
 \newcommand*{\slap}[1]{\hspace{0pt}\hbox to 0pt{#1\hss}}
 \newcommand*{\card}{\mathop{\rm card}\nolimits}
 \newcommand*{\diag}{{\mathop{\rm diag}\nolimits}}
 \newcommand*{\dist}{\mathop{\rm dist}\nolimits}
 \newcommand*{\tdiv}{\mathop{\rm div}\nolimits}
 \newcommand*{\divgrad}{\mathop{\rm div\kern1pt grad}\nolimits}
 \newcommand*{\dom}{\mathop{\rm dom}\nolimits}
 \newcommand*{\essup}{\mathop{\mbox{\rm ess\kern1pt sup}}}
 \newcommand*{\grad}{\mathop{\rm grad}\nolimits}
 \newcommand*{\HS}{{\rm HS}}
 \newcommand*{\id}{{\rm id}}
 \newcommand*{\im}{\mathop{\rm im}\nolimits}
 \newcommand*{\tsupp}{{\rm supp}}
 \newcommand*{\unif}{{\rm unif}}
 \newcommand*{\End}{{\rm End}}
 \newcommand*{\Laut}{{\mathcal L}{\rm aut}}
 \newcommand*{\Lis}{{\mathcal L}{\rm is}}
 \newcommand*{\loc}{{\rm loc}}
 \renewcommand*{\Re}{\mathop{\text{\rm Re}}\nolimits}
 \newcommand*{\imi}{{i\kern1pt}}
 \newcommand*{\is}{\subset}
 \newcommand*{\bt}{\bullet}
 \newcommand*{\es}{\emptyset}
 \newcommand*{\iy}{\infty}
 \newcommand*{\mt}{\mapsto}
 \newcommand*{\pl}{\partial}
 \newcommand*{\pa}{\partial^\alpha}
 \newcommand*{\sh}{\sharp}
 \newcommand*{\tri}{\triangle}
 \newcommand*{\cona}{\kern-1pt}
 \newcommand*{\coU}{\kern-1pt}
 \newcommand*{\coV}{\kern-1pt}
 \newcommand*{\coW}{\kern-1pt}
 \newcommand*{\coY}{\kern-1pt}
 \newcommand*{\al}{\alpha}
 \newcommand*{\ba}{\beta}
 \newcommand*{\da}{\delta}
 \newcommand*{\Ga}{\Gamma}
 \newcommand*{\ga}{\gamma}
 \newcommand*{\mfga}{\rlap{$\mf{\ga}$}{\kern.3pt\mf{\ga}}} 
 \newcommand*{\ka}{\kappa}
 \newcommand*{\tk}{{\tilde\ka}}
 \newcommand*{\tkk}{{\tilde{\ka}\ka}}
 \newcommand*{\ktk}{{\ka\tilde{\ka}}}
 \newcommand*{\Lda}{\Lambda}
 \newcommand*{\lda}{\lambda}
 \newcommand*{\na}{\nabla}
 \newcommand*{\Om}{\Omega}
 \newcommand*{\om}{\omega}
 \newcommand*{\pO}{{\pl\Omega}}
 \newcommand*{\sa}{\sigma}
 \newcommand*{\ta}{\theta}
 \newcommand*{\za}{\zeta}
 \newcommand*{\ve}{\varepsilon}
 \newcommand*{\vp}{\varphi}
 
 \newcommand*{\BgB}{(B,g_B)} 
 \newcommand*{\EE}{(E,E)}
 \newcommand*{\EeEn}{(E_1,E_0)}
 \newcommand*{\EeEz}{(E_1,E_2)}
 \newcommand*{\EnEe}{(E_0,E_1)}
 \newcommand*{\BHE}{(\BH,E)}   
 \newcommand*{\cBHE}{(\ci\BH,E)} 
 
 \newcommand*{\KmKn}{(\BK^m,\BK^n)}
 \newcommand*{\Knm}{\BK^{n\times m}}
 
 \newcommand*{\Mg}{(M,g)}
 \newcommand*{\MBR}{(M,\BR)}
 \newcommand*{\MW}{(M,W)}
 
 \newcommand*{\QmE}{(Q^m,E)}
 \newcommand*{\Qmgm}{(Q^m,g_m)}
 
 \newcommand*{\Rdgd}{(\BR^d,g_d)} 
 \newcommand*{\Rdegde}{(\BR^{d+1},g_{d+1})} 
 \newcommand*{\gRgE}{(\gR,\gE)}
 \newcommand*{\gRcgE}{(\gR,\ci\gE)}
 \newcommand*{\cgRcgE}{(\ci\gR,\ci\gE)}
 \newcommand*{\RhdE}{(\BR^d,E)}
 \newcommand*{\RhdmeE}{(\BR^{d-1},E)}
 \newcommand*{\RmF}{(\BR^m,F)}
 \newcommand*{\Rmgm}{(\BR^m,g_m)}
 \newcommand*{\Rm}{(\BR^m)}
 \newcommand*{\RmE}{(\BR^m,E)}
 \newcommand*{\RnE}{(\BR^n,E)}
 \newcommand*{\RRc}{(R,R^c)}
 \newcommand*{\cRcRc}{(\cR,\cR^c)}
 \newcommand*{\BXE}{(\BX,E)}
 \newcommand*{\BXEe}{(\BX,E_1)}
 \newcommand*{\BXEn}{(\BX,E_0)}
 \newcommand*{\BXEz}{(\BX,E_2)}
 \newcommand*{\iBXhi}{\BX_\iihati}
 \newcommand*{\BXhiE}{(\BX_\ihati,E)}
 \newcommand*{\XnXe}{(X_0,X_1)}
 \newcommand*{\XY}{(X,Y)}
 \newcommand*{\sZE}{(\sZ,E)}
 \newcommand*{\thsZE}{(\thsZ,E)}
 \newcommand*{\sZEn}{(\sZ,E_0)}
 \newcommand*{\sZEe}{(\sZ,E_1)}
 \newcommand*{\sZEz}{(\sZ,E_2)}
 \newcommand*{\thsZX}{(\thsZ,X)}
 \newcommand*{\BUC}{BU\kern-.3ex C}
 \newcommand*{\LCM}{\textrm{LCM}}  
 \newcommand*{\BB}{{\mathbb B}}
 \newcommand*{\BC}{{\mathbb C}}
 \newcommand*{\BE}{{\mathbb E}}
 \newcommand*{\BH}{{\mathbb H}}
 \newcommand*{\BJ}{{\mathbb J}}
 \newcommand*{\BK}{{\mathbb K}}
 \newcommand*{\BN}{{\mathbb N}}
 \newcommand*{\BR}{{\mathbb R}}
 \newcommand*{\BW}{{\mathbb W}}
 \newcommand*{\BX}{{\mathbb X}}
 \newcommand*{\BY}{{\mathbb Y}}
 \newcommand*{\BZ}{{\mathbb Z}}
 \newcommand*{\cA}{{\mathcal A}}
 \newcommand*{\cD}{{\mathcal D}}
 \newcommand*{\cF}{{\mathcal F}}
 \newcommand*{\cH}{{\mathcal H}}
 \newcommand*{\cL}{{\mathcal L}}
 \newcommand*{\cM}{{\mathcal M}}
 \newcommand*{\cR}{{\mathcal R}}
 \newcommand*{\cS}{{\mathcal S}}
 \newcommand*{\cT}{{\mathcal T}}
 \newcommand*{\cW}{{\mathcal W}}
 \newcommand*{\cX}{{\mathcal X}}
 \newcommand*{\cY}{{\mathcal Y}}
 \newcommand*{\cZ}{{\mathcal Z}}
 \newcommand*{\gE}{{\mathfrak E}}
 \newcommand*{\gF}{{\mathfrak F}}
 \newcommand*{\gK}{{\mathfrak K}}
 \newcommand*{\gN}{{\mathfrak N}}
 \newcommand*{\gR}{{\mathfrak R}}
 \newcommand*{\gss}{{\mathfrak s}}
 \newcommand*{\sA}{{\mathsf A}}
 \newcommand*{\sfa}{{\mathsf a}}
 \newcommand*{\sB}{{\mathsf B}}
 \newcommand*{\sC}{{\mathsf C}}
 \newcommand*{\sfm}{{\mathsf m}}
 \newcommand*{\sZ}{{\mathsf Z}}
 \makeatletter
 \newif\ifinany@
 \newcount\column@
 \def\column@plus{%
    \global\advance\column@\@ne
 }
 \newcount\maxfields@
 \def\add@amps#1{%
    \begingroup
        \count@#1
        \DN@{}%
        \loop
            \ifnum\count@>\column@
                \edef\next@{&\next@}%
                \advance\count@\m@ne
        \repeat
    \@xp\endgroup
    \next@
 }
 \def\Let@{\let\\\math@cr}
 \def\restore@math@cr{\def\math@cr@@@{\cr}}
 \restore@math@cr
 \def\default@tag{\let\tag\dft@tag}
 \default@tag

 \newbox\strutbox@
 \def\strut@{\copy\strutbox@}
 \addto@hook\every@math@size{%
  \global\setbox\strutbox@\hbox{\lower.5\normallineskiplimit
         \vbox{\kern-\normallineskiplimit\copy\strutbox}}}

 \renewcommand{\start@aligned}[2]{%
    \RIfM@\else
        \nonmatherr@{\begin{\@currenvir}}%
    \fi
    \null\,%
    \if #1t\vtop \else \if#1b \vbox \else \vcenter \fi \fi \bgroup
        \maxfields@#2\relax
        \ifnum\maxfields@>\m@ne
            \multiply\maxfields@\tw@
            \let\math@cr@@@\math@cr@@@alignedat
        \else
            \restore@math@cr
        \fi
        \Let@
        \default@tag
        \ifinany@\else\openup\jot\fi
        \column@\z@
        \ialign\bgroup
           &\column@plus
            \hfil
            \strut@
            $\m@th\displaystyle{##}$%
           &\column@plus
            $\m@th\displaystyle{{}##}$%
            \hfil
            \crcr
 }
 \renewenvironment{aligned}[1][c]{%
    \start@aligned{#1}\m@ne
 }{%
    \crcr\egroup\egroup
 }
 \makeatother

 \def\cprime{$'$}

\newcommand*{\extraindent}{${}$\kern10pt}       
\newcommand*{\extraroman}{\rm }                 

 \newtheorem{thm}{Theorem}[section] 
 \newtheorem{cor}[thm]{\qquad Corollary}
 \newtheorem{lem}[thm]{\qquad Lemma}
 \newtheorem{pro}[thm]{\qquad Proposition}
 \newtheorem{thm(A)}[thm]{\qquad Theorem}       
 \newtheorem{ex(A)}[thm]{\qquad Example}        
 \newtheorem{rem(A)}[thm]{\qquad Remark}        
 \theoremstyle{remark}

 \makeatletter                                              
 \newtheorem*{proofTheoremI.P}{\quad\rm{\bfseries Proof of  
 Theorem~\@newref{thm-I.P}}}                                
 \newtheorem*{proofTheoremI.S}{\quad\rm{\bfseries Proof of  
 Theorem~\@newref{thm-I.S}}}                                
 \makeatother                                               
 \numberwithin{equation}{section}
\begin{document} 
\let\oldvec\vec
%
%
%
%
%
%
%
%
%

\title[Cauchy Problems for Parabolic Equations]
 {Cauchy Problems for Parabolic Equations  
 in Sobolev--Slobodeckii and H\"older Spaces on  
 Uniformly Regular Riemannian Manifolds}

\author[Herbert Amann]{Herbert Amann}

\address{%
Math.\ Institut\\ 
Universit\"at Z\"urich\\
Winterthurerstr.~190\\ 
CH 8057 Z\"urich\\ 
Switzerland}
\email{\sf herbert.amann@math.uzh.ch}
\date{January 1, 2004}
\dedicatory{Dedicated to Professor Jan Pr\"u\ss\ on the occasion 
of his retirement}

\begin{abstract}
In this paper we establish optimal solvability results~---~maximal regularity 
theorems~---~for the Cauchy problem for linear parabolic differential 
equations of arbitrary order acting on sections of tensor bundles over 
boundaryless complete Riemannian manifolds~$\Mg$ with bounded geometry. 
We employ an anisotropic extension of the Fourier multiplier theorem for 
arbitrary Besov spaces introduced in~\cite{Ama97b}. This allows for a 
unified treatment of Sobolev--Slobodeckii and little H\"older spaces. 
In the flat case 
\hb{\Mg=(\BR^m,|dx|^2)} we recover classical results for 
Petrowskii-parabolic Cauchy problems. 
\end{abstract}

\maketitle
\section{Introduction\label{sec-I}} 
\extraindent 
It is well-known that parabolic differential equations play an important 
role in mathematics as well as in more applied sciences, like physics, 
chemistry, biology, etc. As a rule, sophisticated and complex 
environments are modeled by (systems of) quasilinear or even fully 
nonlinear equations. A~particularly interesting and important class of 
nonlinear equations occurring inside mathematics is related to heat flow 
methods in differential geometry. In such and many other intricate settings 
even local well-posedness is far from being easily established, if known at 
all. 
 
\smallskip 
In geometry in particular, it is often convenient, or even necessary, to deal 
with classes of functions possessing relatively high regularity properties. 
Moreover, it is frequently easier and more appropriate to handle functions 
which are differentiable in the usual point-wise rather than the generalized 
sense of distributions. 
 
\smallskip 
It is a pivotal step in the study of nonlinear parabolic equations to 
establish maximal regularity results for linear equations. With the help of 
such tools it is then relatively straightforward to prove the local 
well-posedness of nonlinear problems by more or less standard linearization 
techniques. 

\smallskip 
This paper contains maximal regularity results 
in Sobolev--Slobodeckii and H\"older spaces of arbitrary order for 
linear parabolic equations acting on sections of tensor bundles over 
a vast class of, generally noncompact, 
Riemannian manifolds. We employ a Fourier-analytic approach which allows 
for a unified treatment of all these function space settings at one stroke.  
In order not to overburden this already long paper, we restrict 
ourselves to manifolds without boundary. Boundary value problems will be 
treated elsewhere. 

\smallskip 
For the presentation of our results we need 
some~---~rather lengthy~---~preparation on concepts and definitions. 
We begin by fixing basic syntax. 

\smallskip 
Let $E$, $E_1$,~$E_2$ be Banach spaces over 
\hb{\BK=\BR} or 
\hb{\BK=\BC}. Then $\cL\EeEz$ is the Banach space of the continuous 
linear maps from~$E_1$ into~$E_2$ endowed with the uniform operator norm, 
and 
\hb{\cL(E):=\cL\EE}. By $\Lis\EeEz$ we mean the open subset of $\cL\EeEz$ 
of all isomorphisms therein, and 
\hb{\Laut(E):=\Lis\EE}. We write 
\hb{\prsn} and 
\hb{\vsdot} for the Euclidean inner product and norm, respectively, 
on~$\BK^n$. We identify 
\hb{a\in\cL\KmKn} with its matrix representation 
\hb{[a^{ij}]\in\Knm} with respect to the standard bases of 
$\BK^m$ and~$\BK^n$, if no confusion seems likely. We endow $\Knm$ with the 
Hilbert--Schmidt norm which means that the identification 
\hb{\Knm=\cL\KmKn} applies. 
\lesubsection{Tensor Bundles} 
\extraindent 
Next we collect the needed facts on tensor bundles and refer to~\cite{Ama12b} 
or, of course, to~\cite{Die69bIII} for more details and explanations. 

\smallskip 
Throughout this paper: 
$$ 
\bal 
\bt\quad 
&\Mg\text{ is a smooth $m$-dimensional Riemannian manifold}\cr 
\noalign{\vskip-1\jot} 
&\text{with or without boundary}.\cr 
\bt\quad 
&F=\bigl(F,\prsn_F\bigr)\text{ is an $n$-dimensional complex inner product 
 space},\cr 
\noalign{\vskip-1\jot} 
&\text{ where }n\in\BN.
\eal 
\npb 
$$ 
If 
\hb{n=0}, then 
\hb{F:=\{0\}} and obvious identifications apply in the following. 

\smallskip 
As usual, $TM$~denotes the tangent and $T^*M$ the cotangent bundle, and 
\begin{equation}\label{I.du} 
\pw\sco T^*M\times TM\ra C^\iy\MBR 
\end{equation} 
the (fibre-wise defined) duality pairing. We always suppose 
$$ 
\bt\quad 
\sa,\tau\in\BN.
$$ 
Then 
\hb{T_\tau^\sa M:=TM^{\otimes\sa}\otimes T^*M^{\otimes\tau}} is the 
$(\sa,\tau)$-tensor bundle over~$M$ consisting of tensors being contravariant 
of order~$\sa$ and covariant of order~$\tau$. In particular, 
\hb{T_0^1M=TM}, 
\ \hb{T_1^0M=T^*M}, and 
\hb{T_0^0M=M\times\BR}, a~trivial line bundle. The covariant metric induced 
by~$g$ on~$T^*M$ is written~$g^*$. We endow~$T_\tau^\sa M$ with the bundle 
metric 
\hb{\prsn_\tau^\sa:=g^{\otimes\sa}\otimes g^{*\otimes\tau}} and the 
corresponding bundle norm 
\hb{\vsdot_\tau^\sa:=\bigl(u\mt\sqrt{(u\sn u)_\tau^\sa}\,\bigr)}. 

\smallskip 
The vector bundle of \hbox{$F$-valued} $(\sa,\tau)$-tensors, 
\hb{T_\tau^\sa M\otimes F}, is defined by 
$$ 
\dl a\otimes f,b\dr:=f\dl a,b\dr 
\qa a\in T_\tau^\sa M 
\qb f\in F 
\qb b\in T_\sa^\tau M. 
$$ 
Here we use the fact that 
\hb{T_\tau^\sa M=(T_\sa^\tau M)'} with respect to the duality pairing~%
\hb{\pw} induced by~\eqref{I.du}. We endow 
\hb{T_\tau^\sa M\otimes F} with the bundle metric 
$$ 
\prsn_{\tau,F}^\sa:=\prsn_\tau^\sa\otimes\prsn_F 
=\prsn_\tau^\sa\prsn_F 
$$ 
and set 
$$ 
\bt\quad 
V=V_{\coV\tau}^\sa=V_{\coV\tau}^\sa(F) 
:=\bigl(T_\tau^\sa M\otimes F,\ \prsn_{\tau,F}^\sa\bigr). 
$$ 
In particular, 
$$ 
V_0^0(F)=(M\times R)\otimes F\wheq M\times F, 
$$ 
a~trivial complex vector bundle of rank~$n$ over$M$ if 
\hb{n\geq1}, and, if 
\hb{n=0}, then  
\hb{V_0^0(F)\wheq M\times\BR}. Here and below, 
\hb{{}\wheq{}}~means `natural identification'. 

\smallskip 
Let 
\hb{W=\bigl(W,\prsn_W\bigr)} be any smooth metric vector bundle over~$M$. 
Then $W_{\coW p}$~is its fiber over 
\hb{p\in M} and 
\hb{\Ga(W)=\Ga\MW} is the \hbox{$\BR$-vector} space of all sections of~$W$ 
(no topology). By~$C^k(W)$, 
\ \hb{k\in\BN\cup\{\iy\}}, we mean the $C^k\MBR$-module of all 
$C^k$~sections, and 
\hb{C^0=C}. 

\smallskip 
We denote by~$dv_g$ the Riemann--Lebesgue volume measure on~$M$. Then 
$L_q(W)$~is, for 
\hb{1\leq q\leq\iy}, the Banach space of all (equivalence classes of) 
\hbox{$dv_g$-measurable} sections~$u$ of~$W$ for which the norm 
$$ 
\|u\|_q 
:=\Bigl(\int_W|u|_W^q\,dv_g\Bigr)^{1/q} 
$$ 
if 
\hb{q<\iy}, respectively 
\hb{\|u\|_\iy:=\essup_W|u|_W} if 
\hb{q=\iy}, is finite. 

\smallskip 
Assume 
\hb{(x^1,\ldots,x^m)} is a coordinate system on some open coordinate 
patch $U$ of~$M$. We set 
$$ 
\frac\pl{\pl x^{(i)}} 
:=\frac\pl{\pl x^{i_1}}\otimes\cdots\otimes\frac\pl{\pl x^{i_r}} 
\qb dx^{(i)}:=dx^{i_1}\otimes\cdots\otimes dx^{i_r} 
$$ 
for 
\hb{(i):=(i_1,\ldots,i_r)\in\BJ_r:=\{1,\ldots,m\}^r}. Then 
\begin{equation}\label{I.b} 
\frac\pl{\pl x^{(i)}}\otimes dx^{(j)} 
\qa (i)\in\BJ_\sa 
\qb (j)\in\BJ_\tau, 
\end{equation} 
is a coordinate frame for~$T_\tau^\sa M$ over~$U$. We use the summation 
convention with $(i)$ and~$(j)$ running through $\BJ_\sa$ and~$\BJ_\tau$, 
respectively, Then 
\hb{a\in V} has on~$U$ the local representation 
\begin{equation}\label{I.aU} 
a\sn U=a_{(j)}^{(i)}\frac\pl{\pl x^{(i)}}\otimes dx^{(j)} 
\qa a_{(j)}^{(i)}\in F^U, 
\npb 
\end{equation} 
where $F$~has to be replaced by~$\BR$ if 
\hb{n=0}. 

\smallskip 
Given 
\hb{a\in\Ga\bigl(V_{\tau+\sa}^{\sa+\tau+\rho}(\cL(F))\bigr)} and 
\hb{b\in\Ga(V_{\tau+\rho}^\sa)} with 
\hb{\rho\in\BN}, we define the 
\emph{complete\footnote{`Complete' means that we contract over a 
maximal number of indices.} contraction} 
\hb{a\btdot b} locally by 
$$ 
(a\btdot b)_{(j)}^{(i)} 
:=a_{(j)(r)}^{(i)(s)}\,b_{(s)}^{(r)} 
$$ 
with $(i)$ and~$(r)$ running through~$\BJ_\sa$, 
\,$(j)$~through~$\BJ_\tau$, and $s$ through~$\BJ_{\tau+\rho}$,~and where 
\hb{(i)(s):=(i_1,\ldots,i_\sa,s_1,\ldots,s_{\tau+\rho})}, etc. Then 
$$ 
\Ga\bigl(V_{\tau+\sa}^{\sa+\tau+\rho}(\cL(F))\bigr) 
\times\Ga(V_{\tau+\rho}^\sa) 
\ra\Ga(V) 
\qb (a,b)\mt a\btdot b 
$$ 
is a bilinear vector bundle map which is continuous in the sense that 
\begin{equation}\label{I.ab} 
|a\btdot b|_V 
\leq|a|_{V_{\tau+\sa}^{\sa+\tau+\rho}(\cL(F))}\,|b|_{V_{\tau+\rho}^\sa}. 
\end{equation} 
We also need to use the complexification~$V_\BC$ of~$V$, defined by 
$$ 
V_\BC:=(T_\tau^\sa M\otimes\BC)\otimes F 
=T_\tau^\sa M\otimes F+\imi T_\tau^\sa M\otimes F, 
\npb 
$$ 
and continue to write 
\hb{a\btdot{}}~for the complexification~%
\hb{(a\btdot{})_\BC} of~%
\hb{a\btdot{}}. 

\smallskip 
For abbreviation, 
\hb{\cT M:=C^\iy(TM)}, the $C^\iy\MBR$-module of smooth vector fields on~$M$. 
Then 
\hb{\na=\na_{\cona g}} denotes the Levi--Civita connection on~$\cT M$. 
The same symbol is used for its extension over~$C^1(T_\tau^\sa M)$, 
considered as an \hbox{$\BR$-linear} map 
$$ 
\na\sco C^1(T_\tau^\sa M)\ra C(T_{\tau+1}^\sa M) 
\qb v\mt\na v, 
$$ 
where 
\hb{\na=d}, the differential, on~$C^1\MBR$ if 
\hb{\sa=\tau=0}. For 
\hb{k\in\BN} we set 
\hb{\na^{k+1}:=\na\circ\na^k}  with 
\hb{\na^0:=\id}, and 
\hb{\na(v\otimes f):=\na v\otimes f} for 
\hb{v\otimes f\in C(V)}. Then $\na^k$~is an \hbox{$\BR$-linear} map 
$$ 
\na^k\in C^k(V)\ra C(V_{\tau+k}^\sa) 
\qb u\mt\na^ku. 
$$ 
Note that the \hbox{$\BR$-linearity} means `real 
differentiation', although $u$~is complex-valued (see~\eqref{I.aU}). 
\lesubsection{Normally Elliptic Operators} 
\extraindent 
Now we are ready to introduce differential operators. We write 
\hb{\thBN:=\BN\ssm\{0\}} and assume 
\begin{equation}\label{I.r} 
\bt\quad 
r\in 2\thBN. 
\end{equation} 
Let 
\hb{a_j\in C\bigl(V_{\tau+\sa}^{\sa+\tau+j}(\cL(F))\bigr)} for 
\hb{0\leq j\leq r}. We consider the linear differential operator 
\begin{equation}\label{I.A} 
\cA:=\sum_{j=0}^ra_j\btdot\na^j  
\end{equation} 
acting on 
\hb{u\in C^r(V)} by 
\hb{(a_j\btdot\na^j)u:=a_j\btdot(\na^ju)}. With~$\cA$ we associate its 
\emph{principal symbol}~$\gss\cA$ defined by 
$$ 
\gss\cA(\cdot,\xi):=(-1)^{r/2}(a_r\btdot\xi^{\otimes r})\btdot{} 
\qa \xi\in\Ga(T^*M). 
$$  
Note that 
\hb{\gss\cA(\cdot,\xi)\in\Ga\bigl(\End(V_\BC)\bigr)}, the map 
\hb{\xi\mt\gss\cA(\cdot,\xi)} is \hbox{$r$-linear}, and 
$$ 
|\gss\cA(\cdot,\xi)|_{\Ga(\End(V_\BC))} 
\leq|a_r|_{V_{\tau+\sa}^{\sa+\tau+r}(\cL(F))}(|\xi|_1^0)^r 
\npb 
$$ 
for 
\hb{\xi\in\Ga(T^*M)}, by~\eqref{I.ab}. 

\smallskip 
We denote by~$\sa(A)$ the spectrum of a linear operator~$A$ in a given 
complex Banach space and write 
\hb{[\Re z\geq\ve]:=\{\,z\in\BC\ ;\ \Re z\geq\ve\,\}},~etc. 

\smallskip 
Let 
\hb{0<\ve\leq1}. Then $\cA$~is \hh{uniformly normally} 
\hbox{$\ve$\hh{-elliptic}} \hh{on}~$\Mg$ if 
\begin{equation}\label{I.ne} 
\sa\bigl(\gss\cA(p,\xi)\bigr)\is[\Re z\geq\ve] 
\end{equation} 
for each 
\hb{p\in M} and 
\hb{\xi\in T_p^*M} with 
\hb{|\xi|_1^0=1}. It is \hh{uniformly normally elliptic} if 
\eqref{I.ne} holds for some 
\hb{\ve\in(0,1]}. 
\begin{rem(A)}\label{rem-I.ne(a)} 
\extraroman 
It is no restriction of generality to assume \eqref{I.r}. Indeed, 
if $r$~is odd, then  
\hb{\sa\bigl((a_r(p)\btdot(-\xi)^{\otimes r})\btdot\bigr) 
   =-\sa\bigl((a_r(p)\btdot\xi^{\otimes r})\btdot\bigr)}. 
Thus the spectrum of 
\hb{(a_r(p)\btdot\xi^{\otimes r})\btdot{}} cannot be contained 
in one and the same half-space of~$\BC$ for all 
\hb{\xi\in T_p^*M} with 
\hb{|\xi|_1^0=1}.\nolinebreak\hfill\nolinebreak\qed 
\end{rem(A)}  
\begin{rem(A)}\label{rem-I.ne(b)} 
\extraroman 
$\cA$~is called \hh{uniformly strongly 
\hbox{$\ve$-elliptic}}
\ if 
$$ 
\Re\bigl(\gss\cA(\cdot,\xi)\eta\bsn\eta\bigr)_{V_\BC} 
\geq\ve(|\xi|_1^0)^r\,|\eta|_{V_\BC}^2 
\qa \xi\in\Ga(T^*M) 
\qb \eta\in\Ga(V_\BC). 
$$ 
It is obvious that this condition implies the uniform normal 
\hbox{$\ve$-ellipticity} 
of~$\cA$.\nolinebreak\hfill\nolinebreak\qed 
\end{rem(A)}  
\begin{rem(A)}\label{rem-I.ne(c)} 
\extraroman 
If  
\hb{\sa=\tau=n=0}, then 
\hb{V=M\times\BR} and 
\hb{\Ga(V)=\BR^M}. It follows that $\cA$~is uniformly normally 
\hbox{[$\ve$-]elliptic} iff it is uniformly strongly 
\hbox{[$\ve$-]elliptic}. In this case, as usual, $\cA$~is simply called 
\hh{uniformly} \hbox{[$\ve$\hh{-}]\hh{elliptic}}.%
\nolinebreak\hfill\nolinebreak\qed 
\end{rem(A)}  
\begin{rem(A)}\label{rem-I.ne(d)} 
\extraroman 
Assume 
\hb{\Mg=\Rmgm}, where $g_m$~is  the Euclidean metric 
\hb{|dx^2|:=(dx^1)^2+\cdots+(dx^m)^2}. If 
\hb{\sa=\tau=0} and 
\hb{n\geq1}, then 
\hb{V=\BR^m\times F}. We set 
$$ 
D:=-\imi\pl=-\imi(\pl_1,\ldots,\pl_m) 
=-\imi(\pl/\pl x^1,\ldots,\pl/\pl x^m) 
$$  
and use standard multiindex notation. Then we can write~$\cA$ in the form 
\begin{equation}\label{I.AD} 
\cA=\sum_{|\al|\leq r}a_\al D^\al 
\qa a_\al\in C\bigl(\BR^m,\cL(F)\bigr), 
\end{equation} 
and 
\hb{\gss\cA(\cdot,\xi)=\sum_{|\al|=r}a_\al\xi^\al} for 
\hb{\xi\in\BR^m}. Note that the top-order coefficients are real. 
\end{rem(A)}  
\begin{proof} 
This follows from 
\hb{\na_{\cona g_m}=\pl}, the latter being identified with the Fr\'echet 
derivative. 
\end{proof} 
\begin{ex(A)}\label{exa-I.ne(e)} 
\extraroman 
We denote by 
$$ 
g^\sh\sco\Ga(T_{\tau+1}^\sa M)\ra\Ga(T_\tau^{\sa+1}) 
\qb a\mt g^\sh a=:a^\sh 
$$ 
the `index rising' bundle isomorphism defined by 
\hb{(g^\sh\om\sn X)_0^1=\dl\om,X\dr} for 
$\om$~in $\Ga(T^*M)$ and $X$~in $\Ga(TM)$. We write 
$$ 
\sC\sco\Ga(T_{\tau+1}^{\sa+1}M)\ra\Ga(T_\tau^\sa M) 
\qb a\mt\sC a 
$$ 
for the contraction, locally defined by 
\hb{(\sC a)_{(j)}^{(i)}:=a_{(j)(k)}^{(i)(k)}}, with $(i)$~running 
through~$\BJ_\sa$, 
\,$(j)$~through~$\BJ_\tau$, and $k$ through~$\BJ_1$. Then 
$$ 
\tdiv:=\tdiv_g\sco C^1(T_\tau^{\sa+1}M)\ra C(T_\tau^\sa) 
\qb a\mt\tdiv a:=\sC(\na a) 
\npb 
$$ 
is the \emph{divergence} of $C^1$~tensor fields of type 
\hb{(\sa+1,\tau)}. 

\smallskip 
The \emph{gradient}, 
\hb{\grad u=\grad_gu}, of 
\hb{u\in C^1(M)} is given by~$g^\sh\,du$. Thus, if 
\hb{u\in C^1(M)} and 
\hb{a\in C^1(T_1^1M)}, 
$$ 
\tdiv(a\btdot\grad u)=a^\sh\btdot\na^2u+\tdiv(a^\sh)\btdot\na u. 
$$ 
In local coordinates 
$$ 
\tdiv(a\grad u)\sn U 
=\frac1{\sqrt g}\,\frac\pl{\pl x^i}\Bigl(\sqrt{g}\,g^{ij} 
\,\frac{\pl u}{\pl x^j}\Bigr),  
$$ 
$[g^{ij}]$~being the inverse of the fundamental matrix, and 
\hb{\sqrt{g}:=\bigl(\det[g_{ij}]\bigr)^{1/2}}. In particular, 
\hb{\tri=\tri_g:=\divgrad} is the Laplace--Beltrami operator of~$\Mg$. 

\smallskip 
Suppose 
\hb{\sa=\tau=n=0}. Then 
\begin{equation}\label{I.Aa} 
\cA:=-\tdiv(a\grad\cdot{})
\end{equation} 
is uniformly \hbox{$\ve$-elliptic} iff 
$$ 
a^\sh\btdot\xi\otimes\xi=\dl\xi,a^\sh\xi\dr 
\geq\ve(|\xi|_1^0)^2 
\qa \xi\in\Ga(T^*M). 
$$ 
In local coordinates this means 
$$ 
g^{ik}a_k^j\xi_i\xi_j\geq\ve g^{ij}\xi_i\xi_j 
\qa \xi=\xi_i\,dx^i. 
\npb 
$$ 
In particular, $-\tri$~is uniformly 
\hbox{$1$-elliptic}.\nolinebreak\hfill\nolinebreak\qed 
\end{ex(A)}  
\begin{ex(A)}\label{exa-I.ne(f)} 
\extraroman 
The \emph{covariant Laplacian} (or \emph{Bochner Laplacian}) is defined 
by~$\na^*\na$, where $\na^*$~is the formal adjoint of 
\hb{\na\sco C^\iy(V)\ra C^\iy(V_{\coV\tau+1}^\sa)} with respect to the 
$L_2(V_{\coV\tau+1}^\sa)$ inner product. It is known 
(e.g.,~\cite[Appendix~C, Proposition~2.1]{Tay11a}) that 
\hb{\na^*\na=-g^*\btdot\na^2}. Hence 
\hb{\gss\na^*\na(\cdot,\xi)=(|\xi|_1^0)^2} for 
\hb{\xi\in\Ga(T^*M)}. Thus 
\hb{\na^*\na\sco C^2(V)\ra C(V)} is uniformly normally 
\hbox{$1$-elliptic}.\nolinebreak\hfill\nolinebreak\qed 
\end{ex(A)}  
\begin{ex(A)}\label{exa-I.ne(g)} 
\extraroman 
For 
\hb{0\leq k\leq m} let 
\hb{\bigwedge^k:=\bigl(\bigwedge^kT^*M,\prsn_k^0\bigr)} be the 
\hbox{$k$-fold} exterior product of~$T^*M$, considered as a subbundle 
of~$V_k^0$. Then the \emph{Hodge Laplacian} 
$$ 
d\da+\da d\sco C^2({\textstyle\bigwedge^k})\ra C({\textstyle\bigwedge^k}) 
\npb 
$$ 
is uniformly normally \hbox{$1$-elliptic} 
(e.g.,~\cite[Example~10.1.22]{Nico07a} 
and~\cite{Ama12b}).\nolinebreak\hfill\nolinebreak\qed 
\end{ex(A)}
\lesubsection{Uniformly Regular Riemannian Manifolds} 
\extraindent 
In order to proceed further we have to assume that $\Mg$~is a uniformly 
regular Riemannian manifold. The precise definition of this concept, 
which has been introduced in~\cite{Ama12b}, is given in Section~\ref{sec-L}. 
Here we content ourselves with a list of examples which indicates the 
extent of this class. If there is no reference given, proofs are found 
in~\cite{Ama15a}. 
\begin{ex(A)}\label{exa-I.ex(a)} 
\extraroman 
$\Rmgm$ and 
\hb{(\BR^m\times\BR^+,g_{m+1})} are uniformly 
regular.\nolinebreak\hfill\nolinebreak\qed 
\end{ex(A)} 
\begin{ex(A)}\label{exa-I.ex(b)} 
\extraroman 
Every compact manifold is uniformly regular (with respect to any 
metric~$g$).\nolinebreak\hfill\nolinebreak\qed 
\end{ex(A)} 
\begin{ex(A)}\label{exa-I.ex(c)} 
\extraroman 
Products of uniformly regular Riemannian manifolds are uniformly 
regular.\nolinebreak\hfill\nolinebreak\qed 
\end{ex(A)} 
\begin{ex(A)}\label{exa-I.ex(d)} 
\extraroman 
Isometric images of uniformly regular Riemannian manifolds are uniformly 
regular.\nolinebreak\hfill\nolinebreak\qed 
\end{ex(A)} 
\begin{ex(A)}\label{exa-I.ex(e)} 
\extraroman 
\uti{Manifolds with tame ends} 
Let $\BgB$ be an 
\hb{(m-1)}-dimensional compact Riemannian submanifold of~$\Rdgd$, 
\,\hb{d\geq m}, without boundary. Suppose 
\hb{0\leq\al\leq1}. Set 
$$ 
F_\al(B):= 
\bigl\{\,(t,t^\al y)\ ;\ t>1,\ y\in B\,\bigr\} 
\is\BR\times\BR^d=\BR^{d+1}. 
$$ 
Then $F_0(B)$~is an infinite cylinder with base~$B$, and $F_1(B)$~is a 
(blunt) cone over~$B$. We endow~$F_\al(B)$ with the Riemannian 
metric~$g_{F_\al(B)}$ induced by its embedding into $\Rdegde$. Assume 
\hb{M=V_0\cup V_1}, where $V_0$ and~$V_1$ are open, $V_0$ and 
\hb{V_0\cap V_1} are relatively compact, and $(V_1,g)$~is isometric to 
$\bigl(F_\al(B),g_{F_\al(B)}\bigr)$. Then $V_1$~is a \emph{tame end} of~$M$. 
Any Riemannian manifold with finitely many pair-wise disjoint tame ends is 
uniformly regular. In particular, manifolds with cylindrical or `infinite' 
conical ends are uniformly 
regular.\nolinebreak\hfill\nolinebreak\qed 
\end{ex(A)} 
\begin{ex(A)}\label{exa-I.ex(f)} 
\extraroman 
\uti{Manifolds with cuspidal singularities} 
Let $(\Om,\tilde g)$ be a Riemannian manifold with 
nonempty compact boundary~$\pO$. Suppose 
\hb{\ba\geq1}. Let $\ci\Om$ be the interior of~$\Om$. Fix 
\hb{\rho\in C^\iy\bigl(\ci\Om,(0,1]\bigr)} with 
\hb{\rho(x)=\bigl(\dist_\Om(x,\pO)\bigr)^\ba} for $x$ in some sufficiently 
small neighborhood of~$\pO$. Set 
\hb{\Mg:=(\ci\Om,\tilde g/\rho^2)}. Then $\Mg$~is uniformly regular. 

\smallskip 
As an example we see that the Poincar\'e model of the hyperbolic 
\hbox{$m$-space}, 
\hb{\bigl(\BB^m,4\,dx^2/(1-|x|^2)^2\bigr)}, where $\BB^m$~is the open 
unit ball in~$\BR^m$, is a uniformly regular Riemannian manifold. 
\nolinebreak\hfill\nolinebreak\qed 
\end{ex(A)} 
\begin{ex(A)}\label{exa-I.ex(g)} 
\extraroman 
If 
\hb{\pl M=\es}, then $\Mg$~is uniformly regular iff it has bounded geometry. 
By this we mean that it is geodesically complete, has a positive 
injectivity radius, and all covariant derivatives of the curvature tensor 
are bounded. 
\end{ex(A)} 
\begin{proof} 
The necessity part is Theorem~4.1 in~\cite{Ama15a}. The sufficiency statement 
has been shown by D.~Disconzi, Y.~Shao, and G.~Simonett~\cite{DSS16a}. 
\end{proof} 
\begin{rem(A)}\label{rem-I.S}  
\extraroman 
Under the conditions of Example~\ref{exa-I.ex(f)}, 
\ $(\ci\Om,\tilde g)$~is an instance of~a \emph{singular manifold} as 
introduced in~\cite{Ama12b}. If $\cA$~is a uniformly normally elliptic 
differential operator on 
\hb{\Mg:=(\ci\Om,\tilde g/\rho^2)}, then, considered as a differential 
operator on~$(\ci\Om,\tilde g)$, its coefficients degenerate near the 
boundary~$\pO$ (cf.~\cite{Ama16a} for a discussion of this aspect 
in the case of second order scalar 
equations).\nolinebreak\hfill\nolinebreak\qed 
\end{rem(A)} 
\lesubsection{Function Spaces} 
\extraindent 
It has been shown in~\cite{Ama12b} (also see~\cite{Ama12c}) that 
Sobolev--Slobodeckii and H\"older spaces on uniformly regular Riemannian 
manifolds are well-behaved in the sense that they possess the same 
embedding, interpolation, and trace properties as in the classical Euclidean 
case. Moreover, what is most crucial for our purposes, they can be 
characterized by local coordinates induced by a uniformly regular atlas 
(see Theorem~\ref{thm-L.n} below). 

\smallskip 
In order to formulate our results on parabolic differential equations 
we have to introduce these function spaces. Thus we assume throughout that 
$$ 
\bal 
\bt\quad 
&\Mg\text{ is a uniformly regular Riemannian manifold}\cr 
\bt\quad 
&1\leq q\leq\iy. 
\eal  
$$ 
We denote by~$C_c^\iy(V)$ the vector space of smooth sections of~$V$ with 
compact support. Furthermore, 
\hb{\pr_{q,\ta}}~is the real, and 
\hb{\pr_{\iy,\ta}^0}~the continuous interpolation functor of order 
\hb{\ta\in(0,1)} (cf.~\cite[Section~I.2]{Ama95a} for a summary of 
interpolation theory). 

\smallskip 
For 
\hb{k\in\BN} we set 
$$ 
\Vsdot_{k,q} 
:=\sum_{j=0}^k\|\na^j\btdot\|_{L_q(V_{\tau+j}^\sa)}. 
$$ 
Suppose 
\hb{q<\iy}. Then 
\hb{W_{\coW q}^k(V):=\bigl(W_{\coW q}^k(V),\Vsdot_{k,q}\bigr)}, the 
\emph{Sobolev space} of order~$k$ (of sections of~$V$), is the completion 
of~$C_c^\iy(V)$ in~$L_q(V)$ with respect to the norm~%
\hb{\Vsdot_{k,q}}. Hence 
\hb{W_{\coW q}^0(V)=L_q(V)}. If 
\hb{k<s<k+1}, then 
$$ 
W_{\coW q}^s(V):=\bigl(W_{\coW q}^k(V),W_{\coW q}^{k+1}(V)\bigr)_{q,s-k} 
\npb 
$$ 
defines the \emph{Slobodeckii space} of order~$s$. 

\smallskip 
By $BC^k(V)$ we mean the closed \hbox{($\BR$-)linear} subspace of~$C^k(V)$ 
consisting of all 
\hb{u\in C^k(V)} satisfying 
\hb{\|u\|_{k,\iy}<\iy}, and 
\hb{BC:=BC^0}. It is a Banach space with the norm~%
\hb{\Vsdot_{k,\iy}}. If 
\hb{k<s<k+1}, then 
$$ 
BC^s(V):=\bigl(BC^k(V),BC^{k+1}(V)\bigr)_{s-k,\iy} 
$$ 
is the \emph{H\"older space} and 
$$ 
bc^s(V):=\bigl(BC^k(V),BC^{k+1}(V)\bigr)_{s-k,\iy}^0 
\npb 
$$ 
the \emph{little H\"older space} of order~$s$. 
\begin{rem(A)}\label{rem-I.H} 
\extraroman 
Suppose 
\hb{\Mg=\Rmgm} and 
\hb{\sa=\tau=0}. For 
\hb{0<\ta<1} we set 
\begin{equation}\label{I.tq} 
[u]_{\ta,q} 
:=\Bigl(\int_{\BR^m\times\BR^m} 
\Bigl(\frac{|u(x)-u(y)|_F}{|x-y|^\ta}\Bigr)^q 
\,\frac{d(x,y)}{|x-y|^m}\Bigr)^{1/q} 
\qa q<\iy, 
\end{equation} 
and 
\begin{equation}\label{I.td} 
[u]_{\ta,\iy}^\da 
:=\sup_{\dlim{x,y\in\BR^m}{0<|x-y|<\da}} 
\frac{|u(x)-u(y)|_F}{|x-y|^\ta}, 
\end{equation} 
where 
\hb{0<\da\leq\iy} and 
\hb{\esdot_{\ta,\iy}:=\esdot_{\ta,\iy}^\iy}. Then, given 
\hb{k<s<k+1} with 
\hb{k\in\BN}, 
$$ 
\|u\|_{s,q} 
:=\|u\|_{k,q}+\sum_{|\al|=k}[\pa u]_{s-k,q} 
$$ 
is an equivalent norm for $W_{\coW q}^s\RmF$ if 
\hb{q<\iy}, and for $BC^s\RmF$ and $bc^s\RmF$ if 
\hb{q=\iy}. Furthermore, 
\hb{u\in bc^s\RmF} iff 
\hb{u\in BC^k\RmF} and 
\hb{\lim_{\da\ra0}[\pa u]_{s-k,\iy}^\da=0} for 
\hb{\al\in\BN^m} with 
\hb{|\al|=k}. This explains the names `Slobodeckii' and 
`little H\"older' spaces.\nolinebreak\hfill\nolinebreak\qed 
\end{rem(A)} 
It should be observed that definitions \eqref{I.tq} and \eqref{I.td} remain 
meaningful if $F$~is replaced by any Banach space and 
$\BR^m$~by an \hbox{$m$-dimensional} interval. 

\smallskip 
Suppose 
\hb{0\leq s_0<s_1}. Then 
\begin{equation}\label{I.eW} 
W_{\coW q}^{s_1}(V)\sdh W_{\coW q}^{s_0}(V) 
\qa q<\iy, 
\end{equation} 
where 
\hb{{}\hr{}}~means `continuous' and 
\hb{{}\sdh{}} `continuous and dense' injection. Similarly, if 
\hb{0\leq s_0<s_1<s_2} with 
\hb{s_1\notin\BN}, 
\begin{equation}\label{I.eBC} 
BC^{s_2}(V)\sdh bc^{s_1}(V)\hr BC^{s_1}(V)\hr BC^{s_0}(V). 
\end{equation} 
Consequently, 
\begin{equation}\label{I.ebc} 
bc^{s_1}(V)\sdh bc^{s_0}(V) 
\qa s_0,s_1\in\BR^+\ssm\BN. 
\end{equation} 

\smallskip 
In addition, we need anisotropic spaces on `time cylinders' over~$M$. 
For this we assume   
$$ 
\bal 
{}
{\rm(i)}\quad 
    &0<T<\iy\text{ and }J=J_T:=[0,T],\text{ or }J=\BR^+;\cr 
{\rm(ii)}\quad 
    &1/\vec r:=(1,1/r), 
\eal 
$$ 
so that 
\hb{s/\vec r=(s,s/r)} for 
\hb{s\in\BR}. Then we set, for 
\hb{s\in\BR^+}, 
$$ 
W_{\coW q}^{s/\vec r}(V\times J) 
:=L_q\bigl(J,W_{\coW q}^s(V)\bigr)\cap W_{\coW q}^{s/r}\bigl(J,L_q(V)\bigr) 
\qa q<\iy, 
$$ 
and 
\begin{equation}\label{I.bca} 
bc^{s/\vec r}(V\times J) 
:=\BUC\bigl(J,bc^s(V)\bigr)\cap bc^{s/r}\bigl(J,BC(V)\bigr) 
\qa s\notin\BN, 
\end{equation} 
where $\BUC$ means `bounded and uniformly continuous'. 
As mentioned above, these spaces have been investigated in~\cite{Ama12b}, 
and in the anisotropic case in~\cite{Ama12c}, to which we refer for proofs 
of \eqref{I.eW} and~\eqref{I.eBC}. More precisely, in those papers only 
\hb{1<q<\iy} and 
\hb{n=0} have been considered. However, it is straightforward to extend 
those results to the present setting. 

\smallskip 
Suppose that 
\hb{q>1} if 
\hb{s\in\BN}. Then it is shown in~\cite{Ama17a} that 
\begin{equation}\label{I.WW} 
\left. 
\begin{split} 
{}
&u\in W_{\coW q}^{(s+r)/\vec r}(V\times J) 
 \text{ iff }\na^ju\in W_{\coW q}^{s/\vec r}(V_{\tau+j}^\sa\times J)\cr 
\noalign{\vskip-1\jot} 
&\text{for $0\leq j\leq r$ and }\pl_tu\in W_{\coW q}^{s/\vec r}(V\times J). 
\end{split} 
\right.  
\end{equation} 
Similarly, if 
\hb{s\notin\BN}, 
\begin{equation}\label{I.bcb} 
\left. 
\begin{split} 
{}
&u\in bc^{(s+r)/\vec r}(V\times J) 
 \text{ iff }\na^ju\in bc^{s/\vec r}(V_{\tau+j}^\sa\times J)\cr 
\noalign{\vskip-1\jot} 
&\text{for $0\leq j\leq r$ and }\pl_tu\in bc^{s/\vec r}(V\times J). 
\end{split} 
\right.  
\end{equation} 
\begin{rem(A)}\label{rem-I.V} 
\extraroman 
For simplicity, we consider \hbox{$F$-valued} tensor bundles only. 
However, all results of this paper remain valid if $V$~is an arbitrary 
uniformly regular vector bundle endowed with a uniformly regular metric 
and a uniformly regular bundle connection (see~\cite{Ama12c} for 
definitions). In particular, the tensor bundles~$\bigwedge^kT^*M$, 
\ \hb{0\leq k\leq m}, are special instances of this more general setting 
(cf.~\cite{Ama12b}). This puts Example~\ref{exa-I.ne(g)} into 
perspective.\nolinebreak\hfill\nolinebreak\qed
\end{rem(A)} 
\lesubsection{Parabolic Equations} 
\extraindent 
We consider initial value problems 
\begin{equation}\label{I.C} 
(\pl_t+\cA)u=f\text{ on }M\times J 
\qb u(0)=u_0\text{ on }M. 
\end{equation} 
Here $\cA$~is a differential operator of the form~\eqref{I.A}, operating 
on sections of~$V$,  but with 
\hbox{$t$-dependent} coefficients. More precisely, $\cA$~is said to be 
\hbox{$\bs$\hh{-regular}}, where 
\hb{\bs\in\BR^+\ssm\BN}, if 
\begin{equation}\label{I.aj} 
a_j\in 
bc^{\bs/\vec r}\bigl(V_{\coV\tau+\sa}^{\sa+\tau+j}(\cL(F))\times J\bigr) 
\qa 0\leq j\leq r. 
\npb 
\end{equation} 
This assumption guarantees the continuity of 
\hb{\pl+\cA} on anisotropic spaces. 
\begin{pro}\label{pro-I.A} 
Let $\cA$ be \hbox{$\bs$-regular}. Then 
\begin{equation}\label{I.AW} 
\pl_t+\cA\in\cL 
\bigl(W_{\coW q}^{(s+r)/\vec r}(V\times J),W_{\coW q}^{s/\vec r}(V\times J) 
\bigr) 
\qa 0\leq s<\bs, 
\end{equation} 
and 
\begin{equation}\label{I.AB} 
\pl_t+\cA\in\cL 
\bigl(bc^{(s+r)/\vec r}(V\times J),bc^{s/\vec r}(V\times J) 
\bigr) 
\qa 0<s\leq\bs 
\qb s\notin\BN.  
\end{equation} 
\end{pro}
\begin{proof} 
This is a consequence of the (straightforward extension of the) 
point-wise multiplier Theorem~9.2 in~\cite{Ama12b}. 
\end{proof} 
\begin{rem(A)}\label{rem-I.A(a)} 
\extraroman 
The \hbox{$\bs$-regularity} assumption has been imposed for simplicity. 
It is optimal for~\eqref{I.AB}, but not for~\eqref{I.AW}. Also note that 
it follows from~\eqref{I.eBC} that condition~\eqref{I.aj} in~\eqref{I.AB} 
can be replaced by 
$$ 
a_j\in 
BC^{\bs/\vec r}\bigl(V_{\coV\tau+\sa}^{\sa+\tau+j}(\cL(F))\times J\bigr) 
\qa 0\leq j\leq r, 
$$ 
if 
\hb{s<\bs}.\nolinebreak\hfill\nolinebreak\qed 
\end{rem(A)} 
\begin{rem(A)}\label{rem-I.A(b)} 
\extraroman 
If $\cA$~is autonomous, that is, its coefficients are independent of 
\hb{t\in J}, then \eqref{I.aj} reduces to 
\hb{a_j\in bc^{\bs}\bigl(V_{\coV\tau+\sa}^{\sa+\tau+j}(\cL(F))\bigr)} for 
\hb{0\leq j\leq r}.\nolinebreak\hfill\nolinebreak\qed 
\end{rem(A)} 
\begin{rem(A)}\label{rem-I.A(c)} 
\extraroman 
Suppose 
\hb{\sa=\tau=0} and 
\hb{\Mg=\Rmgm}. Then, writing~$\cA$ in the form~\eqref{I.AD},  
\hbox{$\bs$-regularity} means 
\hb{a_\al\in bc^{\bs/\vec r}\bigl(\BR^m\times J,\cL(F)\bigr)} for 
\hb{|\al|\leq r}.\nolinebreak\hfill\nolinebreak\qed 
\end{rem(A)} 
Let $\cA$ be \hbox{$\bs$-regular}. We write 
\hb{a_j(t)(p):=a_j(p,t)} for 
\hb{(p,t)\in M\times J}. Then 
$$ 
a_j(t)\in bc^{\bs}\bigl(V_{\coV\tau+\sa}^{\sa+\tau+j}(\cL(E))\bigr) 
\qa t\in J. 
$$ 
Hence 
$$ 
\cA(t):=\sum_{j=0}^ra_j(t)\btdot\na^j 
$$ 
is well-defined for 
\hb{t\in J}. The operator (family)~$\cA$ is \hh{uniformly normally} 
\hbox{[$\ve$\hh{-}]}\hh{elliptic} on 
\hb{M\times J} if $\cA(t)$~has this property uniformly with respect to 
\hb{t\in J}. Then 
\hb{\pl_t+\cA} is \hh{uniformly normally} \hbox{[$\ve$\hh{-}]}\hh{parabolic}. 
\begin{rem(A)}\label{rem-I.PP} 
\extraroman 
Suppose 
\hb{\sa=\tau=0} and 
\hb{\Mg=\Rmgm}. Then 
\hb{\pl_t+\cA} is uniformly normally parabolic iff it is uniformly 
Petrowskii-parabolic (cf.~\cite{LSU68a} or \cite{EiZ98a}, 
for example).\nolinebreak\hfill\nolinebreak\qed 
\end{rem(A)} 
Now we can formulate the main result of this paper. We suppose 
$$ 
\bal 
\rm{(i)}\quad 
&\Mg\text{ is a uniformly regular Riemannian manifold}\cr 
\noalign{\vskip-1\jot} 
&\text{without boundary}.\cr
\rm{(ii)}\quad 
&J=J_T\text{ for some }T>0.\cr 
\rm{(iii)}\quad 
&\cA\text{ is $\bs$-regular and uniformly normally elliptic}\cr 
\noalign{\vskip-1\jot} 
&\text{on $M\times J$ of order }r. 
\eal 
\npb 
$$ 
By~$\ga$ we denote the trace operator 
\hb{u\mt u(0)}. 
\begin{thm(A)}\label{thm-I.P} 
Suppose 
\hb{k\in\BN} and either 
$$ 
\bal 
(\al)\quad 
&kr\leq s<kr+1\text{ and }\bs\geq s,\text{ or}\\ 
(\ba)\quad 
&kr+1<s<(k+1)r\text{ with $s\notin\BN$ and }\bs>(k+1)r. 
\eal 
$$ 
\begin{itemize} 
\item[\rm(i)] 
Assume 
\hb{\bs>s} and 
\hb{1\leq q<\iy} with 
\hb{q>1} if 
\hb{s=kr}. Then 
$$ 
{}\qquad 
(\pl_t+\cA,\ga)\in\Lis\bigl(W_{\coW q}^{(s+r)/\vec r}(V\times J), 
W_{\coW q}^{s/\vec r}(V\times J)\times 
W_{\coW q}^{s+r(1-1/q)}(V)\bigr). 
$$ 
\item[\rm(ii)] 
Let 
\hb{s\neq kr}. Then 
$$ 
(\pl_t+\cA,\ga)\in\Lis\bigl(bc^{(s+r)/\vec r}(V\times J), 
bc^{s/\vec r}(V\times J)\times 
bc^{s+r}(V)\bigr). 
$$ 
\end{itemize} 
\end{thm(A)} 
\begin{rem(A)}\label{rem-I.P(a)} 
\extraroman 
In case~(i) the Cauchy problem~\eqref{I.C} possesses for each 
$(f,u_0)$ in 
\hb{W_{\coW q}^{s/\vec r}(V\times J)\times W_{\coW q}^{s+r(1-1/q)}(V)} 
a~unique solution~$u$ belonging to 
\hb{W_{\coW q}^{(s+r)/\vec r}(V\times J)}, and 
$$ 
\|u\|_{W_{\coW q}^{(s+r)/\vec r}(V\times J)} 
\leq c\bigl(\|f\|_{W_{\coW q}^{s/\vec r}(V\times J)} 
+\|u_0\|_{W_{\coW q}^{s+r(1-1/q)}(V)}\bigr). 
$$ 
Similarly, in case~(ii) problem~\eqref{I.C} has for each 
$$ 
(f,u_0)\in bc^{(s+r)/\vec r}(V\times J)\times bc^{s+r}(V) 
$$ 
a~unique solution 
\hb{u\in bc^{(s+r)/\vec r}}, and 
$$ 
\|u\|_{bc^{(s+r)/\vec r}(V\times J)} 
\leq c\bigl(\|f\|_{bc^{(s+r)/\vec r}(V\times J)} 
+\|u_0\|_{bc^{s+r}(V)}\bigr). 
$$ 
The proofs below show that $c$~depends on~$\ve$, a~bound for the 
$bc^{\bs/\vec r}$~norms of the coefficients, and on~$T$ only, 
but not on the individual 
operators.\nolinebreak\hfill\nolinebreak\qed 
\end{rem(A)} 
\begin{rem(A)}\label{rem-I.P(aa)} 
\extraroman 
Suppose 
\hb{kr<s<kr+1}. Then we can choose 
\hb{\bs=s} in part~(ii) of the theorem. This regularity assumption 
is optimal. In contrast, condition 
\hb{\bs>(k+1)r} if 
\hb{kr+1<s<(k+1)r} is not the best possible one. It stems from the fact 
that we derive the statements in this case by interpolation 
(cf.~the proof in 
Section~\ref{sec-T}).\nolinebreak\hfill\nolinebreak\qed 
\end{rem(A)} 
\begin{rem(A)}\label{rem-I.P(b)} 
\extraroman 
Suppose 
\hb{\sa=\tau=0} and 
\hb{\Mg=\Rmgm}. If 
\hb{s=kr}, then assertion~(i) regains (except for the 
\hbox{$\bs$-regularity} assumption which we could relax in this situation 
also) classical results due to  V.A. Solonnikov (see \cite{Sol65a} and 
\cite[IV.\S5 and VII.\S9]{LSU68a}). Our proof is based on Fourier analytic 
techniques and entirely different from Solonnikov's 
approach.\nolinebreak\hfill\nolinebreak\qed 
\end{rem(A)} 
\begin{rem(A)}\label{rem-I.P(c)} 
\extraroman 
Assume 
\hb{\sa=\tau=0} and 
\hb{\Mg=\Rmgm}. In this case, assertion~(ii) is closely related to the 
H\"older space solvability theory of parabolic equations developed by 
V.A.~Solonnikov (see Theorem~VII.10.2 in~\cite{LSU68a}, 
where even more general parabolic systems are studied).  
 
\smallskip 
In the case of scalar parabolic second order equations, Solonnikov's 
H\"older space results have been partially recovered by 
A.~Lunardi~\cite[Theorem~5.1.10]{Lun95a} using semigroup techniques. 
Although we could establish a H\"older space theory as well, 
we prefer to work with little H\"older spaces since the latter 
enjoy the density 
properties~\eqref{I.ebc}.\nolinebreak\hfill\nolinebreak\qed 
\end{rem(A)} 
\begin{rem(A)}\label{rem-I.P(d)} 
\extraroman 
In~\cite{Gru95b} G.~Grubb presented an elaborate extension of the 
\hbox{$L_p$~theory}, 
\hb{1<p<\iy}, for parabolic (boundary value) problems to manifolds. In fact, 
she studied pseudodifferential boundary value problems for operators 
acting on sections of (general) vector bundles over so-called `admissible 
manifolds', 
introduced by her and N.J.~Kok\-holm~\cite{GrK93a}. These manifolds form a 
subclass of the family of manifolds with finitely many infinite conical ends 
(cf.~Example~\ref{exa-I.ex(e)}. Thus, for this class and 
\hb{1<q<\iy}, Theorem~\ref{thm-I.P}(i) is a very particular special case of 
Grubb's results (except for her very strong regularity assumptions). 
The proofs in~\cite{Gru95b} do, however, not extend to general uniformly 
regular Riemannian manifolds, since they use in an essential way specific 
`admissible' atlases consisting of finitely many charts only 
(cf.~\cite[Lemma~1.5]{GrK93a}).\nolinebreak\hfill\nolinebreak\qed 
\end{rem(A)} 
\begin{rem(A)}\label{rem-I.P(e)} 
\extraroman 
Let the assumptions of~(i) be satisfied. Then it follows from \eqref{I.WW} 
and~(i) that the homogeneous Cauchy problem 
\begin{equation}\label{I.C0} 
(\pl_t+\cA)u=f\text{ on }M\times J 
\qb u(0)=0 
\end{equation} 
has for each 
\hb{f\in W_{\coW q}^{s/\vec r}(V\times J)} a~unique solution~$u$ such that 
$u$, $\cA u$, and~$\pl_tu$ belong to 
\hb{W_{\coW q}^{s/\vec r}(V\times J)}. 

\smallskip 
Similarly, if 
\hb{s\neq kr}, then (ii)~guarantees that \eqref{I.C0} has for each 
$f$ in 
\hb{bc^{s/\vec r}(V\times J)} a~unique solution~$u$ satisfying 
\hb{u,\cA u,\pl_tu\in bc^{s/\vec r}(V\times J)}. This shows that 
\emph{Theorem~\ref{thm-I.P} provides maximal regularity 
results}.\nolinebreak\hfill\nolinebreak\qed 
\end{rem(A)} 
Let $E_0$ and~$E_1$ be Banach spaces with 
\hb{E_1\sdh E_0}. Then 
$\cH\EeEn$ denotes the set of all 
\hb{A\in\cL\EeEn} such that~$-A$, considered as a linear operator in~$E_0$ 
with domain~$E_1$, is the infinitesimal generator of a strongly continuous 
analytic semigroup 
\hb{\{\,e^{-tA}\ ;\ t\geq0\,\}} on~$E_0$, that is, in~$\cL(E_0)$. 

\smallskip 
Suppose \hb{s=0} (so that 
\hb{1<q<\iy}) and let~$\cA$ be autonomous. It follows from 
\hb{W_{\coW q}^{0/\vec r}(V\times J)\doteq L_q\bigl(J,L_q(V)\bigr)} 
that $\cA$~has maximal $L_q\bigl(J,L_q(V)\bigr)$ regularity 
(cf.~\cite{Ama95a} or J.~Pr\"uss and G.~Simonett~\cite{PrS16a} 
for explanations). Thus a result of G.~Dore~\cite{Dor93b} 
guarantees that $\cA$~belongs to 
$\cH\bigl(W_{\coW q}^r(V),L_q^r(V)\bigr)$. The following theorem shows 
that this is also true if 
\hb{s>0}.  
\goodbreak 
\begin{thm(A)}\label{thm-I.S} 
Let $\cA$ be autonomous. 
\begin{itemize} 
\item[\rm(i)] 
Assume 
either 
\hb{s\in r\BN} and 
\hb{1<q<\iy}, or 
\hb{s\notin\BN} and 
\hb{1\leq q<\iy}. Let 
\hb{\bs>s}. Then 
$$ 
\cA\in\cH\bigl(W_{\coW q}^{s+r}(V),W_{\coW q}^s(V)\bigr). 
$$ 
\item[\rm(ii)] 
If 
\hb{s\notin\BN} and 
\hb{\bs\geq s}, then 
$$ 
\cA\in\cH\bigl(bc^{s+r}(V),bc^s(V)\bigr). 
$$ 
\end{itemize} 
\end{thm(A)}
\goodbreak 
\begin{rem(A)}\label{rem-I.S(a)} 
\extraroman 
Suppose 
\hb{s=0} (so that 
\hb{q>1}). Then Theorems \ref{thm-I.P}(i) and \ref{thm-I.S}(i) 
imply~---~independently of the Dore result~---~that $\cA$~has maximal 
$L_q\bigl(J,L_q(V)\bigr)$ regularity. This is already known if 
\hb{\sa=\tau=0} and either $M$~is compact or 
\hb{\Mg=\Rmgm}. In fact, it has been shown by 
H.~Amann, M.~Hieber, and G.~Simonett~\cite{AHS94a} that then $\cA$~has 
a bounded \hbox{$H^\iy$-calculus}, thus, in particular, bounded imaginary 
powers. Now the assertion is a consequence of the Dore--Venni 
theorem~\cite{DoV87a} (see \cite[Theorem~III.4.10.7]{Ama95a} for an 
exposition). More recently, in the Euclidean space case, maximal 
$L_q\bigl(J,L_q\Rm\bigr)$ regularity has been proved for 
\hb{1<q<\iy}~---~even in infinite-dimensional settings~---~by  
R.~Denk, M.~Hieber, and J.~Pr\"uss~\cite{DHP03a} using rather sophisticated 
vector-valued harmonic analysis techniques, namely so-called 
\hbox{$\cR$-boundedness} methods (see~\cite{PrS16a} 
for a detailed exposition; furthermore, Theorem~6.4.3 there\-in contains a 
maximal regularity theorem in higher order Sobolev--Slobodeckii spaces on 
compact hypersurfaces of~$\BR^m$ without boundary). 
The approach of our paper is much simpler. If 
\hb{s\notin\BN}, then it can be extended to infinite-dimensional settings 
also. We refrain from doing this here but refer to~\cite{Ama17a}. 

\smallskip 
Assume $\cA$~is an autonomous second order positive semidefinite differential 
operator with bounded smooth coefficients. Then, by establishing heat kernel 
bounds and using a result of M.~Hieber and J.~Pr\"uss~\cite{HP97a}, 
A.L. Mazzucato and V.~Nistor~\cite{MaN06a} prove the maximal 
$L_p\bigl(J,L_q(V)\bigr)$-regularity of~$\cA$ for 
\hb{1<p,q<\iy}. 

\smallskip 
If 
\hb{s>0} and 
\hb{1<q<\iy}, then R.~Denk and T.~Seger~\cite{DeS16a} showed 
that a scalar elliptic operator with constant coefficients generates an 
analytic semigroup on~$W_{\coW q}^s\Rm$. However, these authors do not 
establish a maximal regularity 
result.\nolinebreak\hfill\nolinebreak\qed  
\end{rem(A)} 
\begin{rem(A)}\label{rem-I.S(b)} 
\extraroman 
Suppose $\cA$~is autonomous and 
\hb{0<s\leq\bs} with 
\hb{s\notin\BN}. Then we can combine Theorem~\ref{thm-I.S}(ii) with the 
continuous maximal regularity theory of G.~Da~Prato and 
P.~Grisvard~\cite{DaG75a} (see \cite[Theorem~III.3.4.1]{Ama95a}). 
For this we set 
$$ 
\cW_\iy^{(s+r,1)}(V\times J) 
:=C\bigl(J,bc^{s+r}(V)\bigr)\cap C^1\bigl(J,bc^s(V)\bigr). 
$$ 
Then it follows 
\begin{equation}\label{I.DG} 
(\pl+\cA,\ga_0) 
\in\Lis\bigl(\cW_\iy^{(s+r,1)}(V\times J), 
C(J,bc^s(V))\times bc^{s+r}(V)\bigr). 
\end{equation} 
Note that, see~\eqref{I.bca}, 
$$ 
bc^{s/\vec r}(V\times J)\hr C\bigl(J,bc^s(V)\bigr) 
$$ 
and 
$$ 
bc^{(s+r)/\vec r}(V\times J) 
=C\bigl(J,bc^{s+r}(V)\bigr)\cap bc^{s/r+1}\bigl(J,BC(V)\bigr). 
$$ 
Thus the maximal regularity result obtained from Theorem~\ref{thm-I.P} is not 
comparable to~\eqref{I.DG}. 

\smallskip 
It is the advantage of the anisotropic spaces 
\hb{bc^{s/\vec r}(V\times J)} over the spaces 
\hb{\cW_\iy^{(s+r,1)}(V\times J)} that the former enjoy all 
embedding, interpolation, and trace properties known from the Euclidean case 
(see \cite{Ama12c} and~\cite{Ama17a}). This is of importance in the study of 
quasilinear problems. Corresponding results for the \hbox{$\cW$-spaces} 
are, to say the least, not obvious. 

\smallskip 
In a recent paper, Y.~Shao and G.~Simonett~\cite{ShS14a} established  
the fact that 
\hb{\cA\in\cH\bigl(bc^{s+2}(V),bc^s(V)\bigr)} (in the case 
\hb{n=0} and 
\hb{0<s<1}), starting with the generation theorem given in the Euclidean case 
in \cite[Theorem~4.2 and Remark~4.6]{Ama01a}. Then, using the 
Da~Prato--Grisvard approach~---~in the extended version of 
S.~Angenent~\cite{Ang90a} which allows for blow-up at 
\hb{t=0} (cf.~\cite[Theorem~III.3.4.1]{Ama95a})~---~and a regularizing 
technique of S.~Angenent in the modified form of 
J.~Escher, J.~Pr\"uss, and G.~Simonett~\cite{EPS03a}, the authors 
establish the local well-posedness and time-analyticity of the Yamabe 
flow in little H\"older spaces on uniformly regular manifolds. 

\smallskip 
For further interesting applications of the little H\"older and 
Sobolev space theory on uniformly regular Riemannian manifolds we refer to 
Y.~Shao \cite{Sh13a}, \cite{Sh14a}, and~\cite{Sh15a}, and 
J.~LeCrone and 
G.~Simonett~\cite{LeCS16a}.\nolinebreak\hfill\nolinebreak\qed 
\end{rem(A)} 
\begin{rem(A)}\label{rem-I.S(c)} 
\extraroman 
Suppose 
\hb{\sa=\tau=n=0}. Then the (generalized) heat operator~\eqref{I.Aa} is the 
negative infinitesimal generator of the `heat semigroup' 
\hb{\{\,e^{-t\cA}\ ;\ t\geq0\,\}} on~$\Mg$. More precisely, 
$$ 
\cA\in\cH\bigl(W_{\coW q}^{s+2}(M),W_{\coW q}^s(M)\bigr) 
\text{ if }0\leq s<\bs, 
$$ 
with 
\hb{q>1} if 
\hb{s\in\BN}, and 
$$ 
\cA\in\cH\bigl(bc^{s+2}(M),bc^s(M)\bigr) 
\qa 0<s\leq\bs 
\qb s\notin\BN. 
$$ 
In addition, $\cA$~has maximal regularity in the sense of 
Remark~\ref{rem-I.P(e)}. The same is true, if 
\hb{n=0} and 
$\sa$ and~$\tau$ are arbitrary, for the covariant Laplacian~$\na^*\na$ 
and for the Hodge Laplacian (with $V$ replaced by 
\hb{\bigwedge^kT^*M}). 

\smallskip 
There is an enormous amount of literature concerning heat semigroups on 
Riemannian manifolds without boundary and bounded geometry. Most of it is 
an \hbox{$L_2$-theory} and deals with kernel estimates and spectral theory 
(see, for example, E.B. Davies~\cite{Dav89a} or A.~Grigor'yan~\cite{Gri09a}). 
Those works rely heavily on curvature bounds which is no issue at all in our 
approach.\nolinebreak\hfill\nolinebreak\qed 
\end{rem(A)} 
Similarly as for compact manifolds, the cornerstones of the proofs of the 
above theorems are the corresponding assertions for Euclidean model cases 
and localizations by means of suitable atlases. In the noncompact setting we 
cannot use finite atlases but have to deal with infinitely many charts. 
This requires uniform local estimates and a somewhat elaborate technical 
machinery. Both of these are developed in the following sections. 

\smallskip 
To allow for a unified approach by Fourier analysis to parabolic and 
elliptic equations we introduce, in the next section, 
general weighted spaces on~$\BR^d$ and closed half-spaces thereof. In 
Section~\ref{sec-B} we collect those of their basic properties which we 
employ in this paper. 

\smallskip 
The study of anisotropic function spaces and the Fourier analysis therein 
are considerably facilitated by the use of anisotropic dilations. The 
latter are introduced in Section~\ref{sec-A} and some easy properties are 
described. 

\smallskip 
The next section belongs to the heart of the matter. Here we introduce the 
Fourier multiplier theorems from which we derive our results. In the case of 
anisotropic Sobolev spaces we rely on the Marcinkiewicz theorem. Anisotropic 
Slobodeckii and H\"older spaces are particular realizations of Besov spaces. 
To handle these cases, we introduce an anisotropic extension of the Fourier 
multiplier theorem established in~\cite{Ama97b}. Although this extension 
holds for operator-valued symbols and arbitrary Banach spaces, we restrict 
ourselves to the case of matrix-valued symbols. By combining the Fourier 
multiplier theorem with a lifting property we arrive at simple criteria for 
Fourier integral operators with (anisotropically) homogeneous symbols to 
realize bounded linear operators between Sobolev--Slobodeckii, 
respectively H\"older spaces. 

\smallskip 
As a first application of these Fourier multiplier theorems we give, in 
Section~\ref{sec-H}, a~very simple proof for the fact that principal part 
parabolic operators with constant coefficients define isomorphisms between 
suitable Sobolev--Slobodeckii and little H\"older spaces on 
\hb{\BR^m\times\BR}. It is the advantage of our approach that it handles 
all these spaces by one and the same technique. In particular, in this 
Fourier-analytic approach we can 
deal with all Slobodeckii spaces, including those with integrability 
index~$1$, as well as with H\"older spaces. 
This stands in 
contrast to the earlier work of other authors. In the Euclidean setting, 
Solonnikov derived his H\"older space results by carefully estimating heat 
kernels (also see~\cite{Fri64a}). However, recently in~\cite{Sol15a} he 
has used an anisotropic extension, 
due to O.A.~Ladyzhenskaya~\cite{Lad03a}, of a Fourier multiplier theorem 
for isotropic H\"older seminorms, given by L.~H{\"o}rmander in 
\cite[Theorem~7.9.6]{Hoe83aI}, to establish the H\"older continuity of 
solutions to a number of model problems (also see~\cite{Deg15a}). 

\smallskip 
The solvability results of G.~Grubb~\cite{Gru95b} in the Slobodeckii space 
setting are obtained by first establishing the corresponding results for 
Bessel potential spaces and then using interpolation. Since the Bessel 
potential space results are restricted to \hbox{$L_q$~settings} with 
\hb{1<q<\iy}, there is no way to cover the spaces~$W_1^{s/\vec r}$ or 
H\"older spaces by this method. In addition, interpolation does not lead to 
optimal regularity conditions for the coefficients. 

\smallskip 
Using an isotropic setting, we give, along the same lines, in 
Section~\ref{sec-S} a~simple proof for Theorem~\ref{thm-I.S}, 
provided $\cA$~is a 
principal part operator on~$\BR^m$ with constant coefficients. (This 
result is already contained in~\cite{Ama97b}.) By combining the findings in 
Sections \ref{sec-H} and~\ref{sec-S} we prove in the next section 
Theorem~\ref{thm-I.P} for the constant coefficient model problem on 
\hb{\BR^m\times\BR^+}. 

\smallskip 
In Section~\ref{sec-L} we present the precise definition of uniformly regular 
Riemannian manifolds and and prove the basic localization 
Theorem~\ref{thm-L.n}. The next two sections contain the localization 
machinery by which we can reduce the proof of Theorem \ref{thm-I.P} 
and~\ref{thm-I.S} to the flat case 
\hb{\Mg=\Rmgm}. This is done by constructing~a retraction-coretraction pair 
between our function spaces on~$M$ and sequence spaces whose elements take 
values in the corresponding function spaces on~$\BR^m$. Here we rely on 
our previous work on function spaces on singular manifolds 
\cite{Ama12c},~\cite{Ama12b}. 

\smallskip 
In the Euclidean setting in Section~\ref{sec-N}, we use for the 
first (and only) time the fact that in the preparatory sections 
\hbox{\ref{sec-F}--\ref{sec-C}} we have dealt with parameter-dependent spaces 
and operators. This is employed to control the lower order terms 
which, by choosing the parameter sufficiently large, can be considered 
to be small perturbations of the principal part operators. Thus our use of 
parameter-dependent spaces is somewhat different from the usual one initiated 
by M.S.~Agranovich and M.I.~Vishik~\cite{AgV64a} and greatly amplified by 
G.~Grubb (see \cite{Gru95b}, \cite{Gru96a}  and the references therein). 

\smallskip 
Finally, in the last section we prove Theorems \ref{thm-I.P} 
and~\ref{thm-I.S} on the basis of the material prepared in the preceding 
parts. 

\smallskip 
It should be mentioned that the global strategy applied in this work is 
more or less well-known, except for the Fourier-analytic treatment of the 
H\"older space case. Nevertheless, our approach 
differs in details~---~even in the Euclidean setting~---~considerably from 
those of other authors. 
\section{Function Spaces in Euclidean Settings\label{sec-F}} 
\extraindent 
We suppose 
$$ 
\bt\quad 
d\in\thBN\text{ and }\BX\in\{\BR^d,\BH\},\text{ where }
\BH=\BH^d:=\BR^{d-1}\times\BR^+ 
$$ 
and endow~$\BX$ with the Euclidean metric~$g_d$. A~\emph{weight system} 
for~$\BX$ is a triple $[\ell,\mfd,\mfnu]$ such that 
$$ 
\left.  
\bal 
{}
&\ell\in\thBN,\ \mfd=(d_1,\ldots,d_\ell), 
 \ \mfnu=(\nu_1,\ldots,\nu_\ell)\in(\thBN)^\ell 
 \text{ with}\cr 
&d_1+\cdots+d_\ell=d,\text{ and }d_\ell=1\text{ if }\BX=\BH. 
\eal 
\quad\right\} 
$$ 
We set 
\hb{\BX_i:=\BR^{d_i}} for 
\hb{1\leq i\leq\ell} with 
\hb{\BX_\ell:=\BR^+} if 
\hb{\BX=\BH}. Then 
\hb{\BX_1\times\cdots\times\BX_\ell} is the \hbox{$\mfd$\emph{-clustering}} 
of~$\BX$. We write 
$$ 
x=(x^1,\ldots,x^d)=(x_1,\ldots,x_\ell) 
\qb x_i=(x_i^1,\ldots,x_i^{d_i}) 
\qb 1\leq i\leq\ell\ , 
\npb 
$$ 
according to the interpretation of~$x$ as an element of~$\BX$ or of 
\hb{\BX_1\times\cdots\times\BX_d}. 

\smallskip 
We call $[\ell,\mfd,\mfnu]$ \emph{reduced weight system} if 
\hb{\ell<d}, and \emph{non-reduced} otherwise. If 
\hb{\ell=d}, then 
\hb{\mfd=\mf{1}=(1,\ldots,1)}. The 
weight system is \hbox{$\nu$\emph{-homogeneous}} if 
\hb{\ell=1}. Then 
\hb{\mfd=(d)} and 
\hb{\mfnu=(\nu)}. In this case we write $[1,d,\nu]$ for 
$[1,\mfd,\mfnu]$. In general, 
$$ 
\nu:=\LCM(\mfnu)=\LCM(\nu_1,\ldots,\nu_\ell)\ , 
\npb 
$$ 
the least common multiple of 
\hb{\nu_1,\ldots,\nu_\ell}. 

\smallskip 
With $[\ell,\mfd,\mfnu]$ we associate its 
\emph{non-reduced version} $[d,\mf{1},\mfom]$, where%
$$ 
\mfom=\mfom(\mfnu)=(\om_1,\ldots,\om_d) 
:=(\nu_1,\ldots,\nu_1,\nu_2,\ldots,\nu_2,\ldots,\nu_\ell,\ldots\nu_\ell) 
$$  
with $d_i$~copies of~$\nu_i$. Thus the non-reduced version 
of $[1,d,\nu]$ equals  $[d,\mf{1},\nu\mf{1}]$. Note 
\hb{\LCM(\mfom)=\nu} and 
$$ 
|\mfom|:=\om_1+\cdots+\om_d 
=\mfd\btdot\mfnu:=d_1\nu_1+\cdots+d_\ell\nu_\ell. 
$$ 
\begin{rem(A)}\label{rem-F.W} 
\extraroman 
In this paper only two weight systems will be of importance, namely 
$$ 
\bal 
{\rm(i)} 
&\quad \text{\emph{trivial}, 
 that is, \emph{$1$-homogeneous weight systems} }[1,1,1] 
 \text{ with }d=m,\cr 
{\rm(ii)} 
&\quad r\text{\emph{-parabolic weight systems} }\bigl[2,(m,1),(1,r)\bigr] 
 \text{ with }d=m+1. 
\eal 
$$ 
Nevertheless, for the sake of a unified presentation it is convenient to 
consider the general case.\nolinebreak\hfill\nolinebreak\qed
\end{rem(A)} 
For the following 
$$ 
\bal 
\bt\quad 
&\text{we fix a weight system }[\ell,\mfd,\mfnu]\text{ for }\BX.\cr 
\bt\quad 
&E\text{ is a Banach space}. 
\eal 
$$ 
Given 
\hb{k\in\nu\BN}, we introduce the parameter-dependent norms 
$$ 
\|u\|_{k/\mfnu,q;\eta} 
:=\sum_{\al\ibtdot\mfom\leq k}\eta^{k-\al\ibtdot\mfom}\,\|\pa u\|_q 
\qa 1\leq q\leq\iy, 
$$ 
for 
\hb{\eta>0}. Then parameter-dependent anisotropic \emph{Sobolev spaces} 
of order~%
\hb{k/\mfnu} over~$L_q$, 
$$ 
W_{\coW q;\eta}^{k/\mfnu} 
=\bigl(W_{\coW q}^{k/\mfnu}\BXE,\Vsdot_{k/\mfnu,q;\eta}\bigr), 
$$ 
are defined for 
\hb{1\leq q<\iy} to be the completion of~$\cS\BXE$ in~$L_q\BXE$ with 
respect to the norm~%
\hb{\Vsdot_{k/\mfnu,q;\eta}}. As usual, $\cS\BXE$~is the Fr\'echet space of 
smooth rapidly decreasing \hbox{$E$-valued} functions on~$\BX$. Then 
\hb{W_{\coW q;\eta}^{0/\mfnu}\doteq L_q} and 
\hb{W_{\coW q;\eta}^{k/\mfnu}\doteq W_{\coW q}^{k/\mfnu} 
   :=W_{\coW q;1}^{k/\mfnu}}, where 
\hb{{}\doteq{}}~means: 
equal except for equivalent norms. 

\smallskip 
We introduce 
$$ 
BC_\eta^{k/\mfnu} 
=\bigl(BC^{k/\mfnu}\BXE,\Vsdot_{k/\mfnu,\iy;\eta}\bigr), 
$$ 
the Banach space of all 
\hb{u\in BC\BXE} with 
\hb{\pa u\in BC\BXE} for 
\hb{\al\btdot\mfom\leq k}, where 
\hb{BC^{k/\mfnu}:=BC_1^{k/\mfnu}}. Then 
$$ 
\BUC_\eta^{k/\mfnu} 
=\bigl(\BUC^{k/\mfnu}\BXE,\Vsdot_{k/\mfnu,\iy;\eta}\bigr) 
$$ 
is the closed linear subspace consisting of 
all~$u$ for which $\pa u$~is uniformly continuous on~$\BX$. 

\smallskip 
We write 
\hb{\BX_\ihati:=\BX_1\times\cdots\times\wh{\BX_i}\times\cdots\times\BX_\ell} 
for 
\hb{1\leq i\leq\ell}, where the hat is the omission symbol, and 
$(x_i;x_\ihati)$ stands for 
\hb{x\in\BX} with 
\hb{x_\ihati\in\BX_\ihati}. Recalling \eqref{I.tq} and~\eqref{I.td}, we set 
$$ 
\eea{u}_{\ta,q;i}:= 
\left\{ 
\bal 
{}
&\big\|x_\ihati 
 \mt\bigl[u(\cdot;x_\ihati)\bigr]_{\ta,q}\big\|_{L_q(\iBXhi)},
 &&\quad \text{if }q<\iy,\cr 
&\sup_{x_\iihati\in\iBXhi}
 \bigl[u(\cdot;x_\ihati)\bigr]_{\ta,\iy},
 &&\quad \text{if }q=\iy. 
\eal 
\right. 
$$ 
Suppose 
\hb{k_i\in\nu_i\BN} and 
\hb{k_i<s<k_i+\nu_i} for 
\hb{1\leq i\leq\ell}. Put 
$$ 
\|u\|_{s/\mfnu,q;\eta} 
:=\sum_{i=1}^\ell\Bigl(\sum_{j=0}^{k_i/\nu_i} 
\eta^{s-j\nu_i}\ \|\na_{\cona x_i}^ju\|_q 
+\eea{\na_{\cona x_i}^{k_i/\nu_i}u}_{(s-k_i)/\nu_i,q;i}\Bigr). 
$$ 
Let 
\hb{1\leq q<\iy}. Then the parameter-dependent anisotropic \emph{Slobodeckii 
space} of order~%
\hb{s/\mfnu} over~$L_q$, 
$$ 
W_{\coW q;\eta}^{s/\mfnu} 
=\bigl(W_{\coW q}^{s/\mfnu}\BXE,\Vsdot_{s/\mfnu,q;\eta}\bigr), 
$$ 
is the completion of~$\cS\BXE$ in~$L_q$ with respect to the norm~%
\hb{\Vsdot_{s/\mfnu,q;\eta}}. The parameter-dependent anisotropic 
\emph{H\"older space} of order~%
\hb{s/\mfnu} is the Banach space 
$$ 
\BUC_\eta^{s/\mfnu} 
=\bigl(\BUC^{s/\mfnu}\BXE,\Vsdot_{s/\mfnu,q;\eta}\bigr) 
$$ 
consisting of all 
\hb{u\in\BUC\BXE} such that 
$$ 
\bigl(x_i\mt u(x_i;\cdot)\bigr) 
\in\BUC^{s/\nu_i}\bigl(\BX_i,\BUC\BXhiE\bigr) 
$$ 
for 
\hb{1\leq i\leq\ell}. (In the Euclidean setting we use the conventional 
notation $\BUC^t$ for~$BC^t$ if 
\hb{t\in\BR^+\ssm\BN}.) Lastly, the \emph{little H\"older 
space}~$buc_\eta^{s/\mfnu}$ is the closed linear subspace 
of~$\BUC_\eta^{s/\mfnu}$ formed by all~$u$ satisfying 
$$ 
\lim_{\da\ra0}\sup_{x_\iihati\in\BX_\iihati} 
\bigl[\na_{x_i}^{k_i/\nu_i}u(\cdot;x_\ihati)\bigr] 
_{(s-k_i)/\nu_i,\iy}^\da=0 
\qa 1\leq i\leq\ell. 
$$ 
\section{Basic Properties\label{sec-B}} 
\extraindent 
In this section we collect the fundamental facts about the spaces introduced 
above which are needed in what follows. We do not give proofs but refer 
to~\cite{Ama17a} for a detailed exposition, even in vector-valued settings. 
(Also see \cite{Ama09a} for a preliminary account which, however, does not 
include H\"older spaces). 

\smallskip 
Henceforth, we denote by 
\hb{c,c_0,c_1,\ldots} constants~%
\hb{\geq1} which may depend in an increasing way on nonnegative parameters 
\hb{\al,\ba,\ldots}, whereupon we write 
\hb{c(\al,\ba,\ldots)} etc. These constants may vary from occurrence to 
occurrence but are always independent of the free variables in a 
given setting. 

\smallskip 
Let $f$ and~$g$ be nonnegative functions on some set~$S$. Then 
\hb{f\sim g} means 
\begin{equation}\label{B.eq} 
g/c\leq f\leq cg. 
\end{equation} 
Suppose 
\hb{f_\eta,g_\eta\sco S\ra\BR^+} for 
\hb{\eta>0}. Then we write 
\hb{\smash{f_\eta\simeta g_\eta}\vph{\big|}} if 
\hb{f_\eta\sim g_\eta} holds 
\ \hbox{$\eta$-uniformly}, that is, the constant~$c$ in~\eqref{B.eq} is 
independent of 
\hb{\eta>0}. Let $X_\eta^{(i)}$ be normed vector spaces with norm~%
\hb{\Vsdot_\eta^{(i)}} for 
\hb{\eta>0}. Then  
\hb{\smash{X_\eta^{(1)}\doteqeta X_\eta^{(2)}}} iff 
\hb{\smash{\Vsdot_\eta^{(1)}\simeta\Vsdot_\eta^{(2)}}\vph{\Big|}}. Suppose 
\hb{a_\eta\in\cL(X_\eta^{(1)},X_\eta^{(2)})} for 
\hb{\eta>0}. Then we say: $a_\eta$~belongs to 
$\cL(X_\eta^{(1)},X_\eta^{(2)})$ \hbox{$\eta$-uniformly}, if the norm 
of~$a_\eta$ can be bounded independently of 
\hb{\eta>0}. If, in addition, 
\hb{a_\eta^{-1}\in\cL(X_\eta^{(2)},X_\eta^{(1)})} 
\hbox{$\eta$-uniformly}, then 
\hb{a_\eta\in\Lis(X_\eta^{(1)},X_\eta^{(2)})} 
\hbox{$\eta$-uniformly}. 

\smallskip 
To avoid lengthy repetitions, we call 
$(s,q)$ \hbox{\emph{$\nu$-admissible}}, if either 
\hb{s\in\nu\BN} and 
\hb{1<q<\iy}, or 
\hb{s\notin\BN} and \hb{1\leq q\leq\iy}. 

\smallskip 
Throughout this section 
$$ 
\bal 
{}
{\rm(i)}\quad 
    &(s,q)\text{ is $\nu$-admissible}.\cr 
{\rm(ii)}\quad 
    &E,E_0,E_1,\ldots\text{ are finite-dimensional}\cr
\noalign{\vskip-1\jot} 
    &\text{complex Banach spaces}. 
\eal 
$$ 
Then we set 
$$ 
\gF_{q;\eta}^{s/\mfnu}=\gF_{q;\eta}^{s/\mfnu}\BXE:= 
\left\{ 
\bal 
{}
&W_{\coW q;\eta}^{s/\mfnu},
 &&\quad \text{if }q<\iy,\cr 
&buc_{\eta}^{s/\mfnu},
 &&\quad \text{if }q=\iy, 
\eal 
\right. 
$$ 
for 
\hb{\eta>0}. We omit~$\eta$ if it equals~$1$ and write~$s$ for 
\hb{s/\mfnu} if the weight system is trivial. Thus $\gF_q^s$~is a standard 
isotropic Sobolev--Slobodeckii space if 
\hb{q<\iy} and an isotropic little H\"older space if 
\hb{q=\iy}. Observe that 
\hb{\gF_{q;\eta}^{s/\mfnu}\doteq\gF_q^{s/\mfnu}} (but not 
\hbox{$\eta$-uniformly}!). 
\goodbreak 
\begin{thm(A)}\label{thm-B.F} 
\begin{itemize} 
\item[] 
\item[\rm(i)] 
Assume 
\hb{0\leq s_0<s_1} and $(s_i,q)$ are \hbox{$\nu$-admissible}. Then 
\hb{\gF_q^{s_1/\mfnu}\sdh\gF_q^{s_0/\mfnu}} and 
$$ 
\Vsdot_{s_0/\mfnu,q;\eta} 
\leq c\kern1pt\eta^{s_0-s_1}\,\Vsdot_{s_1/\mfnu,q;\eta} 
\qa \eta>0. 
$$ 
\item[\rm(ii)] 
If 
\hb{\al\in\BN^d}, then 
\hb{\pa 
   \in\cL(\gF_{q;\eta}^{(s+\al\ibtdot\mfom)/\mfnu},\gF_{q;\eta}^{s/\mfnu})} 
$\eta$-uniformly. 
\end{itemize}
\end{thm(A)} 
The spaces~$\gF_q^{s/\mfnu}$ enjoy an important intersection space 
characterization. For this 
\hb{\BX=\BX\dotp\times\BX_\ell} with 
\hb{\BX\dotp=\BX_1\times\cdots\times\BX_{\ell-1}=\BR^{d-d_\ell}} and 
\hb{\mfnu=(\mfnu\dotp,\nu_\ell)}.  
\begin{thm(A)}\label{thm-B.I} 
If\/ 
\hb{q<\iy}, then 
$$ 
\bal 
\gF_{q;\eta}^{s/\mfnu} 
&=W_{\coW q;\eta}^{s/\mfnu}(\BX\dotp\times\BX_\ell,E)\cr 
&\doteqeta L_q\bigl(\BX_\ell,W_{\coW q;\eta}^{s/\mfnu\idotp} 
 (\BX\dotp,E)\bigr) 
 \cap W_{\coW q;\eta}^{s/\nu_\ell}\bigl(\BX_\ell,L_q(\BX\dotp,E)\bigr). 
\eal 
$$ 
Suppose 
\hb{q=\iy}. Then 
$$ 
\bal 
\gF_{\iy;\eta}^{s/\mfnu} 
&=buc_\eta^{s/\mfnu}(\BX\dotp\times\BX_\ell,E)\cr 
&\doteqeta\BUC\bigl(\BX_\ell,buc_\eta^{s/\mfnu\idotp}(\BX\dotp,E)\bigr) 
 \cap buc_\eta^{s/\nu_\ell}\bigl(\BX_\ell,\BUC(\BX\dotp,E)\bigr). 
 \eal 
$$ 
\end{thm(A)} 
The next theorem concerns point-wise multiplications. For Banach spaces 
$X_0$, $X_1$, and~$X_2$ we denote by $\cL(X_0,X_1;X_2)$ the Banach space
of all continuous bilinear maps 
\hb{\ba\sco X_0\times X_1\ra X_2}. If 
\hb{\ba\in\cL(E_0,E_1;E_2)}, then we write~$\sfm_\ba$ for its point-wise 
extension. 
\begin{thm(A)}\label{thm-B.M} 
Suppose 
\hb{s\leq s_0} with 
\hb{s_0\notin\BN} and 
\hb{s<s_0} if 
\hb{q<\iy}. Let $\ba$ belong to $\cL(E_0,E_1;E_2)$. Then 
$$ 
\sfm_\ba\in\cL\bigl(buc^{s_0/\mfnu}\BXEn,\gF_{q;\eta}^{s/\mfnu}\BXEe; 
\gF_{q;\eta}^{s/\mfnu}\BXEz\bigr) 
\quad\text{$\eta$-uniformly}. 
$$ 
If 
\hb{s\in\nu\BN}, then 
$$ 
\sfm_\ba\in\cL\bigl(\BUC^{s/\mfnu}\BXEn,\gF_{q;\eta}^{s/\mfnu}\BXEe; 
\gF_{q;\eta}^{s/\mfnu}\BXEz\bigr) 
\quad\text{$\eta$-uniformly}. 
\npb  
$$ 
In either case, the map 
\hb{\ba\mt\sfm_\ba} is linear and continuous. 
\end{thm(A)} 
For the the next theorem we recall that 
\hb{\XnXe_{\ta,q}=\XnXe_{\ta,q}^0}, if $X_0$ and~$X_1$ are Banach spaces 
with 
\hb{\smash{X_1\sdh X_0}\vph{\ci X}} and 
\hb{q<\iy}. 
\begin{thm(A)}\label{thm-B.In} 
Let $(s_0,q)$, $(s_1,q)$, and $(s_\ta,q)$ be \hbox{$\nu$-admissible} with 
\hb{s_0<s_1} and 
\hb{s_\ta:=(1-\ta)s_0+\ta s_1}. Then 
\hb{(\gF_{q;\eta}^{s_0/\mfnu},\gF_{q;\eta}^{s_1/\mfnu})_{\ta,q}^0  
   \doteqeta\gF_{q;\eta}^{s_\ta/\mfnu}}. 
\end{thm(A)} 
A~\emph{retraction} from~$X_0$ onto~$X_1$ is a continuous linear map 
\hb{r\sco X_0\ra X_1} possessing a continuous right inverse~$r^c$, 
a~coretraction. Any such pair~$(r,r^c)$ is said to be an \hbox{r-e} 
\emph{pair} for~$\XnXe$ (e~stands for `extension'). 

\smallskip 
We identify 
\hb{\pl\BH=\BX\dotp\times\{0\}} naturally with 
\hb{\BX\dotp=\BR^{d-1}} if convenient. Then the 
\emph{trace operator of order}~$k$ is the map 
\hb{\ga^k:=\bigl(u\mt\pl_\ell^ku(0)\bigr)} for 
\hb{k\in\BN}, defined for sufficiently smooth functions 
\hb{u\sco\BH\ra E}. Thus 
\hb{\ga=\ga^0}. 
\begin{thm(A)}\label{thm-B.T} 
Suppose 
\hb{s>\nu_\ell(k+1/q)} and 
\hb{s\notin\BN+\nu_\ell/q}. Then the trace map 
\hb{\vec\ga^k:=(\ga^0,\ga^1,\ldots,\ga^k)} is an \hbox{$\eta$-uniform} 
retraction 
$$ 
\text{ from }\gF_{q;\eta}^{s/\vec\nu}\BHE\text{ onto } 
\prod_{j=0}^k\gF_{q;\eta} 
   ^{(s-\nu_\ell(j+1/q))/\mfnu\idotp}(\BX\dotp,E). 
\npb 
$$  
It possesses an \hbox{$\eta$-uniform} coretraction. 
\end{thm(A)} 
It follows from Theorem~\ref{thm-B.I} that 
\hb{\gF_q^{s/\mfnu}\RhdE\hr L_q\bigl(\BR,\gF_q^{s/\nu\idotp}\RhdmeE\bigr)}. 
Hence 
$$ 
\vph{\gF}_0\gF_q^{s/\mfnu}  
:=\bigl\{\,u\in\gF_q^{s/\mfnu}\RhdE 
\ ;\ u(t)=0\text{ a.a. }t<0\,\bigr\} 
\npb 
$$ 
is a well-defined linear subspace of 
$L_q\bigl(\BR,\gF_q^{s/\mfnu\idotp}\RhdmeE\bigr)$. 

\smallskip 
Suppose 
\hb{k\in\BN} and 
\begin{equation}\label{B.ns} 
\nu_\ell(k+1/q)<s<\nu_\ell(k+1+1/q) 
\qb s\notin(\BN+1/q)\cup(\BN+\nu_\ell/q). 
\end{equation} 
It is a consequence of this trace theorem that 
$$ 
\gF_q^{s/\mfnu}\cBHE 
:=\bigl\{\,u\in\gF_q^{s/\mfnu}\BHE\ ;\ \vec\ga^ku=0\,\bigr\} 
$$ 
is a closed linear subspace of~$\gF_q^{s/\mfnu}\BHE$. The next theorem shows 
that we can extend the elements of $\gF_q^{s/\mfnu}\BHE$ 
and~$\gF_q^{s/\mfnu}\cBHE$ over~$\BR^d$ preserving their regularity. 
We denote by~$\gR$ the operator of point-wise 
restriction from~$\BR^d$ onto~$\BH$, and $\ci\gE$~is the operator of 
extension by zero from~$\BH$ over~$\BR^d$. 
\goodbreak 
\begin{thm(A)}\label{thm-B.E} 
\begin{itemize} 
\item[] 
\item[\rm(i)] 
${\gR\in\cL\bigl(\gF_{q;\eta}^{s/\mfnu}\RhdE,
   \gF_{q;\eta}^{s/\mfnu}\BHE\bigr)}$ \hbox{$\eta$-uniformly} 
and there exists a universal \hbox{$\eta$-uniform} coretraction~$\gE$ for it. 
Moreover, $\gR$~commutes with~$\pa$ for 
\hb{\al\btdot\mfom\leq s}. 
\item[\rm(ii)] 
Assume \eqref{B.ns} applies. There exists~$\ci\gR$ such that 
$\cgRcgE$ is an \hbox{$\eta$-uniform} 
\hbox{r-e} pair for 
\hb{\bigl(\gF_{q;\eta}^{s/\mfnu}\RhdE,\gF_{q;\eta}^{s/\mfnu}\cBHE\bigr)} and 
such that the restriction of~$\ci\gR$ to~$\im(\ci\gE)$ equals
\hb{\gR\sn\im(\ci\gE)}. 
\item[\rm(iii)] 
Suppose 
\hb{0\leq s<\nu_\ell/q}. Then $\gRcgE$ is an \hbox{$\eta$-uniform} 
\hbox{r-e} pair for 
\newline 
\hb{\bigl(\gF_{q;\eta}^{s/\mfnu}\RhdE,\gF_{q;\eta}^{s/\mfnu}\BHE\bigr)}. 
\end{itemize} 
\end{thm(A)} 
\goodbreak 
\begin{cor}\label{cor-B.E} 
\begin{itemize} 
\item[] 
\item[\rm(i)] 
Let either \eqref{B.ns} be satisfied or 
\hb{0\leq s<\nu_\ell/q}. Then $\vph{\gF}_0\gF_q^{s/\mfnu}$ is a closed linear 
subspace of $\gF_q^{s/\mfnu}\RhdE$. 
\item[\rm(ii)] 
If \eqref{B.ns} applies, then 
\hb{\ci\gE\in\cL\bigl(\gF_{q;\eta}^{s/\mfnu}\cBHE, 
   \vph{\gF}_0\gF_{q;\eta}^{s/\mfnu}\bigr)} \hbox{$\eta$-uniformly}. 
\item[\rm(iii)] 
Assume 
\hb{0\leq s<\nu_\ell/q}. Then 
\hb{\ci\gE\in\cL\bigl(\gF_{q;\eta}^{s/\mfnu}\BHE, 
   \vph{\gF}_0\gF_{q;\eta}^{s/\mfnu}\bigr)} \hbox{$\eta$-uniformly}. 
\end{itemize} 
\end{cor} 
The universality of~$\gE$ means that it has a representation which is 
independent of $s$, $q$, and~$\eta$. 

\smallskip 
It is of fundamental importance for what follows that all estimates 
contained implicitly or explicitly in the preceding theorems hold 
\hbox{$\eta$-uniformly}. 
\section{Anisotropic Dilations\label{sec-A}} 
\extraindent 
Henceforth, 
\hb{\sZ:=\BR^d\times\BR^+}. Its general point is written as 
\hb{\za=(\xi,\eta)} with 
$$ 
\xi=(\xi_1,\ldots,\xi_\ell) 
\in\BR^{d_1}\times\cdots\times\BR^{d_\ell}=\BR^d. 
$$ 
We equip~$\sZ$ with the \hbox{$\nu$\emph{-augmented}} 
\emph{weight system} 
\begin{equation}\label{A.a} 
\bigl[\ell+1,(\mfd,1),(\mfnu,\nu)\bigr], 
\end{equation} 
that is, we assign the weight~$\nu$ to the variable~$\eta$. Then 
$$ 
t\btdot\za:=(t^{\nu_1}\xi_1,\ldots,t^{\nu_\ell}\xi_\ell,t^\nu\eta) 
\qa t>0 
\qb \za\in\sZ, 
\npb 
$$ 
is the \emph{anisotropic dilation} on~$\sZ$ associated with~\eqref{A.a}. 

\smallskip 
Let $X$ be a Banach space and 
\hb{\thsZ:=\sZ\ssm\{0\}}. Given 
\hb{u\in C\thsZX}, we set 
\hb{\sa_tu(\za):=u(t\btdot\za)} for 
\hb{t>0} and 
\hb{\za\in\thsZ}. Then $u$~is \emph{positively \hbox{$z$-homogeneous}} 
(with respect to~\eqref{A.a}), where 
\hb{z\in\BC}, if 
\hb{\sa_tu=t^zu} for 
\hb{t>0}. 

\smallskip 
The \emph{natural quasinorm}, 
\hb{\Lda\sco\sZ\ra\BR^+}, on~$\sZ$ (with respect to~\eqref{A.a}) 
is defined by 
$$ 
\Lda(\za):=\Bigl(\sum_{i=1}^\ell|\xi_i|^{2\nu/\nu_i}+\eta^2\Bigr)^{1/2\nu} 
\qa \za\in\sZ. 
$$ 
It is positively \hbox{$1$-homogeneous}. Moreover, 
$$ 
r_\Lda\sco\thsZ\ra[\Lda=1] 
\qb \za\mt\Lda^{-1}(\za)\btdot\za 
$$ 
is the \hbox{$\Lda$\emph{-retraction}} onto 
\hb{[\Lda=1]:=\bigl\{\,\za\in\sZ\ ;\ \Lda(\za)=1\,\bigr\}}. It is 
a continuous (topological) retraction, since 
\hb{\Lda\bigl(r_\Lda(\za)\bigr)=\Lda^{-1}(\za)\Lda(\za)=1}. 

\smallskip 
Assume 
\hb{z\in\BC} and 
\hb{a\in C\thsZX} is positively \hbox{$z$-homogeneous}. If 
\hb{\al\in\BN^d} and 
\hb{\pa_\xi a\in C\thsZX}, then $\pa_\xi$~is positively 
\hb{(z-\al\btdot\mfom)}-homogeneous, 
\begin{equation}\label{A.da} 
\pa_\xi a=\Lda^{z-\al\ibtdot\mfom}(\pa_\xi a)\circ r_\Lda 
\end{equation} 
and 
\begin{equation}\label{A.dan} 
|\pa_\xi a|_X 
\leq\Lda^{\Re z-\al\ibtdot\mfom}\,\|(\pa_\xi a)\circ r_\Lda\|_\iy 
\npb 
\end{equation} 
(cf.~\cite[Lemma~1.2.1]{Ama09a}). 

\smallskip 
By $\cH_z\sZE$ we denote the vector space of all positively 
\hbox{$z$-homogeneous} 
\hb{a\in C\thsZE} such that 
\hb{\pa_\xi a\in C\thsZE} for 
\hb{\al\in\BN^d} with 
\hb{\al\btdot\mfom\leq2\,|\mfom|}. It is a Banach space with the norm 
$$ 
\|a\|_{\cH_z}:=\max_{\al\ibtdot\mfom\leq2\,|\mfom|} 
\|(\pa_\xi a)\circ r_\Lda\|_\iy. 
$$ 
It is easily verified that 
\begin{equation}\label{A.L} 
\Lda^z\in\cH_z(\sZ). 
\end{equation} 
Let 
\hb{\ba\in\cL(E_0,E_1;E_2)} and 
\hb{z_0,z_1\in\BC}. Using Leibniz' rule, we get 
\begin{equation}\label{A.b} 
\sfm_\ba\in\cL 
\bigl(\cH_{z_0}\sZEn,\cH_{z_1}\sZEe;\cH_{z_0+z_1}\sZEz\bigr) 
\end{equation} 
and the map 
\hb{\ba\mt\sfm_\ba} is linear and continuous. If 
\hb{a\in\cH_z\bigl(\sZ,\Lis\EnEe\bigr)}, then 
$$ 
a^{-1}:=\bigl(\za\mt a(\za)^{-1}\bigr) 
\in\cH_{-z}\bigl(\sZ,\Lis\EeEn\bigr) 
$$ 
and 
\begin{equation}\label{A.1} 
\|a^{-1}\|_{\cH_{-z}} 
\leq c\bigl(\|a\|_{\cH_z},\|a^{-1}\circ r_\Lda\|_\iy\bigr) 
\npb 
\end{equation} 
(cf.~Lemmas 1.4.1 and~1.4.3 in~\cite{Ama09a}). 

\smallskip 
Given 
\hb{a\sco\thsZ\ra X}, we set 
\hb{a_\eta:=a(\cdot,\eta)\sco\BR^d\ra X} for 
\hb{\eta>0}. The linear subspace of~$C\RhdE$ of 
all~$a$ satisfying 
\hb{\pa a\in C\RhdE} for 
\hb{\al\btdot\mfom\leq2\,|\mfom|}, endowed with the norm 
$$ 
\|a\|_{\cM_\eta}:= 
\max_{\al\ibtdot\mfom\leq2\,|\mfom|}\|\Lda_\eta^{\al\ibtdot\mfom}\pa a\|_\iy 
<\iy, 
$$ 
 is denoted by 
\hb{\cM_\eta(E)=\cM_\eta\RhdE}. It is a Banach space. 
As a consequence of \eqref{A.da} and~\eqref{A.dan} we obtain  
\begin{equation}\label{A.HM} 
(a\ra a_\eta) 
\in\cL\bigl(\cH_0\sZE,\cM_\eta(E)\bigr) 
\quad\text{$\eta$-uniformly}. 
\end{equation} 
Similarly as above, if 
\hb{\ba\in\cL(E_0,E_1;E_2)}, then 
$$ 
\sfm_\ba 
\in\cL\bigl(\cM_\eta(E_0),\cM_\eta(E_1);\cM_\eta(E_2)\bigr) 
\quad\text{$\eta$-uniformly}. 
$$ 
\section{Fourier Multipliers\label{sec-M}} 
\extraindent 
We write 
\hb{\cF=(u\mt\hat u)} for the Fourier transform on $\cS'\RhdE$, the space of 
\hbox{$E$-valued} tempered distributions on~$\BR^d$, and 
\hb{D:=-\imi\pl=-\imi(\pl_1,\ldots,\pl_d)}. If 
\hb{a\in C\bigl(\BR^d,\cL(E)\bigr)}, then 
\hb{a(D):=\cF^{-1}a\cF} is the Fourier multiplier operator with symbol~$a$. 
It is a linear map in~$\cS'\RhdE$ whose domain is the set of all 
\hb{u\in\cS'\RhdE} with 
\hb{a\hat u\in\cS'\RhdE}. In particular, 
$$ 
J_\eta^z:=\Lda_\eta^z(D)\in\cL\bigl(\cS'\RhdE\bigr). 
$$ 

\smallskip 
The next two theorems form the fundament on which we build our proofs. 
Throughout this section, 
\hb{\BX=\BR^d}. 
\begin{thm(A)}\label{thm-M.J} 
Let $(s_0,q)$ and $(s_1,q)$ be \hbox{$\nu$-admissible}. Then 
$$ 
J_\eta^{s_1-s_0} 
\in\cL(\gF_{q;\eta}^{s_1/\mfnu},\gF_{q;\eta}^{s_0/\mfnu}) 
\quad\text{$\eta$-uniformly}. 
$$ 
\end{thm(A)} 
\goodbreak 
\begin{thm(A)}\label{thm-M.M} 
Let $(s,q)$ be \hbox{$\nu$-admissible}. 
\begin{itemize} 
\item[\rm(i)] 
Suppose \hb{a_\eta\in\cM_\eta\bigl(\cL(E)\bigr)} for 
\hb{\eta>0}. Then 
\hb{a_\eta(D)\in\cL(\gF_q^{s/\mfnu})} and 
$$ 
\|a_\eta(D)\|_{\cL(\gF_{q;\eta}^{s/\mfnu})} 
\leq c\,\|a_\eta\|_{\cM_\eta} 
\quad\text{$\eta$-uniformly}. 
$$ 
\item[\rm(ii)] 
If 
\hb{b_\eta\in\cM_\eta\bigl(\cL(E)\bigr)} for 
\hb{\eta>0}, then
\hb{(a_\eta b_\eta)(D)=a_\eta(D)b_\eta(D)}. 
\end{itemize} 
\end{thm(A)} 
Detailed proofs for these two theorems are given in~\cite{Ama17a} (see 
also~\cite{Ama09a} for some preliminary results not covering the case 
\hb{q=\iy}). Here we restrict ourselves to some remarks. 

\smallskip 
(1) 
Consider the trivial weight system $[1,1,1]$. Let 
\hb{\eta=1} and assume $s$, $s_0$, and~$s_1$ belong to~$\BN$ (so that 
\hb{1<q<\iy} by admissibility). Then 
\hb{\gF_q^s\doteq H_q^s}, a~Bes\-sel potential space. 
In this case the `lifting' Theorem~\ref{thm-M.J} is well-known 
(e.g., \cite{Tri78a},~\cite{Tri83a}). Its anisotropic version is contained 
in \cite[Theorem~3.7.1]{Ama09a}. 

\smallskip 
In the isotropic, resp.\ anisotropic, case each 
\hb{a\in\cM\bigl(\cL(E)\bigr)} is~a Mikhlin, resp.\ Marcinkiewicz, 
multiplier. Thus, in the present setting, Theorem~\ref{thm-M.M} follows by 
combining Theorem~\ref{thm-M.J} with the Mikhlin, resp.\ Marcinkiewicz, 
multiplier theorem for~$L_q\RnE$. 

\smallskip 
It should be noted that the \hbox{$\nu$-admissibility} assumption excludes 
the choices 
\hb{q=1} and 
\hb{q=\iy} for which these multiplier theorems do not hold. 

\smallskip 
(2) 
Let 
\hb{s,s_0,s_1\notin\BN} and 
\hb{1\leq q\leq\iy}. If 
\hb{q<\iy}, then 
\hb{\gF_q^{s/\mfnu}\doteq B_{q,q}^{s/\mfnu}}, an an\-i\-so\-tropic 
Besov space, and 
\hb{\gF_\iy^{s/\mfnu}\doteq b_{\iy,\iy}^{s/\mfnu}}, an anisotropic little 
Besov space. Thus it follows that the above theorems are parameter-dependent 
anisotropic extensions of the corresponding results established 
in~\cite{Ama97b} in the isotropic case. As in that paper, $E$~can then be 
replaced by an arbitrary infinite-dimensional Banach space. 

\smallskip 
By combining these two theorems we arrive at multiplier theorems involving 
\hbox{$\gF_q^{s/\mfnu}$-spaces} of different order. 
\goodbreak 
\begin{thm(A)}\label{thm-M.JM} 
Let $(s_0,q)$ and $(s_1,q)$ be \hbox{$\nu$-admissible}. 
\begin{itemize} 
\item[\rm(i)] 
Assume 
\hb{a_\eta\in C\bigl(\BR^d,\cL(E)\bigr)} satisfies 
\hb{\Lda_\eta^{s_0-s_1}a_\eta\in\cM_\eta\bigl(\cL(E)\bigr)} for 
\hb{\eta>0}. Then 
\hb{a_\eta(D)\in\cL(\gF_q^{s_1/\mfnu},\gF_q^{s_0/\mfnu})} and 
\begin{equation}\label{M.aL} 
\|a_\eta(D)\|_{\cL(\gF_{q;\eta}^{s_1/\mfnu},\gF_{q;\eta}^{s_0/\mfnu})} 
\leq c\,\|\Lda_\eta^{s_0-s_1}a_\eta\|_{\cM_\eta} 
\quad\text{$\eta$-uniformly}. 
\end{equation} 
\item[\rm(ii)] 
If, in addition, 
\hb{a_\eta\in C\bigl(\BR^d,\Laut(E)\bigr)} with 
\begin{equation}\label{M.a1} 
\Lda_\eta^{s_1-s_0}a_\eta^{-1}\in BC\bigl(\BR^d,\cL(E)\bigr) 
\quad\text{$\eta$-uniformly}, 
\end{equation} 
then 
\hb{a_\eta(D)\in\Lis(\gF_q^{s_1/\mfnu},\gF_q^{s_0/\mfnu})} with 
\hb{a_\eta(D)^{-1}=a_\eta^{-1}(D)} and 
$$ 
\|a_\eta(D)^{-1}\|_{\cL(\gF_{q;\eta}^{s_0/\mfnu},\gF_{q;\eta}^{s_1/\mfnu})} 
\leq c\bigl(\|\Lda_\eta^{s_0-s_1}a_\eta\|_{\cM_\eta}, 
\|\Lda_\eta^{s_1-s_0}a_\eta^{-1}\|_\iy\bigr)  
\npb 
$$ 
\hbox{$\eta$-uniformly}. 
\end{itemize} 
\end{thm(A)}
\begin{proof} 
(1) 
We set 
\hb{b_\eta:=\Lda_\eta^{s_0-s_1}a_\eta}. Then the assumptions and 
Theorem~\ref{thm-M.M} imply 
\hb{b_\eta(D)\in\cL(\gF_q^{s_1/\mfnu})} and 
\begin{equation}\label{M.b} 
\|b_\eta\|_{\cL(\gF_{q;\eta}^{s_1/\mfnu})} 
\leq c\,\|\Lda_\eta^{s_0-s_1}a_\eta\|_{\cM_\eta} 
\quad\text{$\eta$-uniformly}. 
\end{equation} 
Hence 
$$ 
\bal 
a_\eta(D)u 
&=\cF^{-1}a_\eta\cF u 
 =\cF^{-1}\Lda_\eta^{s_1-s_0}\Lda_\eta^{s_0-s_1}a_\eta\cF u\cr 
&=\cF^{-1}\Lda_\eta^{s_1-s_0}\cF\cF^{-1}\Lda_\eta^{s_0-s_1}a_\eta\cF u 
 =J_\eta^{s_1-s_0}b_\eta(D)u 
\eal 
\npb 
$$ 
for 
\hb{u\in\gF_q^{s_1}}. Now \eqref{M.aL} follows from \eqref{M.b} and 
Theorem~\ref{thm-M.J}. 
 
\smallskip 
(2)  
Let the additional hypothesis be satisfied. We obtain from \eqref{M.a1} and 
Lemma~1.4.2 in~\cite{Ama09a} that 
\hb{b_\eta^{-1}\in\cM_\eta\bigl(\cL(E)\bigr)} and 
$$ 
\|b_\eta^{-1}\|_{\cM_\eta} 
\leq c\bigl(\|b_\eta\|_{\cM_\eta},\|b_\eta^{-1}\|_\iy\bigr) 
\quad\text{$\eta$-uniformly}. 
$$ 
Thus, as in step~(1), 
\ \hb{a_\eta^{-1}(D)\in\cL(\gF_q^{s_0/\mfnu},\gF_q^{s_1/\mfnu})} and 
$$ 
\|a_\eta^{-1}(D)\|_{\cL(\gF_{q;\eta}^{s_0/\mfnu},\gF_{q;\eta}^{s_1/\mfnu})}  
\leq c\bigl(\|\Lda_\eta^{s_0-s_1}a_\eta\|_{\cM_\eta}, 
\|\Lda_\eta^{s_1-s_0}a_\eta^{-1}\|_\iy\bigr) 
\npb  
$$ 
\hbox{$\eta$-uniformly}. 

\smallskip 
If 
\hb{u\in\gF_q^{s_1/\mfnu}}, then 
$$ 
a_\eta^{-1}(D)a_\eta(D)u 
=\cF^{-1}a_\eta^{-1}\cF\cF^{-1}a_\eta\cF u 
=\cF^{-1}a_\eta^{-1}a_\eta\cF u 
=u.  
$$ 
Analogously, 
$$ 
a_\eta(D)a_\eta^{-1}(D)v 
=\cF^{-1}a_\eta\cF\cF^{-1}a_\eta^{-1}\cF v=v 
\qa v\in\cF_q^{s_0/\mfnu}, 
\npb 
$$ 
Now the assertion is clear. 
\end{proof} 
\begin{cor}\label{cor-M.H} 
Let $(s_0,q)$ and $(s_1,q)$ be \hbox{$\nu$-admissible}. 
\begin{itemize} 
\item[\rm(i)] 
Suppose 
\hb{a\in\cH_{s_1-s_0}\bigl(\sZ,\cL(E)\bigr)}. Then 
\hb{a_\eta(D)\in\cL(\gF_q^{s_1/\mfnu},\gF_q^{s_0/\mfnu})} and 
$$ 
\|a_\eta(D)\|_{\cL(\gF_{q;\eta}^{s_1/\mfnu},\gF_{q;\eta}^{s_0/\mfnu})}  
\leq c\,\|a\|_{\cH_{s_1-s_0}} 
\quad\text{$\eta$-uniformly}. 
$$ 
\item[\rm(ii)] 
Let 
\hb{a\in\cH_{s_1-s_0}\bigl(\sZ,\Laut(E)\bigr)} satisfy 
$a^{-1}\circ r_\Lda\in BC\bigl([\Lda=\nolinebreak 1],\cL(E)\bigr)$. Then 
\hb{a_\eta(D)\in\Lis(\gF_q^{s_1/\mfnu},\gF_q^{s_0/\mfnu})} with 
\hb{a_\eta(D)^{-1}=a_\eta^{-1}(D)} and 
$$ 
{}
\quad\  
\|a_\eta^{-1}(D)\|_{\cL(\gF_{q;\eta}^{s_0/\mfnu},\gF_{q;\eta}^{s_1/\mfnu})} 
\leq c\bigl(\|a\|_{\cH_{s_1-s_0}},\|a^{-1}\circ r_\Lda\|_\iy\bigr) 
\quad\text{$\eta$-uniformly}. 
$$  
\end{itemize} 
\end{cor}
\begin{proof} 
It follows from \eqref{A.L} and~\eqref{A.b} that 
\hb{\Lda^{s_0-s_1}a\in\cH_0\bigl(\sZ,\cL(E)\bigr)} and 
\newline 
\hb{\|\Lda^{s_0-s_1}a\|_{\cH_0}\leq c\,\|a\|_{\cH_{s_1-s_0}}}. Hence 
the first assertion is a consequence of~\eqref{A.HM} and part~(i) 
of the theorem. Now we get assertion~(ii) by analogous arguments 
from~\eqref{A.1}. 
\end{proof} 
\section{The Full-Space Model Case\label{sec-H}} 
\extraindent 
In this section we consider the flat case 
\hb{\Mg=\Rmgm}. We restrict ourselves to constant coefficient 
principal part operators. More precisely, we assume 
\begin{equation}\label{H.ass} 
\left. 
\bal 
\bt\quad 
&\cA={\textstyle\sum_{|a|=r}}a_\al D^\al,\ a_\al\in\cL(E).\cr
\bt\quad 
&\cA\text{ is normally $\ve$-elliptic} 
\eal 
\quad\right\} 
\end{equation} 
for some 
\hb{\ve\in(0,1]}. We set 
$$ 
\sfa:=\sum_{|a|=r}|a_\al|_{\cL(E)} 
$$ 
and note that 
\hb{|\gss\cA(\xi)|_{\cL(E)}\leq\sfa} for 
\hb{|\xi|=1}. We fix a constant~$\bka$ satisfying 
\hb{\sfa+\ve^{-1}\leq\bka}. We set 
\hb{d:=m}, endow~$\BR^m$ with the trivial weight system, and equip 
\hb{\sZ:=\BR^m\times\BR^+} with the 
\hbox{$1$-augmentation} of it. Then we put 
$$ 
a(\za):=\eta^r+\gss\cA(\xi) 
\qa \za=(\xi,\eta)\in\sZ. 
$$ 
Observe that 
\hb{\Lda(\za)=(|\xi|^2+\eta^2)^{1/2}} and 
\begin{equation}\label{H.ar} 
a\in\cH_r\bigl(\sZ,\cL(E)\bigr) 
\qb \|a\|_{\cH_r}\leq c(\sfa). 
\npb 
\end{equation} 
As usual, 
\hb{\rho(A):=\BC\ssm\sa(A)} is the resolvent set of a linear operator~$A$. 
\begin{lem}\label{lem-H.a} 
${}$
\hb{[\Re z\geq0]\is\rho\bigl(-a(\za)\bigr)} and 
$$ 
\big|\bigl(\lda+a(\za)\bigr)^{-1}\big|_{\cL(E)} 
\leq c(\bka)\bigl(\Lda^r(\za)+|\lda|\bigr)^{-1} 
\npb 
$$ 
for 
\hb{\Re\lda\geq0} and 
\hb{\za\in\thsZ}. 
\end{lem}
\begin{proof} 
By the normal \hbox{$\ve$-ellipticity} and the \hbox{$r$-homogeneity} 
of~$\gss\cA$ we get 
$$ 
\sa\bigl(a(\za)\bigr)\is\big[\Re z\geq\ve\Lda^r(\za)\bigr] 
\qa \za\in\thsZ. 
$$ 
Let 
\hb{\Lda(\za)=1}. If 
\hb{|\xi|^2\geq1/2}, then 
\begin{equation}\label{H.sa} 
\sa\bigl(a(\za)\bigr)\is[\Re z\geq\ve/2^{r/2}].  
\npb 
\end{equation} 
Otherwise, 
\hb{\eta^2\geq1/2} and \eqref{H.sa} applies as well. 

\smallskip 
Suppose 
\hb{z\in\sa\bigl(\lda+a(\za)\bigr)} with 
\hb{\Re\lda\geq0} and 
\hb{\Lda(\za)=1}. Then 
\hb{z=\lda+\mu} with 
\hb{\mu\in\sa\bigl(a(\za)\bigr)}. Hence 
\hb{|\mu|\geq\Re\mu\geq\ve/2^{r/2}} by~\eqref{H.sa}. Since 
\hb{\det\bigl(\lda+a(\za)\bigr)} equals the product of the eigenvalues of 
\hb{\lda+a(\za)}, counted with multiplicities, 
$$ 
\big|\det\bigl(\lda+a(\za)\bigr)\big|\geq(\ve/2^{r/2})^N 
\qa \Re\lda\geq0 
\qb \za\in[\Lda=1], 
$$ 
where 
\hb{N=\dim(E)}. Now we deduce from Cramer's rule 
(e.g.,~\cite[(I.4.12]{Kat66a}) that 
\hb{\lda\in\rho\bigl(-a(\za)\bigr)} and 
\begin{equation}\label{H.z} 
\big|\bigl(\lda+a(\za)\bigr)^{-1}\big|_{\cL(E)}\leq c(\bka) 
\qa \za\in[\Lda=1], 
\end{equation} 
provided 
\hb{\Re\lda\geq0} with 
\hb{|\lda|\leq 2(1+\sfa)}. If 
\hb{|\lda|\geq2(1+\sfa)\geq2\,\|a\circ r_\Lda\|_\iy}, then a Neumann series 
argument shows that 
$$ 
|\lda|\,\big|\bigl(\lda+a(\za)\bigr)^{-1}\big|_{\cL(E)} 
=\big|\bigl(1+\lda^{-1}a(\za)\bigr)^{-1}\big|_{\cL(E)}\leq 2 
\qa \za\in[\Lda=1]. 
$$ 
By combining this with \eqref{H.z} we find 
$$ 
\big|\bigl(\lda+a(\za)\bigr)^{-1}\big|_{\cL(E)} 
\leq c(\bka)(1+|\lda|)^{-1} 
\qa \Re\lda\geq0 
\qb \Lda(\za)=1. 
\npb 
$$ 
Now the assertion follows from 
\hb{\lda+a=\Lda^r(\Lda^{-r}\lda+a\circ r_\Lda)}. 
\end{proof} 
We set 
\hb{\tilde d:=d+1=m+1} and consider the \hbox{$r$-parabolic} weight system 
\hb{[\tilde\ell,\tilde\mfd,\tilde\mfnu]=\bigl[2,(m,1),(1,r)\bigr]} on 
\hb{\BR^{\tilde d}=\BR^m\times\BR}. Then we set 
$$ 
\gF_{q;\eta}^{s/\vec r}:=\gF_{q;\eta}^{s/\tilde\mfnu}(\BR^m\times\BR,E). 
$$ 
We also let 
\hb{\cA_\eta:=\eta+\cA} and study the normally \hbox{$\ve$-parabolic} 
differential operator 
\hb{\pl_t+\cA_\eta} on 
\hb{\BR^m\times\BR}. 
\begin{thm(A)}\label{thm-H.P} 
Let $(s,q)$ be \hbox{$r$-admissible}. Then 
\hb{\pl_t+\cA_\eta} is an element of  
$\Lis(\gF_q^{(s+r)/\vec r},\gF_q^{s/\vec r})$ and 
$$ 
\|\pl_t+\cA_\eta\| 
_{\cL(\gF_{q;\eta}^{(s+r)/\vec r},\gF_{q;\eta}^{s/\vec r})} 
+\|(\pl_t+\cA_\eta)^{-1}\| 
_{\cL(\gF_{q;\eta}^{s/\vec r},\gF_{q;\eta}^{(s+r)/\vec r})} 
\leq c(\bka) 
\npb 
$$ 
\hbox{$\eta$-uniformly}. 
\end{thm(A)}
\begin{proof} 
We endow 
\hb{\tilde\sZ:=\BR^m\times\BR\times\BR^+} with 
the \hbox{$r$-augmentation} of 
\hb{[\tilde\ell,\tilde\mfd,\tilde\mfnu]}. Then 
\hb{r=\LCM(\tilde\mfnu)}, and the natural quasinorm on~$\tilde\sZ$ 
is given by 
$$ 
\tilde\Lda(\tilde\za)=\bigl(|\xi|^{2r}+|\tau|^2+\eta^2\bigr)^{1/2r} 
\sim\bigl(\Lda^{2r}(\xi,\eta^{1/r})+|\tau|^2\bigr)^{1/2r} 
\qa \tilde\za=(\xi,\tau,\eta)\in\tilde\sZ, 
$$ 
with 
\hb{\za=(\xi,\eta)\in\sZ}. We set 
$$ 
\tilde a(\tilde\za):=-\imi\tau+\eta+\gss\cA(\xi). 
$$ 
It is obvious that 
\begin{equation}\label{H.a} 
\tilde a\in\tilde\cH_r:=\cH_r\bigl(\tilde\sZ,\cL(E)\bigr) 
\qb \|\tilde a\|_{\tilde\cH_r}\leq c(\bka). 
\end{equation} 
Since 
\hb{\tilde a(\tilde\za)=-\imi\tau+a(\xi,\eta^{1/r})}, it follows from 
Lemma~\ref{lem-H.a} that it is invertible for 
\hb{\tilde\za\neq0} and 
$$ 
|\tilde a^{-1}(\tilde\za)|_{\cL(E)} 
\leq c(\bka)\bigl(\Lda^r(\xi,\eta^{1/r})+|\tau|\bigr)^{-1} 
\leq c(\bka)\tilde\Lda^{-r}(\tilde\za). 
$$ 
Thus 
\hb{\|\tilde a^{-1}\circ r_{\tilde\Lda}\|_\iy\leq c(\bka)}. Hence we infer  
from \eqref{A.1} and~\eqref{H.a} that 
\hb{\tilde a^{-1}\in\tilde\cH_{-r}} and 
\hb{\|\tilde a^{-1}\|_{\tilde\cH_{-r}}\leq c(\bka)}. Now the assertion is a 
consequence of Corollary~\ref{cor-M.H} and the fact that 
\hb{\tilde a_\eta(\tilde D)=\pl_t+\cA_\eta}, where 
\hb{\tilde D:=(D,D_t)}. 
\end{proof} 
\section{The Semigroup\label{sec-S}} 
\extraindent 
We continue to presuppose conditions~\eqref{H.ass} and use the notations 
of the preceding section. Then 
\hb{\gF_q^s=\gF_q^s\RmE}. 
\begin{thm(A)}\label{thm-S.R} 
Let $(s,q)$ be \hbox{$1$-admissible}. Then 
\hb{\cA_\eta\in\Lis(\gF_q^{s+r},\gF_q^s)}, the half-plane 
\hb{[\Re z\geq0]} is contained in 
\hb{\rho(-\cA_\eta)}, and 
\begin{equation}\label{S.R} 
\|\cA_\eta\|_{\cL(\gF_{q;\eta}^{s+r},\gF_{q;\eta}^s)} 
+(|\lda|+\eta)^{1-j}\,\|(\lda+\cA_\eta)^{-1}\| 
_{\cL(\gF_{q;\eta}^s,\gF_{q;\eta}^{s+jr})} 
\leq c(\bka) 
\npb  
\end{equation} 
for 
\hb{\Re\lda\geq0}, 
\ \hb{\eta>0}, and 
\hb{j=0,1}. 
\end{thm(A)} 
\begin{proof} 
First we infer from \eqref{H.ar}, \eqref{A.L}, and~\eqref{A.b} that 
\hb{\Lda^{-r}a\in\cH_0\bigl(\sZ,\cL(E)\bigr)} and 
\hb{\|\Lda^{-r}a\|_{\cH_0}\leq c(\sfa)}. Hence, 
by Corollary~\ref{cor-M.H}(i), 
\begin{equation}\label{S.aD} 
a_\eta(D)\in\cL(\gF_q^{s+r},\gF_q^s) 
\qb \|a_\eta(D)\|_{\cL(\gF_{q;\eta}^{s+r},\gF_{q;\eta}^s)}\leq c(\sfa) 
\end{equation} 
\hbox{$\eta$-uniformly}. Using \eqref{H.ar} once more, we obtain from 
\eqref{A.dan} and Lemma~\ref{lem-H.a} that 
$$ 
\Lda^{\ba\ibtdot\mfom}\,|(\pl_\xi^\ba a)(\lda+a)^{-1}|_{\cL(E)}(\za)
\leq c(\sfa)\Lda^r(\za)\bigl(\Lda^r(\za)+|\lda|\bigr)^{-1} 
\leq c(\sfa) 
$$ 
for 
\hb{\ba\in\BN^m}, 
\ \hb{\za\in\thsZ}, and 
\hb{\Re\lda\geq0}. From this, \cite[Lemma~1.4.2]{Ama09a}, and 
Lemma~\ref{lem-H.a} we get 
\begin{equation}\label{S.a1} 
\left. 
\begin{split} 
\Lda^{\al\ibtdot\mfom}(\za)\,|\pa_\xi(\lda+a)^{-1}(\za)|_{\cL(E)} 
&\leq c(\sfa)\,\big|\bigl(\lda+a(\za)\bigr)^{-1}\big|_{\cL(E)}\cr 
&\leq c(\bka)\bigl(\Lda^r(\za)+|\lda|\bigr)^{-1} 
\end{split} 
\right. 
\end{equation} 
for 
\hb{\al\in\BN^m}, 
\ \hb{\za\in\thsZ}, and 
\hb{\Re\lda\geq0}. Using \eqref{A.L}, \eqref{A.da}, \eqref{S.a1}, and 
Leibniz' rule, we find 
\begin{equation}\label{S.aL} 
\Lda^{\al\ibtdot\mfom}\,\big|\pa_\xi\bigl(\Lda^r(\lda+a)^{-1}\bigr)
\big|_{\cL(E)}(\za)
\leq c(\bka)\Lda^r(\za)\bigl(\Lda^r(\za)+|\lda|\bigr)^{-1} 
\leq c(\bka) 
\npb 
\end{equation} 
for 
\hb{\al\in\BN^m} with 
\hb{\al\btdot\mfom\leq2\,|\mfom|=2m}, 
\ \hb{\za\in\thsZ}, and 
\hb{\Re\lda\geq0}. 

\smallskip 
Note that \eqref{S.a1} guarantees 
$$ 
(\lda+a_\eta)^{-1}\in\cM_\eta\bigl(\BR^m,\cL(E)\bigr) 
\qb \|(\lda+a_\eta)^{-1}\|_{\cM_\eta}\leq c(\bka)(|\lda|+\eta^r)^{-1} 
$$ 
\hbox{$\eta$-uniformly} for 
\hb{\Re\lda\geq0}. Similarly, by \eqref{S.aL}, 
$$ 
\Lda_\eta^r(\lda+a_\eta)^{-1}\in\cM_\eta\bigl(\BR^m,\cL(E)\bigr) 
\qb \|\Lda_\eta^r(\lda+a_\eta)^{-1}\|_{\cM_\eta} 
\leq c(\bka) 
$$ 
\hbox{$\eta$-uniformly} for 
\hb{\Re\lda\geq0}. Hence, by Theorem~\ref{thm-M.M}, 
\begin{equation}\label{S.aD1} 
(\lda+a_\eta)^{-1}(D)\in\cL(\gF_\eta^s) 
\qb \|(\lda+a_\eta)^{-1}(D)\|_{\cL(\gF_{q;\eta}^s)} 
\leq c(\bka)(|\lda|+\eta^r)^{-1} 
\end{equation} 
\hbox{$\eta$-uniformly}, and, similarly, 
$$ 
J_\eta^r(\lda+a_\eta)^{-1}(D) 
=\bigl(\Lda_\eta^r(\lda+a_\eta)^{-1}\bigr)(D)\in\cL(\gF_q^s) 
$$ 
and, due to Theorem~\ref{thm-M.J}, 
$$ 
\|(\lda+a_\eta)^{-1}(D)\|_{\cL(\gF_{q;\eta}^s,\gF_{q;\eta}^{s+r})} 
\leq c\,\|J_\eta^r(\lda+a_\eta)^{-1}(D)\|_{\cL(\gF_{q;\eta}^s)} 
\leq c(\bka) 
$$ 
\hbox{$\eta$-uniformly} for 
\hb{\Re\lda\geq0}. Using \eqref{S.aD} we find, similarly as in the proof of 
Theorem~\ref{thm-M.J}, that 
\hb{(\lda+a_\eta)^{-1}(D)=\bigl(\lda+a_\eta(D)\bigr)^{-1}}. 
Now the assertion follows from~\eqref{S.aD}, 
\ \hb{\cA_\eta=a_{\eta^{1/r}}(D)}, and \eqref{S.aD1}. 
\end{proof} 
\begin{cor}\label{cor-S.R} 
Let $(s,q)$ be \hbox{$1$-admissible}. Then 
\hb{\cA_\eta\in\cH(\gF_q^{s+r},\gF_q^s)} and the semigroup 
\hb{\{\,e^{-t\cA_\eta}\ ;\ t\geq0\,\}} is exponentially decaying. 
\end{cor} 
\begin{proof} 
Since 
\hb{\cA_\eta\in\Lis(\gF_q^{s+r},\gF_q^s)}, it follows that $\cA_\eta$~is 
closed if we consider it as a linear operator in~$\gF_q^s$ with 
domain~$\gF_q^{s+r}$ (cf.~\cite[Lemma~I.1.1.2]{Ama95a}). Moreover, it is then 
densely defined, due to \eqref{I.eW} and~\eqref{I.ebc}. Now 
\hb{\cA_\eta\in\cH(\gF_q^{s+r},\gF_q^s)} is a well-known consequence 
of the resolvent estimate contained in~\eqref{S.R}. 

\smallskip 
From semigroup theory it is known that there exists 
\hb{\vp\in(\pi/2,\pi)} such that 
\hb{[\,|\arg z|\leq\vp]\is\rho(-\cA_\eta)}. From this and the fact that 
\hb{0\in\rho(-\cA_\eta)} it follows that there exists 
\hb{\ga=\ga(\eta)>0} such that 
\hb{\sa(-\cA_\eta)\is[\Re z\leq-\ga]}, that is, 
the spectral bound of~$-\cA_\eta$ is negative. Hence the growth bound is 
negative too.
\end{proof} 
\begin{pro}\label{pro-S.S} 
Let $(s,q)$ be \hbox{$r$-admissible}. If 
\hb{f\in\gF_q^{s/\vec r}}, then 
\begin{equation}\label{S.int} 
(\pl_t+\cA_\eta)^{-1}f 
=\int_{-\iy}^te^{-(t-\tau)\cA_\eta}f(\tau)\,d\tau 
\qa \text{a.a.\ }t\in\BR. 
\end{equation} 
\end{pro} 
\begin{proof} 
(1) 
We fix 
\hb{\eta>0} and set 
\hb{U(t)=V(t):=e^{-t\cA_\eta}} for 
\hb{t\geq0}, and 
\hb{V(t)=0} for 
\hb{t<0}. Since the semigroup 
\hb{\bigl\{\,U(t)\ ;\ t\geq0\,\bigr\}} is exponentially decaying, 
it follows that $V$~belongs to 
$L_1\bigl(\BR,\cL(\gF_q^s)\bigr)$. Hence, by Young's inequality, 
\begin{equation}\label{S.V} 
(g\mt V*g)\in\cL\bigl(L_q(\BR),\gF_q^s\bigr) 
\qa 1\leq q\leq\iy, 
\end{equation} 
and 
$$ 
V*g(t)=\int_{-\iy}^tU(t-\tau)g(\tau)\,d\tau 
\qa \text{a.a.\ }t\in\BR, 
\qb g\in L_1(\BR,\gF_q^s). 
\npb 
$$ 
This remains valid if $L_\iy$~is replaced by~$\BUC$. 

\smallskip 
It is a consequence of Theorem~\ref{thm-S.R} that 
\hb{\Vsdot_{\gF_q^{s+r}}\sim\|\cA_\eta\btdot\|_{\gF_q^s}}. Thus we infer from 
Corollary~\ref{cor-S.R} that 
\hb{\bigl\{\,U(t)\ ;\ t\geq0\,\bigr\}} restricts to a strongly 
continuous exponentially decaying analytic semigroup on~$\gF_q^{s+r}$ 
(e.g.,~\cite[Theorem~V.2.1.3]{Ama95a}). 

\smallskip 
(2) 
Assume 
\hb{g\in\BUC(\BR,\gF_q^{s+r})}. Then the arguments of step~(1) show that 
\hb{v:=V*g} belongs to 
$\BUC(\BR,\gF_q^{s+r})$. Given 
\hb{h>0}, 
$$ 
v(t+h)-v(t) 
=\int_t^{t+h}U(t+h-\tau)g(\tau)\,d\tau+\bigl(U(h)-1\bigr)v(t) 
\qa t\in\BR. 
$$ 
From this we deduce that the right derivative~$\pl_t^+v$ exists in~$\gF_q^s$ 
and equals 
\hb{g-\cA_\eta v}. Since this function is continuous, 
\hb{v\in C^1(\BR,\gF_q^s)} and 
\hb{(\pl_t+\cA_\eta)v=g}, that is, 
\hb{(\pl_t+\cA_\eta)^{-1}g=V*g}. 

\smallskip 
(3) 
Suppose 
\hb{q<\iy}. Then, see \cite{Ama17a} or~\cite{Ama09a}, 
$$ 
\cS\bigl(\BR,\cS\RmE\bigr)\sdh\cS(\BR^m\times\BR,E) 
\sdh W_{\coW q}^{(s+r)/\vec r}\sdh W_{\coW q}^{s/\vec r} 
$$ 
and 
$$ 
\cS\bigl(\BR,\cS\RmE\bigr)\hr\cS(\BR,W_{\coW q}^{s+r}) 
\hr\BUC(\BR,W_{\coW q}^{s+r}). 
$$ 
Thus, if 
\hb{f\in W_{\coW q}^{s/\vec r}}, there exists a sequence~$(f_j)$ in 
\hb{W_{\coW q}^{s/\vec r}\cap\BUC(\BR,W_{\coW q}^{s+r})} converging 
in~$W_{\coW q}^{s/\vec r}$, hence, 
by Theorem~\ref{thm-B.I} in~$L_q(\BR,W_{\coW q}^s)$, towards~$f$. 
By step~(2), 
\ \hb{(\pl_t+\cA_\eta)^{-1}f_j=V*f_j} for 
\hb{j\in\BN}. It follows from \eqref{S.V} that 
\hb{V*f_j\ra V*f} in~$L_q(\BR,W_{\coW q}^s)$. Theorem~\ref{thm-H.P} implies 
that 
\hb{(\pl_t+\cA_\eta)^{-1}f_j\ra(\pl_t+\cA_\eta)^{-1}f} 
in $W_{\coW q}^{(s+r)/\vec r}$, hence in~$L_q(\BR,W_{\coW q}^s)$. 
Consequently, 
\hb{(\pl_t+\cA_\eta)^{-1}f=V*f}, which proves the assertion in this case. 

\smallskip 
(4) 
Assume 
\hb{q=\iy} and 
\hb{f\in buc^{s/\vec r}}. We see from 
\hb{buc^{(s+r)/\vec r}\sdh buc^{s/\vec r}} that there exists 
a~sequence~$(f_j)$ in~$buc^{(s+r)/\vec r}$ converging in~$buc^{s/\vec r}$ , 
hence, once more by Theorem~\ref{thm-B.I}, in~$\BUC(\BR,buc^s)$, 
towards~$f$. Since $f_j$~belongs to $\BUC(\BR,buc^{s+r})$ by 
Theorem~\ref{thm-B.I}, we get from step~(2) that 
\hb{(\pl_t+\cA_\eta)^{-1}f_j} equals 
\hb{V*f_j} for 
\hb{j\in\BN}. This implies 
\hb{(\pl_t+\cA_\eta)^{-1}f=V*f} by the arguments of the preceding step. 
\end{proof} 
\begin{cor}\label{cor-S.S} 
Let $(s,q)$ be \hbox{$r$-admissible}. Let either \eqref{B.ns} be satisfied 
or suppose 
\hb{0\leq s<r/q} and set 
\hb{\gF_q^{s/\vec r}\cBHE:=\gF_q^{s/\vec r}\BHE}. Then 
$$ 
\gR\circ(\pl_t+\cA_\eta)^{-1}\circ\ci\gE 
\in\cL\bigl(\gF_{q;\eta}^{s/\vec r}\cBHE, 
\gF_{q;\eta}^{(s+r)/\vec r}\cBHE\bigr) 
\npb 
$$ 
\hbox{$\eta$-uniformly}. 
\end{cor} 
\begin{proof} 
It follows from Corollary~\ref{cor-B.E} and Theorem~\ref{thm-H.P} that 
$$ 
(\pl_t+\cA_\eta)^{-1}\circ\ci\gE 
\in\cL\bigl(\gF_{q;\eta}^{s/\vec r}\cBHE, 
\gF_{q;\eta}^{(s+r)/\vec r}\bigr) 
\quad\text{$\eta$-uniformly}. 
$$ 
Given 
\hb{\ci\gE f\in\vph{\gF}_0\gF_q^{s/\vec r}}, we read off \eqref{S.int} that  
\begin{equation}\label{S.f} 
u(t):=(\pl_t+\cA_\eta)^{-1}\circ\ci\gE f(t)=0 
\qb \text{a.a. }t<0. 
\end{equation} 
Note that 
\hb{(s+r)/r>1+k+1/q}, where 
\hb{k:=-1} if 
\hb{s<r/q}. Hence Theorem~\ref{thm-B.I} and the (Banach-space-valued) Sobolev 
embedding theorem imply 
$$ 
\gF_q^{(s+r)/\vec r} 
\hr\gF_q^{(s+r)/r}\bigl(\BR,L_q\RmE\bigr) 
\hr C^{k+1}\bigl(\BR,L_q\RmE\bigr). 
\npb 
$$ 
From this and \eqref{S.f} we infer that 
\hb{\vec\ga^{k+1}u=0}. Now the claim follows. 
\end{proof} 
\section{Cauchy Problems\label{sec-C}} 
\extraindent 
Now we turn to the Cauchy problem 
$$ 
(\pl_t+\cA_\eta)u=f\text{ on }\BH 
\qb \ga u=u_0\text{ on }\pl\BH, 
\npb 
$$ 
retaining assumption~\eqref{H.ass}. 
\begin{thm(A)}\label{thm-C.C} 
Let $(s,q)$ be \hbox{$r$-admissible}. Then 
$$ 
(\pl_t+\cA_\eta,\ga) 
\in\Lis\bigl(\gF_{q;\eta}^{(s+r)/\vec r}\BHE,\gF_{q;\eta}^{s/\vec r}\BHE 
\times\gF_{q;\eta}^{s+r(1-1/q)}\bigr)
$$ 
\hbox{$\eta$-uniformly} \emph{with \hbox{$c(\bka)$-bounds}}, that is, 
\hb{(\pl_t+\cA_\eta,\ga_0)} and 
\hb{(\pl_t+\cA_\eta,\ga_0)^{-1}} are bounded by~$c(\bka)$, uniformly 
with respect to 
\hb{\eta>0}. 
\end{thm(A)} 
\begin{proof} 
(1) 
We write $M_\eta$, resp.~$L_\eta$, for 
\hb{\pl_t+\cA_\eta} if this operator is considered on 
\hb{\BR^m\times\BR}, resp.~$\BH$. Let $\gRgE$ be the \hbox{r-e} pair of 
Theorem~\ref{thm-B.E} for 
\hb{d=m+1}. Then Theorems \ref{thm-B.E}, \ref{thm-H.P}, and~\ref{thm-B.T} 
imply 
$$ 
(\gR\circ M_\eta\circ\gE,\ga) 
\in\cL\bigl(\gF_{q;\eta}^{(s+r)/\vec r}\BHE,\gF_{q;\eta}^{s/\vec r}\BHE 
\times\gF_{q;\eta}^{s+r(1-1/q)}\bigr)
$$ 
\hbox{$\eta$-uniformly} with \hbox{$c(\bka)$-bounds}. Since 
$\gR$~commutes with $\pa$ and~$\pl_t$, we see 
$$ 
L_\eta=\gR\circ M_\eta\circ\gE. 
$$ 

\smallskip 
(2) 
Let 
\hb{k\in\BN} and suppose  
\begin{equation}\label{C.ks} 
r(k+1/q)<s<r(k+1+1/q) 
\qa s\notin(\BN+1/q)\cup(\BN+r/q). 
\end{equation}\goodbreak  
\noindent 
Then 
\hb{s+r(1-j-1/q)} is \hbox{$r$-admissible} and, by Theorem~\ref{thm-B.F}(ii), 
\begin{equation}\label{C.A} 
\cA_\eta 
\in\cL(\gF_{q;\eta}^{s+r(1-j-1/q)},\gF_{q;\eta}^{s-r(j+1/q)}) 
\npb 
\end{equation} 
\hbox{$\eta$-uniformly} with \hbox{$c(\sfa)$-bounds} for 
\hb{0\leq j\leq k}. 

\smallskip 
Suppose 
\hb{u\in\gF_q^{(s+r)/\vec r}\BHE} and set 
\hb{f:=L_\eta u}. Then we get from Theorem~\ref{thm-B.T} and~\eqref{C.A} 
\begin{equation}\label{C.g1} 
\ga^{j+1}u=\pl_t^{j+1}u(0)=\pl_t^jf(0)-\cA_\eta\pl_t^ju(0) 
\in\gF_q^{s-r(j+1/q)} 
\end{equation} 
and 
\begin{equation}\label{C.j} 
\|\ga^{j+1}u\|_{\gF_{q;\eta}^{s-r(j+1/q)}} 
\leq c(\sfa)\,\|(f,\ga u)\| 
_{\gF_{q;\eta}^{s/\vec r}\BHE\times \gF_{q;\eta}^{s+r(1-1/q)}} 
\npb 
\end{equation} 
for 
\hb{0\leq j\leq k} and 
\hb{\eta>0}. 

\smallskip 
Assume 
\hb{(L_\eta u,\ga u)=(0,0)}. Then we see from \eqref{C.ks} and \eqref{C.j} 
that $u$ belongs to $\gF_q^{(s+r)/\vec r}\cBHE$. Hence its trivial extension 
\hb{\tilde u:=\ci\gE u} lies in~$\gF_q^{(s+r)/\vec r}$ and satisfies 
\hb{M_\eta\tilde u=0}. Consequently, 
\hb{\tilde u=0} by Theorem~\ref{thm-H.P}. Thus, taking 
Theorem~\ref{thm-B.E}(ii) into consideration, 
\hb{u=\gR\tilde u=0}. This shows that $(L_\eta,\ga)$ is injective. 

\smallskip 
(3) 
Keeping assumption~\eqref{C.ks}, we let 
\hb{(f,u_0)\in\gF_q^{s/\vec r}\BHE\times\gF_q^{s+r(1-1/q)}}. 
Define~$u_{j;\eta}$ for 
\hb{0\leq j\leq k} inductively by 
$$ 
u_{0;\eta}:=u_0 
\qb u_{j+1;\eta}:=\ga^jf-\cA_\eta u_{j;\eta}. 
$$ 
It follows from \eqref{C.A} and Theorem~\ref{thm-B.T}(i) that 
$$ 
u_{j;\eta}\in\gF_{q;\eta}^{s-r(j+1/q)} 
\qa 0\leq j\leq k, 
\npb 
$$ 
\hbox{$\eta$-uniformly} with \hbox{$c(\sfa)$-bounds}. 

\smallskip 
We set 
$$ 
\vec\gF_q^{s+r(1-1/q)}:=\prod_{j=0}^k\gF_q^{s-r(j+1/q)} 
\qb V_{\coV\eta}(f,u_0):=(u_{0;\eta},\ldots,u_{k;\eta}). 
$$ 
Then 
$$ 
V_{\coV\eta}\in\cL\bigl(\gF_{q;\eta}^{s/\vec r}\BHE 
\times\gF_{q;\eta}^{s+r(1-1/q)},\vec\gF_q^{s+r(1-1/q)}\bigr)
$$ 
\hbox{$\eta$-uniformly} with \hbox{$c(\sfa)$-bounds}. 
Theorem~\ref{thm-B.T}(i) 
guarantees the existence of an \hbox{$\eta$-uniform} coretraction 
$(\vec\ga^k)^c$ for the trace operator 
$$ 
\vec\ga^k\in\cL\bigl(\gF_{q;\eta}^{(s+r)/\vec r}(\BH), 
\vec\gF_{q;\eta}^{s+r(1-1/q)}\bigr). 
$$ 
Hence 
$$
W_{\coW\eta}:=(\vec\ga)^c\circ V_{\coV\eta} 
\in\cL\bigl(\gF_{q;\eta}^{s/\vec r}\BHE 
\times\gF_{q;\eta}^{s+r(1-1/q)},   
\gF_{q;\eta}^{(s+r)/\vec r}\BHE\bigr) 
\npb 
$$ 
\hbox{$\eta$-uniformly} with \hbox{$c(\sfa)$-bounds}. 

\smallskip 
Let 
$$ 
w_\eta:=W_{\coW\eta}(f,u_0) 
\qb g_\eta:=f-L_\eta w_\eta. 
$$ 
Then 
\hb{w_\eta\in\gF_q^{(s+r)/\vec r}\BHE} and \eqref{C.g1} imply 
$$ 
\ga^jg_\eta=\ga^jf-\ga^jL_\eta w_\eta 
=\ga^jf-\ga^{j+1}w_\eta-\cA_\eta\ga^jw_\eta=0 
\qa 0\leq j\leq k. 
$$ 
Hence 
\hb{v_\eta:=\gR\circ M_\eta^{-1}\circ\ci\gE g_\eta 
   \in\gF^{(s+r)/\vec r}\cBHE} by Corollary~\ref{cor-S.S}. 
The second part of Theorem~\ref{thm-B.E}(i) implies  
\hb{L_\eta\circ\gR=\gR\circ M_\eta}. Consequently, 
$$ 
L_\eta v_\eta=\gR\circ\ci\gE g_\eta=f-L_\eta w_\eta 
\qa \ga v_\eta=0. 
$$ 
Hence 
\hb{u_\eta:=v_\eta+w_\eta} satisfies 
\hb{L_\eta u_\eta=f} on 
\hb{M\times\BR^+} and 
\hb{\ga u_\eta=u_0}. This shows that $(L_\eta,\ga)$ is surjective, 
thus bijective, and 
$$ 
(L_\eta,\ga)^{-1}(f,u_0) 
=\gR\circ M_\eta^{-1}\circ\ci\gE\bigl(f-L_\eta W_{\coW\eta}(f,u_0)\bigr) 
+W_{\coW\eta}(f,u_0). 
\npb 
$$ 
This implies the assertion in this case. 

\smallskip 
(4) 
Assume 
\hb{0\leq s<r/q}. In this case analogous arguments result in 
$$ 
(L_\eta,\ga)(f,u_0) 
=\gR\circ M_\eta^{-1}\circ\ci\gE(f-L_\eta\ga^cu_0)+\ga^cu_0. 
\npb 
$$ 
Thus the claim holds in this case too. 

\smallskip 
(5) 
Suppose 
\hb{s\in(\BN+1/q)\cup(\BN+r/q)}. We fix 
\hb{s_0<s<s_1} such that $(s_0,q)$ and $(s_1,q)$ are \hbox{$r$-admissible} 
and 
\hb{s_0,s_1\notin(\BN+1/q)\cup(\BN+r/q)}. Then, setting 
\hb{\ta:=(s-s_0)/(s_1-s_0)}, the assertion follows by interpolation, due to 
Theorem~\ref{thm-B.In}, from what has just been shown. The theorem is proved. 
\end{proof} 
\section{Localizations of Function Spaces\label{sec-L}} 
\extraindent 
We assume that the topological space underlying~$M$ is separable and 
metrizable. Let 
\hb{Q:=(-1,1)\is\BR}. If $\ka$~is a local chart for~$M$, then we 
write~$U_{\coU\ka}$ for the
corresponding coordinate patch~$\dom(\ka)$.
A~local chart~$\ka$ is \emph{normalized} if
\hb{\ka(U_{\coU\ka})=Q^m} whenever
\hb{U_{\coU\ka}\is\ci{M}}, the interior of~$M$, and 
\hb{\ka(U_{\coU\ka})=Q^m\cap\BH^m} if
\hb{U_{\coU\ka}\cap\pl M\neq\es}. 
 
\smallskip  
An atlas~$\gK$ for~$M$ has \emph{finite multiplicity} if there exists
\hb{k\in\BN} such that any intersection of more than $k$ coordinate
patches is empty. In this case 
$$ 
\gN(\ka):=\{\,\tk\in\gK 
\ ;\ U_{\coU\tk}\cap U_{\coU\ka}\neq\es\,\}
$$
has cardinality~%
\hb{\leq k} for each 
\hb{\ka\in\gK}. An atlas is \emph{shrinkable} if it consists of
normalized charts and there exists
\hb{r\in(0,1)} such that
\hb{\big\{\,\ka^{-1}\bigl(r\ka(U_{\coU\ka})\bigr)\ ;\ \ka\in\gK\,\big\}} 
is a cover of~$M$.

\smallskip 
$\Mg$~is a \hh{uniformly regular Riemannian manifold} if 
\begin{equation}\label{L.ur} 
\left. 
\begin{split} 
\rm{(i)}\quad    
&\text{it possesses a shrinkable atlas $\gK$ of finite multiplicity}\cr 
\noalign{\vskip-1\jot} 
&\text{which is orientation preserving if $M$ is oriented}.\cr
\rm{(ii)}\quad    
&\|\tk\circ\ka^{-1}\|_{k,\iy}\leq c(k),
                    \ \ka,\tk\in\gK,\ k\in\BN.\cr
\rm{(iii)}\quad    
&\ka_*g\sim g_m,\ \ka\in\gK.\cr
\rm{(iv)}\quad    
&\|\ka_*g\|_{k,\iy}\leq c(k),\ \ka\in\gK,\ k\in\BN.  
\end{split} 
\right.  
\end{equation}  
In~(ii) and in similar situations it is understood that only 
\hb{\ka,\tk\in\gK} with 
\hb{U_{\coU\ka}\cap U_{\coU\tk}\neq\es} are being considered. 
Here and below, we employ the standard definitions of push-forward and 
pull-back operators. An atlas satisfying \eqref{L.ur}(i) and~(ii) is called 
\emph{uniformly regular}. Henceforth, it is assumed that 
$$ 
\bal 
\bt\quad 
&\Mg\text{ is a uniformly regular Riemannian manifold without boundary}\cr 
\noalign{\vskip-1\jot} 
&\text{and $\gK$ is an atlas possessing properties \eqref{L.ur}}. 
\eal 
$$ 
Observe that $\gK$~is countable. A~\emph{localization system} for~$M$ 
subordinate to~$\gK$ is a family 
\hb{\bigl\{\,(\pi_\ka,\chi_\ka)\ ;\ \ka\in\gK\,\bigr\}} such that 
\begin{equation}\label{L.LS} 
\left. 
\begin{split} 
\rm{(i)}\quad    
&\pi_\ka\in\cD\bigl(U_{\coU\ka},[0,1]\bigr)\text{ and }
 \{\,\pi_\ka^2\ ;\ \ka\in\gK\,\}\text{ is a partition of unity}\cr
\noalign{\vskip-1\jot} 
&\text{ on $M $ subordinate to the covering }
 \{\,U_{\coU\ka}\ ;\ \ka\in\gK\,\}.\cr
\rm{(ii)}\quad    
&\chi_\ka=\ka^*\chi\text{ with }\chi\in\cD\bigl(Q^m,[0,1]\bigr)\cr 
\noalign{\vskip-1\jot} 
&\text{ and $\chi\sn\tsupp(\ka_*\pi_\ka)=\mf{1}$ for }\ka\in\gK.\cr
\rm{(iii)}\quad    
&\|\ka_*\pi_\ka\|_{k,\iy}+\|\ka_*\chi_\ka\|_{k,\iy}\leq c(k),
                     \ \ka\in\gK,\ k\in\BN.
\end{split} 
\right. 
\npb 
\end{equation}  
Lemma~3.2 of~\cite{Ama12b} guarantees the existence of such systems. 

\smallskip 
Using 
\hb{T_pQ^m=\BR^m} for 
\hb{p\in Q^m} we get  
$$ 
T_\tau^\sa Q^m\otimes F 
=Q^m\times\bigl((\BR^m)^{\otimes\sa}\otimes\Rm^{*\otimes\tau} 
\otimes F\bigr). 
$$ 
Of course, we identify~$\Rm^*$ canonically with~$\BR^m$, 
but for clarity we continue to denote it by~$\Rm^*$. We endow 
\hb{T_\tau^\sa Q^m\otimes F} with the inner product 
\begin{equation}\label{L.ip} 
\prsn_{T_\tau^\sa Q^m\otimes F}
:=\prsn^{\otimes\sa}\otimes\prsn^{\otimes\tau} 
\otimes\prsn_F. 
\end{equation} 
The standard basis 
\hb{(e_1,\ldots,e_m)} of~$\BR^m$ and its dual basis 
\hb{(\ve^1,\ldots,\ve^m)} of $\Rm^*$ induce the coordinate frame 
$$ 
\bigl\{\,e_{(i)}\otimes\ve^{(j)}
\ ;\ (i)\in\BJ_\sa,\ (j)\in\BJ_\tau\,\bigr\} 
$$ 
on~$T_\tau^\sa Q^m$, where 
\hb{e_{(i)}:=e_{i_1}\otimes\cdots\otimes e_{i_\sa}}, etc.  
Then 
$$ 
u\in(T_\tau^\sa Q^m\otimes F)_p 
=\cL\bigl((\Rm^*)^{\otimes\sa}\otimes\Rm^{\otimes\tau},F\bigr) 
$$ 
has the matrix representation 
\hb{\big[u_{(j)}^{(i)}\big]\in F^{m^\sa\times m^\tau}}. If 
\hb{n=0}, then 
\hb{F=\BR}. We endow~$F^{m^\sa\times m^\tau}$ with the inner product 
$$ 
\bigl(\big[u_{(j)}^{(i)}\big] 
\bsn\big[v_{(\tilde\jmath)}^{(\tilde\imath)}\big]\bigr)_{\HS,F} 
:=\sum_{(i)\in\BJ_\sa,\,(j)\in\BJ_\tau} 
\big(u_{(j)}^{(i)}\bsn v_{(j)}^{(i)}\big)_F. 
$$ 
It coincides with the Hilbert--Schmidt inner product if 
\hb{F=\BR}. From now on, by~$E$ we always mean 
\hb{\bigl(E,\prsn_E\bigr)}, where 
$$  
\bt\quad 
E=E_\tau^\sa=E_\tau^\sa(F):=F^{m^\sa\times m^\tau} 
\qb \prsn_E:=\prsn_{\HS,F}. 
$$ 
It follows from \eqref{L.ip} that 
\hb{u\mt\big[u_{(j)}^{(i)}\big]} defines an isometric isomorphism by which 
\begin{equation}\label{L.TE} 
\text{we identify } 
T_\tau^\sa Q^m\otimes F\text{ with }Q^m\times E. 
\end{equation} 

\smallskip 
Given Banach spaces $X_0$,~$X_1$ and 
\hb{j\in\thBN}, we denote by~$\cL^j(X_0;X_1)$ the Banach space of all 
\hbox{$j$-linear} maps from 
\hb{X_0\times\cdots\times X_0} ($j$~copies) into~$X_1$, and 
\hb{\cL^0(X_0;X_1):=X_1}. 

\smallskip 
Suppose $v$~is a \hbox{$C^j$-section} of 
\hb{Q^m\times E}, that is, 
\hb{v\in C^j\QmE}. Then 
$$ 
\pl^jv\in C\bigl(Q^m,\cL^j\RmE\bigr)=C(Q^m,E_{\tau+j}^\sa), 
\npb 
$$ 
using canonical identifications. 

\smallskip 
Let 
\hb{\ka\in\gK}. Suppose 
\hb{u\in C(V)}. Denote by~$\big[u_{(k)}^{(j)}\big]$ the representation 
of~$u$ on~$U_{\coU\ka}$ with respect to the coordinate frame~\eqref{I.b}. 
Then 
$$ 
\ka_*u 
:=\big[\ka_*u_{(k)}^{(j)}\big]=\big[u_{(k)}^{(j)}\circ\ka^{-1}\big] 
\in C\QmE. 
$$ 
The push-forward of 
\hb{\na^j\sco C^j(V)\ra C(V_{\tau+j}^\sa)} is defined by 
$$ 
(\ka_*\na^j)v:=\ka_*\bigl(\na^j(\ka^*v)\bigr) 
\qa v\in C^j\QmE. 
$$ 
Then $\ka_*\na$~is a metric connection on
\hb{Q^m\times E} which satisfies 
\begin{equation}\label{L.Nj} 
\ka_*\na^jv=\pl^jv+\sum_{i=0}^{j-1}b_{j,i}^\ka\pl^iv 
\qa v\in C^j\QmE, 
\end{equation} 
with 
\hb{b_{j,i}^\ka\in C^\iy\bigl(Q^m,\cL(E_{\tau+i}^\sa,E_{\tau+j}^\sa\bigr)} 
and 
$$ 
\|b_{j,i}^\ka\|_{k,\iy}\leq c(j,k) 
\qa 0\leq i\leq j-1 
\qb j,k\in\BN 
\qb \ka\in\gK, 
\npb 
$$ 
(see the proof of \cite[Lemma~3.1]{Ama12b}). 

\smallskip 
Considering~$\gK$ as an index set endowed with the discrete topology, we set 
\hb{\mf{\gF}_q^s:=C(\gK,\gF_q^s)}, the space of all `sequences' 
in~$\gF_q^s$ `enumerated' by~$\gK$, and 
$$ 
\ell_q(\gF_q^s):=\ell_q(\gK,\gF_q^s) 
\qa 1\leq q\leq\iy. 
$$ 
If 
\hb{k<s<k+1} with 
\hb{k\in\BN}, then $\ell_{\iy,\unif}(buc^s)$ is the closed linear subspace 
of~$\ell_\iy(\gF_\iy^s)$ of all 
\hb{\mfv=(v_\ka)} such that 
\hb{\lim_{\da\ra0}\max_{|\al|=k}[\pa v_\ka]_{s-k,\iy}^\da=0}, 
uniformly with respect to 
\hb{\ka\in\gK}. 

\smallskip 
Now we fix a localization system for~$M$. Then we define 
$$ 
\cR\mfu:=\sum_\ka\pi_\ka\ka^*u_\ka 
\qb \cR^cu:=\bigl(\ka_*(\pi_\ka u)\bigr)_{\ka\in\gK} 
$$ 
for 
\hb{\mfu=(u_\ka)\in\mf{\gF}_q^s} and 
\hb{u\in L_{1,\loc}(V_{\coV\tau}^\sa)}, whenever the series is absolutely 
convergent. In the following, we often identify functions with 
multiplication operators. 
\begin{thm(A)}\label{thm-L.l} 
Suppose 
\hb{s\in\BR^+} and 
\hb{1\leq q<\iy}. Then $\cRcRc$ is an \hbox{r-e} pair for 
$$ 
\bigl(\ell_q(W_{\coW q}^s),W_{\coW q}^s(V)\bigr), 
\ \bigl(\ell_\iy(\BUC^s),BC^s(V)\bigr),\text{ and }
\bigl(\ell_{\iy,\unif}(buc^s),bc^s(V)\bigr), 
\npb 
$$ 
provided 
\hb{s\notin\BN} in the last instance. 
\end{thm(A)} 
\begin{proof} 
\cite[Theorem~6.1]{Ama12b} and \cite[Theorem~12.5]{Ama12c}. 
\end{proof} 
The next theorem shows that, similarly as in the compact case, general 
uniformly regular Riemannian manifolds can be characterized by means of local 
coordinates. 
\begin{thm(A)}\label{thm-L.n} 
Suppose 
\hb{s\in\BR^+} and 
\hb{1\leq q<\iy}. Then 
$$ 
u\mt\Bigl(\sum_{\ka\in\gK}\|\ka_*u\|_{W_{\coW q}^s\QmE}^q\Bigr)^{1/q} 
$$ 
is a norm for~$W_{\coW q}^s(V)$, and 
$$ 
u\mt\sup_{\ka\in\gK}\|\ka_*u\|_{BC^s\QmE} 
\npb 
$$ 
is one for~$BC^s(V)$. 

\smallskip 
If 
\hb{k<s<k+1} with 
\hb{k\in\BN}, then 
\hb{u\in buc^s(V)} iff 
\hb{u\in BC^k(V)} and 
$$ 
\lim_{\da\ra0}\bigl[\pa(\ka_*u)\bigr]_{s-k,\iy}^\da=0 
\qa \al\in\BN^m 
\qb |\al|=k, 
\npb 
$$ 
uniformly with respect to 
\hb{\ka\in\gK}. 
\end{thm(A)} 
\begin{proof} 
(1) 
We set 
\hb{S_\ktk:=\ka_*\circ\tk^*\circ\chi} for 
\hb{\ka,\tk\in\gK}. If 
\hb{s\in\BN}, then it is a consequence of \eqref{L.ur}(ii) and the chain 
rule that 
\begin{equation}\label{L.Sk} 
S_\ktk\in\cL(W_{\coW q}^s)\cap\cL(BC^s), 
\end{equation} 
uniformly with respect to 
\hb{\ka\in\gK} and 
\hb{\tk\in\gN(\ka)}. From this we obtain \eqref{L.Sk} for 
\hb{s\notin\BN} by interpolation with the real interpolation functor~%
\hb{\pr_{\ta,q}}, respectively~%
\hb{\pr_{\ta,\iy}} in the case of $BC$ spaces. 

\smallskip 
Since 
\hb{\|\tk\circ\ka^{-1}\|_{k+1,\iy}\leq c(k)} for 
\hb{\ka\in\gK} and 
\hb{\tk\in\gN(\ka)}, the mean-value theorem implies that 
\hb{\pl^k(\tk\circ\ka^{-1})} is uniformly Lipschitz continuous, 
uniformly with respect 
\hb{\ka\in\gK} and 
\hb{\tk\in\gN(\ka)}. From this we get 
\hb{S_\ktk\in\cL(\BUC^k)} for 
\hb{k\in\BN}, uniformly with respect to 
\hb{\ka\in\gK} and 
\hb{\tk\in\gN(\ka)}. Now, given 
\hb{s\in\BR^+\ssm\BN}, we deduce by continuous interpolation 
\begin{equation}\label{L.Skk} 
S_\ktk\in\cL(buc^s), 
\npb 
\end{equation} 
uniformly with respect 
\hb{\ka\in\gK} and 
\hb{\tk\in\gN(\ka)}. 

\smallskip 
(2) 
Using 
\hb{\sum_\ka\pi_\ka^2=1} we find, due to 
\hb{\chi_\ka\pi_\ka=\pi_\ka},  
\begin{equation}\label{L.ku} 
\ka_* u=\sum_\tk\ka_*(\pi_\tk^2u) 
=\sum_{\tk\in\gN(\ka)}(\ka_*\pi_\tk) 
S_\ktk\bigl(\tk_*(\pi_\tk u)\bigr) 
\end{equation} 
for 
\hb{u\in C(V)} and 
\hb{\ka\in\gK}. Observing 
\hb{\ka_*\pi_\tk=S_\ktk(\tk_*\pi_\tk)}, we infer from 
\eqref{L.LS}(iii) and step~(1) 
$$ 
\|\ka_*\pi_\tk\|_{\ell,\iy}\leq c(\ell) 
\qa \ka\in\gK 
\qb \tk\in\gN(\ka) 
\qb \ell\in\BN. 
$$ 
From this, \eqref{L.Sk}, \,\eqref{L.ku}, and Theorem~\ref{thm-L.l} 
it follows 
\begin{equation}\label{L.Ws} 
\Bigl(\sum_\ka\|\ka_*u\|_{W_{\coW q}^s\QmE}^q\Bigr)^{1/q} 
\leq c\,\|\cR^cu\|_{\ell_q(W_{\coW p}^s)} 
\end{equation}  
and 
\begin{equation}\label{L.Bs} 
\sup_\ka\|\ka_*u\|_{BC^s\QmE} 
\leq c\,\|\cR u\|_{\ell_\iy(BC^s)}. 
\end{equation}  
On the other hand, 
\hb{\ka_*(\pi_\ka u)=(\ka_*\pi_\ka)\ka_*u} and \eqref{L.LS}(iii) imply 
$$ 
\|\ka_*(\pi_\ka u)\|_{W_{\coW q}^s\RmE} 
\leq c(k)\,\|\ka_*u\|_{W_{\coW q}^s\QmE} 
$$ 
and 
$$ 
\|\ka_*(\pi_\ka u)\|_{BC^s\RmE} 
\leq c(k)\,\|\ka_*u\|_{BC^s\QmE} 
$$ 
for 
\hb{\ka\in\gK}, 
\ \hb{k\in\BN}, and 
\hb{0\leq s\leq k}. Consequently, the left-hand sides of \eqref{L.Ws} 
and \eqref{L.Bs} can be bounded from below by 
\hb{c^{-1}\,\|\cR^cu\|_{\ell_q(W_{\coW q}^s)}}, respectively by 
\hb{c^{-1}\,\|\cR^cu\|_{\ell_\iy(BC^s)}}. 

\smallskip 
It follows from Theorem~\ref{thm-L.l} and general properties of 
retractions and coretractions (e.g.,~(7.8) and (7.9) in~\cite{Ama12c}) 
that 
\hb{u\mt\|\cR^cu\|_{\ell_q(W_{\coW q}^s)}} is an 
equivalent norm for $W_{\coW q}^s(V)$ and 
\hb{u\mt\|\cR^cu\|_{\ell_\iy(BC^s)}} is one for~$BC^s(V)$. 
This implies the first part of the assertion. The last one is now a 
consequence of \eqref{L.ku}, \,\eqref{L.Skk}, and Theorem~\ref{thm-L.l}.  
\end{proof} 
\section{Localizations of Elliptic Operators\label{sec-E}} 
\extraindent 
Unless explicitly stated otherwise, it is assumed that 
\begin{equation}\label{E.ass} 
\left. 
\bal 
{\rm(i)} 
&\quad 0<\bs<1.\cr
{\rm(ii)} 
&\quad \cA={\textstyle\sum_{j=0}^r}a_r\btdot\na^r 
 \text{ is $\bs$-regular and}\cr 
\noalign{\vskip-1\jot} 
&\quad \text{uniformly normally $\ve$-elliptic on }\Mg.\cr 
{\rm(iii)} 
&\quad 0\leq s\leq\bs\text{ and $s<\bs$ if }q<\iy.  
\eal 
\quad\right\} 
\npb 
\end{equation} 
Thus we consider low-regularity autonomous problems. We also suppose 
$$ 
\bt\quad 
(s,q)\text{ is $1$-admissible}. 
$$ 

\smallskip 
For 
\hb{\ka\in\gK} we define~$\ka_*\cA$ by 
\hb{(\ka_*\cA)v=\ka_*\bigl(\cA(\ka^*v)\bigr)} for 
\hb{v\in C^r\QmE}. Then  
$$ 
\ka_*\cA=\sum_{j=0}^r(\ka_*a_j)\btdot\ka_*\na^j. 
$$ 
It follows from Theorem~\ref{thm-L.n} that, setting 
\hb{\cL^j:=\cL^j\RmE},  
\begin{equation}\label{E.kj} 
(\ka_*a_j)_{\ka\in\gK} 
\in\ell_{\iy,\unif}\bigl(buc^{\bs}(Q^m,\cL^j)\bigr) 
\qa 0\leq j\leq r. 
\end{equation} 
Note that 
\begin{equation}\label{E.sk} 
\gss(\ka_*\cA)(\cdot,\xi)=\ka_*\bigl(\gss\cA(\cdot,\ka^*\xi)\bigr) 
\qa \xi\in\BR^m. 
\end{equation} 
It is a consequence of \eqref{L.ur} that 
\hb{|\ka^*\xi|_1^0\sim\ka^*\,|\xi|} for 
\hb{\xi\in\Ga(T^*Q^m)} and 
\hb{\ka\in\gK} (cf.~\cite[Lemma~3.1]{Ama12b}). From this, \eqref{E.sk}, 
$$ 
\ka_*\bigl(\gss\cA(\cdot,\ka^*\xi)\bigr) 
=\ka_*\bigl((|\ka^*\xi|_1^0)^r\gss\cA(\cdot,\ka^*\xi/|\ka^*\xi|_1^0)\bigr), 
$$ 
and the uniform normal \hbox{$\ve$-ellipticity} of~$\cA$ we deduce the 
existence of a constant 
\hb{c\geq1} such that, setting 
\hb{\ve_1:=\ve/c}, 
\begin{equation}\label{E.e} 
\left. 
\bal  
{}
&\ka_*\cA\text{ is uniformly normally $\ve_1$-elliptic on }\Qmgm,\cr 
\noalign{\vskip-1\jot} 
&\text{uniformly with respect to }\ka\in\gK. 
\eal 
\quad\right\} 
\end{equation} 
For 
\hb{\da>0} we denote by 
\hb{h_\da\sco\BR^m\ra\da Q^m} the radial retraction. Thus 
\hb{h_\da(x)=x} if 
\hb{x\in\da Q^m}, and 
\hb{h_\da(x)=\da x/|x|_\iy} otherwise. Note that $h_\da$~is uniformly 
Lipschitz continuous with Lipschitz constant~$2$ 
(cf.~\cite[Lemma~19.8]{Ama90g}). We set 
\begin{equation}\label{E.ak1} 
a_\ka:=(\ka_*a_r)\circ h_1 
\qa \ka\in\gK. 
\end{equation} 
Then 
\begin{equation}\label{E.aa} 
a_\ka(x)=\ka_*a_r(x) 
\qa x\in Q^m , 
\end{equation} 
and 
$$ 
a_\ka\in buc^{\bs}(\BR^m,\cL^r) 
\qb \|a_\ka\|_{\bs,\iy}\leq2\,\|a_r\|_{\bs,\iy} 
\qa \ka\in\gK. 
$$ 
These estimates, \eqref{E.kj}, and \eqref{E.e} imply 
\begin{equation}\label{E.a} 
(a_\ka)\in\ell_{\iy,\unif}\bigl(buc^{\bs}(\BR^m,\cL^r)\bigr) 
\end{equation} 
and 
\begin{equation}\label{E.ke} 
a_\ka\btdot\pl^r\text{ is uniformly normally $\ve_1$-elliptic on }\Rmgm, 
\npb 
\end{equation} 
uniformly with respect to 
\hb{\ka\in\gK}. 

\smallskip 
For each~$\al$ in a countable index set~$\sA$ let $E_\al$ and~$F_\al$ be 
Banach spaces. Then 
\hb{\mf{\cL}(\mfE,\mfF):=\prod_\al\prod_\ba\cL(E_\ba,F_\al)}. 
Using obvious matrix notation, we define a linear map 
\hb{\mfA\sco\mfE\ra\mfF} by 
 $$ 
(\mfA\mfu)_\al:={\textstyle\sum_\ba}A_{\al\ba}u_\ba 
\qa \al\in\sA 
\qb [A_{\al\ba}]\in\mf{\cL}(\mfE,\mfF) 
\qb \mfu=(u_\ba)\in\mfE, 
$$ 
whenever these series converge absolutely in~$E_\al$. We often 
identify~$[A_{\al\ba}]$ with~$\mfA$. Furthermore, 
$$ 
\diag[A_\al]:=[A_\al\da_{\al\ba}]\in\mf{\cL}(\mfE,\mfF) 
\qa A_\al\in\cL(E_\al,F_\al), 
\npb 
$$ 
where $\da_{\al\ba}$ is the Kronecker symbol. 

\smallskip 
We fix \hb{q\in[1,\iy]} and set 
$$ 
\BE^s:= 
\left\{ 
\bal 
{}
&\ell_q(\gF_q^s),
 &\quad 1\leq q<\iy,\cr 
&\ell_{\iy,\unif}(\gF_\iy^s),
 &\quad q=\iy. 
\eal 
\right. 
$$ 
It follows from \eqref{E.a} that 
\begin{equation}\label{E.A} 
A:=\diag[A_\ka]:=\diag[a_\ka\btdot\pl^r]\in\cL(\BE^{s+r},\BE^s). 
\end{equation} 
\begin{lem}\label{lem-E.R} 
There exist 
\begin{equation}\label{E.BB} 
B,B'\in\cL(\BE^{s+r-1},\BE^s) 
\end{equation} 
such that 
\begin{equation}\label{E.AR} 
\cA\circ\cR=\cR\circ(A+B) 
\qb \cR^c\circ\cA=(A+B')\circ\cR^c. 
\end{equation} 
\end{lem} 
\begin{proof} 
(1) 
We set 
\hb{\na_{\cona\ka}:=\ka_*\na} and denote by~%
\hb{\pe} commutators. Then, given 
\hb{u_\ka\in\gF_q^{s+r}}, 
\begin{equation}\label{E.aN} 
a_j\btdot\na^j(\pi_\ka\ka^*u_\ka)
=\pi_\ka\ka^*(\ka_*a_j\btdot\na_{\cona\ka}^ju_\ka) 
+a_j\btdot[\na^j,\pi_\ka]\,\ka^*u_\ka. 
\end{equation} 
We multiply the last term  with 
\hb{1=\sum_\tk\pi_\tk^2} and use 
\hb{\ka^*u_\ka=\tk^*(S_\tkk u_\ka)}. Then it takes the form 
\begin{equation}\label{E.kN} 
\sum_{\tk\in\gN(\ka)}\pi_\tk\tk^* 
\bigl((\tk_*\pi_\tk)\tk_*a_j\btdot
[\na_{\cona\tk}^j,\tk_*\pi_\ka]S_\tkk u_\ka\bigr). 
\end{equation}   
Note that 
\hb{\tsupp(\ka_*\pi_\ka)\is\chi^{-1}(1)} and \eqref{E.aa} imply 
\begin{equation}\label{E.ak2} 
\left. 
\begin{split} 
\pi_\ka\ka^*(\ka_*a_r\btdot\na_{\cona\ka}^ru_\ka) 
&=\ka^*\bigl((\ka_*\pi_\ka)\ka_*a_r\btdot\na_{\cona\ka}^ru_\ka\bigr)\cr  
&=\ka^*\bigl((\ka_*\pi_\ka)a_\ka\btdot\na_{\cona\ka}^ru_\ka\bigr)\cr 
&=\pi_\ka\ka^*(a_\ka\btdot\pl^ru_\ka) 
 +\pi_\ka\ka^*\bigl(a_\ka\btdot(\na_{\cona\ka}^r-\pl^r)u_\ka\bigr)\cr 
&=\pi_\ka\ka^*(A_\ka u_\ka)+\pi_\ka\ka^* 
 \bigl(a_\ka\btdot(\na_{\cona\ka}^r-\pl^r)u_\ka\bigr). 
\end{split} 
\right.  
\end{equation} 
We put 
$$ 
\bal 
B_\tkk u_\ka 
&:=\da_\tkk 
 \Bigl(a_\ka\btdot(\na_{\cona\ka}^r-\pl^r)u_\ka 
 +\sum_{j=0}^{r-1}\ka_*a_j\btdot\na_{\cona\ka}^ju_\ka\Bigr)\cr 
&\ph{{}:={}} 
 +\sum_{j=0}^r(\tk_*\pi_\tk)\tk_*a_j 
 \btdot[\na_{\cona\tk}^j,\tk_*\pi_\ka]S_\tkk u_\ka 
\eal 
$$ 
for 
\hb{\ka\in\gK} and 
\hb{\tk\in\gN(\ka)}, and 
\hb{B_\tkk:=0} if 
\hb{\tk\notin\gN(\ka)}. It follows from \eqref{E.kj}, \,\eqref{E.a}, 
\,\eqref{L.Nj}, \,\eqref{L.Sk}, \,\eqref{L.Skk}, and \eqref{L.LS} that 
\begin{equation}\label{E.Bkk} 
B_\tkk\in\cL(\gF_q^{s+r-1},\gF_q^s) 
\qb \|B_\tkk\|\leq c 
\qa \ka,\tk\in\gK. 
\end{equation} 
From \hbox{\eqref{E.aN}--\eqref{E.ak2}} 
we get, due to \eqref{L.LS},  
$$ 
\cA(\pi_\ka\ka^*u_\ka) 
=\pi_\ka\ka^*(A_\ka u_\ka) 
+\sum_{\tk\in\gN(\ka)}\pi_\tk\tk_*(B_\tkk u_\ka) 
\qa \ka\in\gK. 
$$ 
Now we sum over 
\hb{\ka\in\gK} and interchange the order of summation in the resulting 
double sum. Then we obtain 
\begin{equation}\label{E.ARu} 
\cA(\cR\mfu)=\cR(A\mfu) 
+\cR\Bigl(\Bigl(\sum_\tk B_\ktk u_\tk\Bigr)_{\ka\in\gK}\Bigr) 
\qa \mfu=(u_\ka)\in\BE^{s+r}. 
\end{equation} 
We set 
\hb{B:=[B_\ktk]}. Let 
\hb{k\in\BN} be such that 
\hb{\card\bigl(\gN(\ka)\bigr)\leq k} for 
\hb{\ka\in\gK}. Then $[B_\ktk]$~has for each 
\hb{\ka\in\gK} at most~$k$ non-zero off-diagonal elements. From this and 
\eqref{E.Bkk} it follows that 
$$ 
B\in\cL\bigl(\ell_q(\gF_q^{s+r-1}),\ell_q(\gF_q^s)\bigr). 
$$ 
If 
\hb{\mfu\in\ell_{\iy,\unif}(\gF_\iy^s)}, then it is verified that 
$B\mfu$~belongs to the same space. This proves \eqref{E.BB} for~$B$. 
The first relation of \eqref{E.AR} follows from \eqref{E.ARu}. 

\smallskip 
(2) 
Similarly as above, 
\begin{equation}\label{E.aj} 
\ka_*(\pi_\ka a_j\btdot\na^ju) 
=\ka_*a_j\btdot\na_{\cona\ka}^j\bigl(\ka_*(\pi_\ka u)\bigr) 
-\ka_*a_j\btdot[\na_{\cona\ka}^j,\ka_*\pi_\ka]\,\ka_*u. 
\end{equation} 
Due to \eqref{L.ku}, the last term can be rewritten as 
\begin{equation}\label{E.S} 
-\sum_{\tk\in\gN(\ka)} 
\ka_*a_j\btdot[\na_{\cona\ka}^j,\ka_*\pi_\ka](\ka_*\pi_\tk) 
S_\ktk\bigl(\tk_*(\pi_\tk u)\bigr). 
\end{equation} 
We put, for 
\hb{u_\tk\in\gF_q^{s+r-1}}, 
\begin{equation}\label{E.Bu} 
\left. 
\begin{split} 
B'_\ktk u_\tk
&:=\da_\ktk\Bigl( 
 a_\ka\btdot(\na_{\cona\ka}^r-\pl^r) u_\tk
 +\sum_{j=0}^{r-1}\ka_*a_j\btdot\na_{\cona\ka}^ju_\tk\Bigr)\cr 
&\ph{{}:={}} 
-\ka_*a_j\btdot[\na_{\cona\ka}^j,\ka_*\pi_\ka](\ka_*\pi_\tk)S_\ktk u_\tk 
\end{split} 
\right.  
\npb 
\end{equation} 
if 
\hb{\ka\in\gK} and 
\hb{\tk\in\gN(\ka)}, and 
\hb{B_\ktk'u_\tk:=0} if 
\hb{\tk\notin\gN(\ka)}. Then 
\hb{B':=[B_\ktk']} satisfies \eqref{E.BB}. Furthermore, 
\hbox{\eqref{E.aj}--\eqref{E.Bu}} 
and \eqref{E.ak2} imply 
$$ 
\ka_*(\pi_\ka\cA u) 
=A_\ka\bigl(\ka_*(\pi_\ka u)\bigr) 
+\sum_\tk B_\ktk'\bigl(\tk_*(\pi_\tk u)\bigr) 
\qa \ka\in\gK. 
\npb 
$$ 
This shows that the second relation of \eqref{E.AR} is also satisfied. 
\end{proof} 
\begin{cor}\label{cor-E.R} 
Suppose 
\hb{0\in\rho(A+B)\cap\rho(A+B')}. Then 
\hb{0\in\rho(\cA)} and 
\begin{equation}\label{E.A1} 
\cA^{-1}=\cR\circ(A+B)^{-1}\circ\cR^c. 
\end{equation} 
\end{cor} 
\begin{proof} 
Let 
\hb{u\in\gF_q^{s+r}(V)} satisfy 
\hb{\cA u=0}. Then \eqref{E.AR} implies 
$$ 
0=\cR^c\cA u=(A+B')\cR^cu=0. 
\npb 
$$ 
Hence 
\hb{\cR^cu=0} and, thus, 
\hb{u=\cR\cR^cu=0}. Consequently, $\cA$~is injective. 

\smallskip 
Suppose 
\hb{f\in\gF_q^s(V)}. There is a unique 
\hb{\mfu\in\BE_q^{s+r}} with 
\hb{(A+B)\mfu=\cR^cf}. Setting 
$$ 
u:=\cR\mfu=\cR(A+B)^{-1}\cR^cf\in\gF_q^{s+r}(V), 
$$ 
we get 
$$ 
\cA u=\cA\cR(A+B)^{-1}\cR^cf 
=\cR(A+B)(A+B)^{-1}\cR^cf=\cR\cR^cf=f 
$$ 
by \eqref{E.AR}. Thus $\cA$~is surjective and \eqref{E.A1} applies. Since 
$\cA$~is closed, when considered as a linear  operator in~$\gF_q^s(V)$, 
we get 
\hb{0\in\rho(\cA)}. 
\end{proof} 
\section{Localizations of Parabolic Operators\label{sec-P}} 
\extraindent 
We require again assumption~\eqref{E.ass} and assume that $(s,q)$ is 
\hbox{$r$-admissible}. Then 
$$ 
\mf{\gF}_q^{s/\vec r}:=C\bigl(\gK,\gF_q^{s/\vec r}\BHE\bigr) 
$$ 
and 
\begin{equation}\label{P.Edef} 
\BE^{s/\vec r}:= 
\left\{ 
\bal 
{}
&\ell_q\bigl(\gF_q^{s/\vec r}\BHE\bigr), 
 &\quad 1\leq q<\iy,\cr 
&\ell_{\iy,\unif}\bigl(\gF_\iy^{s/\vec r}\BHE\bigr),
 &\quad       q=\iy. 
\eal 
\right. 
\end{equation} 
We denote the point-wise extension of~$\cRcRc$ to \hbox{$t$-dependent} 
functions again by the same symbol. It is easy to extend 
Theorem~\ref{thm-L.l} to obtain the following analogue. 
\begin{thm(A)}\label{thm-P.l} 
$\cRcRc$~is an \hbox{r-e} pair for 
\hb{\bigl(\BE^{s/\vec r},\gF_q^{s/\vec r}(V\times\BR^+)\bigr)}. 
\end{thm(A)} 
Let the hypotheses of Lemma~\ref{lem-E.R} be satisfied. We set 
\hb{\pl_{t,\ka}:=\pl_t} for 
\hb{\ka\in\gK} and 
\hb{\mfpl_t:=\diag[\pl_{t,\ka}]}. We write 
\hb{\ga_\ka:=\ga_{\pl\BH}} for 
\hb{\ka\in\gK}, where 
$\ga_\BH$~is the trace operator on~$\pl\BH$, and 
\hb{\mfga:=\diag[\ga_\ka]}. 

\smallskip 
The next lemma and its corollary are obvious consequences of the results 
of the preceding section. 
\begin{lem}\label{lem-P.R} 
It holds 
$$ 
\bal 
(\pl_t+\cA)\circ\cR 
&=\cR\circ(\mfpl_t+A+B),\cr 
\cR^c(\pl_t+\cA) 
&=(\mfpl_t+A+B')\circ\cR^c, 
\eal 
$$  
and 
$$ 
\ga\circ\cR=\cR\circ\mfga 
\qb \cR^c\circ\ga=\mfga\circ\cR^c. 
$$ 
\end{lem} 
\begin{cor}\label{cor-P.R} 
Suppose 
$$ 
\bal 
{}
&(\mfpl_t+A+B,\mfga)\text{ and }(\mfpl_t+A+B',\mfga)\cr  
&\text{belong to }\Lis(\BE^{(s+r)/\vec r}, 
\BE^{s/\vec r}\times\BE^{s+r(1-1/q)}). 
\eal 
$$ 
Then 
$$ 
(\pl_t+\cA,\ga) 
\in\Lis\bigl(\gF_q^{(s+r)/\vec r}(V\times\BR^+), 
\gF_q^{s/\vec r}(V\times\BR^+)\times\gF_q^{s+r(1-1/q)}(V)\bigr) 
$$ 
and 
$$ 
(\pl_t+\cA,\ga)^{-1} 
=\cR\circ(\mfpl_t+A+B,\mfga)^{-1}\circ(\cR^c\times\cR^c). 
$$ 
\end{cor} 
\section{The Flat Case\label{sec-N}} 
\extraindent 
Now we assume 
\newsavebox{\Eass}
\sbox{\Eass}{\eqref{E.ass}} 
\newcommand*{\Eassspez}{\usebox{\Eass}} 
\begin{equation}\label{N.ass} 
\left. 
\bal 
{\rm(i)} 
&\quad \Mg=\Rmgm.\cr 
{\rm(ii)} 
&\quad \text{Assumption \Eassspez is satisfied}.\cr 
{\rm(iii)} 
&\quad 
 \ve^{-1}+{\textstyle\sum_{j=0}^r}\|a_j\|_{\bs,\iy}\leq\bka. 
\eal 
\quad\right\} 
\end{equation} 
Recall from \eqref{L.TE} that 
\hb{V=\BR^m\times E}. We also suppose that 
$$ 
\bt\quad 
(s,q)\text{ is $r$-admissible} 
\npb 
$$ 
and write 
\hb{X_\eta^j:=\gF_{q;\eta}^{(s+jr)/\vec r}\BHE} for 
\hb{j=0,1} and 
\hb{Y_{\coY\eta}:=\gF_{q;\eta}^{s+r(1-1/q)}\RmE}. 

\smallskip 
It follows from \eqref{E.ass} that the constant coefficient operator 
\hb{a_r(x)\btdot\pl^r} is normally \hbox{$\ve$-elliptic} and 
\hb{|a_r(x)|_{\cL^r}\leq\|a_r\|_{\bs,\iy}}, uniformly with respect to 
\hb{x\in\BR^m}. Hence Theorem~\ref{thm-C.C} implies 
\begin{equation}\label{N.ax} 
\bigl(\pl_t+\eta+a_r(x)\btdot\pl^r,\ga\bigr) 
\in\Lis(X^1,X^0\times Y)  
\end{equation} 
and there exists 
\hb{c_0=c_0(\bka)} such that, for  
\hb{x\in\BR^m} and 
\hb{\eta>0}, 
\begin{equation}\label{N.ax1} 
\big\|\bigl(\pl_t+\eta+a_r(x)\btdot\pl^r,\ga\bigr)^{-1}\big\| 
_{\Lis(X_\eta^0\times Y_{\coY\eta},X_\eta^1)}\leq c_0. 
\end{equation} 

\smallskip 
Since 
\hb{a_r\in buc^{\bs}(\BR^m,\cL^r)}, 
\begin{equation}\label{N.a0} 
\big|a_r\bigl(\da(x+z)\bigr)-a_r(\da z)\big|_{\cL^r} 
\leq c\da^{\bs} 
\qa x\in Q^m 
\qb z\in\BZ^m, 
\end{equation} 
and 
\begin{equation}\label{N.ad} 
\sup_{x,y\in Q^m} 
\frac{\big|a_r\bigl(\da(x+z)\bigr)-a_r\bigl(\da(y+z)\bigr)\big|_{\cL^r}} 
{|\da(x-y)|^{\bs}}\ra0 
\quad\text{as }\da\ra0,  
\end{equation} 
uniformly with respect to 
\hb{z\in\BZ^m}. With the radial retraction~$h_\da$ we put 
$$ 
\sfa_{z,\da}(x):=a_r\bigl(\da z+h_\da(x-\da z)\bigr) 
\qa x\in\BR^m 
\qb z\in\BZ^m. 
$$ 
Then, as in \eqref{E.a}, 
\begin{equation}\label{N.a} 
(\sfa_{z,\da})_{z\in\BZ^m} 
\in\ell_{\iy,\unif}\bigl(buc^\bs(\BR^m,\cL^r)\bigr) 
\end{equation} 
(where we now employ the index set~$\BZ$), and 
$$ 
\sfa_{z,\da}(x)=a_r(x) 
\qa x\in\da(z+Q^m). 
$$ 
From this, \eqref{N.a0}, \,\eqref{N.ad}, and Theorems \ref{thm-B.F}(ii) 
and~\ref{thm-B.M} we infer that 
$$ 
\big\|\bigl(\sfa_{z,\da} 
-a_r(\da z)\bigr)\btdot\pl^r\big\|_{\cL(X_\eta^1,X_\eta^0)} 
\leq c\,\big\|a_r\bigl(\da(z+\cdot)\bigr) 
-a_r(\da z)\big\|_{BC^{\bs}(Q^m,\cL^r)}\ra0 
$$ 
as 
\hb{\da\ra0}, uniformly with respect to 
\hb{z\in\BZ^m}. Hence we can fix 
\hb{\da=\da(\bka)\in(0,1)} such that 
\begin{equation}\label{N.d} 
\big\|\bigl(\sfa_{z,\da} 
-a(\da z)\btdot\pl^r\bigr)\big\|_{\cL(X_\eta^1,X_\eta^0)}\leq1/2c_0, 
\npb 
\end{equation} 
uniformly with respect to 
\hb{\eta>0} and 
\hb{z\in\BZ^m}. 

\smallskip 
We set 
\hb{\ka_z(x):=-z+x/\da} for 
\hb{x\in U_{\coU\ka_z}:=\da(z+Q^m)} and 
\hb{z\in\BZ^m}. Then 
\hb{\gK:=\{\,\ka_z\ ;\ z\in\BZ^m\,\}} is a uniformly regular atlas 
for~$\Mg$. We fix a localization system 
\hb{\bigl\{\,(\pi_\ka,\chi_\ka)\ ;\ \ka\in\gK\,\bigr\}} subordinate 
to~$\gK$ and put 
\begin{equation}\label{N.R} 
R\mfu:=\sum_\ka\pi_\ka u_\ka 
\qb R^cu:=(\pi_\ka u) 
\npb 
\end{equation} 
for 
\hb{\mfu=(u_\ka)\in\mf{\gF}_q^{s/\vec r}} and 
\hb{u\in\gF_q^{s/\vec r}(V\times\BR^+)}. 

\smallskip 
The following lemma is a parameter-dependent equivalent of 
Theorem~\ref{thm-L.l}. Its proof, however, is much simpler since the 
atlas~$\gK$ is not explicitly involved. 
\begin{lem}\label{lem-N.R} 
$\RRc$ is an \hbox{$\eta$-uniform} \hbox{r-e} pair for 
\hb{\bigl(\BE_{q;\eta}^{s/\vec r}, 
   \gF_{q;\eta}^{s/\vec r}(V\times\BR^+)\bigr)}. 
\end{lem} 
For easy reference we include the following well-known perturbation theorem. 
\begin{lem}\label{lem-N.P} 
Let $X$ and~$Y$ be Banach spaces and 
\hb{a\in\Lis\XY}. Suppose 
\hb{b\in\cL\XY} satisfies 
\hb{\|ba^{-1}\|\leq1/2}, then 
\hb{a+b} belongs to $\Lis\XY$ and 
\hb{\|(a+b)^{-1}\|\leq2\,\|a^{-1}\|}. 
\end{lem} 
\begin{proof} 
A~Neumann series argument shows that 
$$ 
1+ba^{-1}\in\Laut(Y) 
\text{ and } 
\|(1+ba^{-1})^{-1}\|\leq2. 
\npb 
$$ 
Hence the claim follows from 
\hb{a+b=(1+ba^{-1})a}. 
\end{proof} 

We set 
$$ 
\sA_\ka:=a_{z,\da}\btdot\pl^r\text{ for }\ka=\ka_z\in\gK 
$$ 
and 
$$ 
\BX_\eta^j:=\BE_{q;\eta}^{(s+jr)/\vec r}\text{ for }j=0,1 
\qb \BY_{\coY\eta}:=\BE_{q;\eta}^{s+r(1-1/q)}. 
\npb 
$$ 
Clearly, $\BX_\eta^j$~is obtained by replacing~$X^j$ in~\eqref{P.Edef} 
by $X_\eta^j$,~etc. 
\begin{lem}\label{lem-N.A} 
Set 
\hb{\sA:=\diag[\sA_\ka]}. Then 
\hb{(\mfpl_t+\eta+\sA,\mfga)\in\Lis(\BX^1,\BX^0\times\BY)} and 
$$ 
\|(\mfpl_t+\eta+\sA,\mfga)^{-1}\| 
_{\cL(\BX_\eta^0\times\BY_{\coY\eta},\BX_\eta^1)} 
\leq c(\bka) 
\quad\text{$\eta$-uniformly}. 
$$ 
\end{lem} 
\begin{proof} 
We put 
\hb{\sA_\ka^0:=a(\da z)\btdot\pl^r} for 
\hb{\ka=\ka_z}. Then \eqref{N.ax} and \eqref{N.ax1} imply that 
\hb{(\pl_t+\eta+\cA_\ka^0,\ga_\ka)} is an isomorphism from~$X^1$ onto 
\hb{X^0\times Y}, and 
\begin{equation}\label{N.Ak0} 
\|(\pl_t+\eta+\sA_\ka^0,\ga_\ka)^{-1}\| 
_{\cL(\BX_\eta^0\times\BY_{\coY\eta},\BX_\eta^1)} 
\leq c_0 
\qa \ka\in\gK 
\qb \eta>0. 
\end{equation} 
Set 
\hb{\sB_\ka:=\sA_\ka-\sA_\ka^0\in\cL(X^1,X^0)}. Then 
\hb{\|\sB_\ka\|_{\cL(\BX_\eta^1,\BX_\eta^0)}\leq1/2c_0} by \eqref{N.d}, 
uniformly with respect to 
\hb{\ka\in\gK} and 
\hb{\eta>0}. Hence it follows from 
$$ 
(\pl_t+\eta+\sA_\ka,\ga_\ka) 
=(\pl_t+\eta+\sA_\ka^0,\ga_\ka)+(\sB_\ka,0), 
$$ 
estimate~\eqref{N.Ak0}, and Lemma~\ref{lem-N.P} that 
\hb{(\pl_t+\eta+\sA_\ka,\ga_\ka)\in\Lis(X^1,X^0\times Y)} and 
$$ 
\|(\pl_t+\eta+\sA_\ka,\ga_\ka)^{-1}\| 
_{\cL(X_\eta^0\times Y_{\coY\eta},X_\eta^1)}\leq2c_0, 
$$ 
uniformly with respect to 
\hb{\ka\in\gK} and 
\hb{\eta>0}. Now, taking \eqref{N.a} into consideration, 
the assertion is clear. 
\end{proof} 
The next lemma is an analogue to \ref{lem-P.R} in the present setting. 
Its proof is obtained by simplifying the demonstration of 
Lemma~\ref{lem-E.R} based on the fact that the local charts do not occur 
in~\eqref{N.R}. 

\smallskip 
We set 
\hb{\BW_\eta:=\ell_q(\gF_{q;\eta}^{(s+r-1)/\vec r})} if 
\hb{q<\iy}, and 
\hb{\BW_\eta:=\ell_{\iy,\unif}(\gF_{\iy;\eta}^{(s+r-1)/\vec r})} if 
\hb{q=\iy}. 
\begin{lem}\label{lem-N.AR} 
There exist 
\hb{\sB,\sB'\in\cL(\BW_\eta,\BX_\eta^0)} such that 
\begin{equation}\label{N.AR} 
\left. 
\begin{split} 
(\pl_t+\cA_\eta)\circ R 
&=R\circ(\mfpl_t+\eta+\sA+\sB),\cr 
R^c\circ(\pl_t+\cA_\eta) 
&=(\mfpl_t+\eta+\sA+\sB')\circ R^c. 
\end{split} 
\right.  
\end{equation} 
\end{lem} 
Now we are ready to prove the main result of this section. Observe that 
\hb{V\times J\wheq\BH\times E}. 
\begin{thm(A)}\label{thm-N.A} 
Let \eqref{N.ass} be satisfied. There exists 
\hb{\eta_0=\eta_0(\bka)\geq1} such that 
\hb{(\pl_t+\cA_\eta,\ga)\in\cL(X^1,X^0\times Y)} and 
$$ 
\|(\pl_t+\cA_\eta,\ga)^{-1}\|_{\cL(X_\eta^0\times Y_{\coY\eta},X_\eta^1)} 
  \leq c(\bka) 
\qa \eta\geq\eta_0. 
$$ 
\end{thm(A)} 
\begin{proof} 
Theorem~\ref{thm-B.F}(i) guarantees 
\hb{\gF_q^{(s+r)/\vec r}\hr\gF_q^{(s+r-1)/\vec r}} and 
$$ 
\Vsdot_{(s+r+1)/\vec r,q;\eta} 
\leq c\kern1pt\eta^{-1}\,\Vsdot_{(s+r)/\vec r,q;\eta} 
\qa \eta>0, 
\npb 
$$ 
This implies 
\hb{\BX_\eta^1\hr\BW} and 
\hb{\Vsdot_{\BW_\eta}\leq c\kern1pt\eta^{-1}\,\Vsdot_{\BX_\eta^1}} for 
\hb{\eta>0}. 

\smallskip 
We write~$c_0$ for the constant~$c(\bka)$ of Lemma~\ref{lem-N.A}. 
Then we get 
$$ 
\|(\mfpl_t+\eta+\sA,\mfga)^{-1}\| 
_{\cL(\BX_\eta^0\times\BY_{\coY\eta},\BW_\eta)} 
\leq c_0/\eta 
\qa \eta>0. 
$$ 
Lemma~\ref{lem-N.AR} guarantees the existence of 
\hb{c_1\geq1} such that 
$$ 
\|\sB\|_{\cL(\BW_\eta,\BX_\eta^0)} 
+\|\sB'\|_{\cL(\BW_\eta,\BX_\eta^0)} 
\leq c_1 
\quad\text{$\eta$-uniformly}. 
$$ 
Hence, setting 
\hb{\eta_0:=2c_1c_0^2}, we find 
$$ 
\big\|\bigl(\sB\circ(\mfpl_t+\eta+\sA,\mfga)^{-1},0\bigr)\big\| 
_{\cL(\BX_\eta^0\times\BY_{\coY\eta})} 
\leq1/2c_0 
\qa \eta\geq\eta_0. 
$$ 
From this and Lemma~\ref{lem-N.P}  we obtain that 
\hb{(\mfpl_t+\eta+\sA+\sB,\mfga)} belongs to 
\hb{\Lis(\BX^1,\BX^0\times\BY)} and 
\begin{equation}\label{N.dAB} 
\|(\mfpl_t+\eta+\sA+\sB,\ga)^{-1}\| 
_{\cL(\BX_\eta^0\times\BY_{\coY\eta},\BX_\eta^1)} 
\leq2c_0 
\qa \eta\geq\eta_0. 
\npb 
\end{equation} 
The same argument shows that \eqref{N.dAB} holds with $\sB$ replaced 
by~$\sB'$. 

\smallskip 
It is obvious that 
$\ga\circ R=R\circ\mfga$ and 
$R^c\circ\nolinebreak\mfga=\ga\circ\nolinebreak R^c$. Using this and 
\eqref{N.AR}, 
the assertion thus follows from Lemma~\ref{lem-N.R} and (the analogue of) 
Corollary~\ref{cor-P.R}. 
\end{proof} 
Now we assume that $(s,q)$ is \hbox{$1$-admissible}. Going through the 
above proofs, neglecting any reference to 
\hb{t\in\BR^+}, using Theorem~\ref{thm-S.R} instead of 
Theorem~\ref{thm-C.C}, and appealing to Corollary~\ref{cor-E.R} 
instead of Corollary~\ref{cor-P.R}, etc., we obtain the following resolvent 
estimate. Details are left to the reader. 
\begin{thm(A)}\label{thm-N.E} 
Let assumption~\eqref{N.ass} be satisfied, but assume that $(s,q)$ is 
\hbox{$1$-admissible} . Then there exist 
\hb{\eta_0=\eta_0(\bka)\geq1} such that 
\hb{\lda+\cA_\eta} belongs to $\Lis(\gF_q^{s+r},\gF_q^s)$ and 
$$ 
(|\lda|+\eta)^{1-j}\,\|(\lda+\cA_\eta)^{-1}\| 
_{\cL(\gF_{q;\eta}^s,\gF_{q;\eta}^{s+j+r})} 
\leq c(\ka) 
\npb 
$$ 
for 
\hb{\Re\lda\geq0} and 
\hb{\eta>0}. 
\end{thm(A)} 
\section{Proof of the Main Theorems\label{sec-T}} 
\extraindent 
After all the preparation in the preceding sections it is no longer too 
difficult to demonstrate the validity of Theorems \ref{thm-I.P} 
and~\ref{thm-I.S}. 
\begin{proofTheoremI.P} 
\extraroman 
First we observe that the assumptions on~$(s,q)$, where 
\hb{q:=\iy} in claim~(ii), amount to: 
$(s,q)$~is \hbox{$r$-admissible}. We fix~$\bka$ satisfying 
$$ 
\ve^{-1}+\sum_{j=0}^r\|a_j\|_{\bs/\vec r,\iy}\leq\bka. 
$$ 

\smallskip 
(1) 
Assume 
\hb{0<\bs<1} and $\cA$~is independent of 
\hb{t\in\BR^+}. Define~$a_\ka$ by \eqref{E.ak1}. It follows from \eqref{E.a}, 
\,\eqref{E.ke}, and Theorem~\ref{thm-N.A} that there are 
\hb{\eta_0=\eta_0(\bka)\geq1} and 
\hb{c_0=c_0(\bka)\geq1} such that 
$$ 
(\pl_t+\eta+a_\ka\btdot\pl^r,\ga_\ka) 
\in\Lis(X^1,X^0\times Y) 
$$ 
and 
$$ 
\|(\pl_t+\eta+a_\ka\btdot\pl^r,\ga_\ka)^{-1}\|_  
{\cL(X_\eta^0\times Y_{\coY\eta},X_\eta^1)} \leq c_0, 
$$ 
uniformly with respect to \hb{\eta\geq\eta_0} and 
\hb{\ka\in\gK}. From this, \eqref{E.A}, and Theorem~\ref{thm-B.I} we infer 
\hb{(\mfpl_t+\eta+A,\mfga)\in\Lis(\BX^1,\BX^0\times\BY)} and 
$$ 
\|(\mfpl_t+\eta+A,\mfga)^{-1}\|_  
{\cL(\BX_\eta^0\times\BY_{\coY\eta},\BX_\eta^1)} \leq c_0 
\qa \eta\geq\eta_0. 
$$ 
Using \eqref{E.BB} and the arguments of the proof of Theorem~\ref{thm-N.A} 
we see that we can find 
\hb{\eta\geq\eta_0\geq1} so that 
$$ 
(\mfpl_t+\eta+A+B,\mfga),(\mfpl_t+\eta+A+B',\mfga) 
\in\Lis(\BX^1,\BX^0\times\BY) 
\npb 
$$ 
and the inverses of these linear operators are bounded by~$c(\bka)$. 

\smallskip 
Set 
\hb{\cX^j:=\gF_q^{(s+jr)/\vec r}(V\times\BR^+)} for 
\hb{j=0,1}, and 
\hb{\cY:=\gF_q^{s+r(1-1/q)}(V)}. Then Corollary~\ref{cor-P.R} implies 
$$ 
(\pl_t+\cA_\eta,\ga)\in\Lis(\cX^1,\cX^0\times\cY) 
\qb \|(\pl_t+\cA_\eta,\ga)^{-1}\|_{\cL(\cX^0\times\cY,\cX^1)} 
\leq c(\tk).  
$$ 

\smallskip 
(2) 
Suppose 
\hb{0<\bs<1}. We write 
$\cX^j(S):=\pagebreak[0]\gF_q^{(s+jr)/\vec r} 
   \bigl(V\times\pagebreak[0][0,S]\bigr)$ for 
\hb{S>0}. Given 
\hb{\tau\in J=J_T}, we denote by 
\hb{\pl_t+\cA(\tau)} the autonomous operator whose coefficients 
are frozen at 
\hb{t=\tau}. Then $\cA(\tau)$~is \hbox{$\bs$-regular} and normally 
\hbox{$\ve$-elliptic}, uniformly with respect to 
\hb{\tau\in J}. Thus, by step~(1), 
\hb{\bigl(\pl_t+\cA_\eta(\tau),\ga\bigr)} belongs to 
$\Lis(\cX^1,\cX^0\times\cY)$ and 
\begin{equation}\label{T.k} 
\big\|\bigl(\pl_t+\cA_\eta(\tau),\ga\bigr)^{-1}\| 
_{\cL(\cX^0\times\cY,\cX^1)}\leq c(\bka) 
\qa 0\leq\tau\leq T. 
\end{equation} 
The fact that the coefficients of~$\cA$ belong to 
\hb{bc^{\bs/\vec r}(V\times J)} implies (similarly as in Section~\ref{sec-N}) 
\begin{equation}\label{T.0} 
\|\cA(\tau+\cdot)-\cA(\tau)\|_{\cL(\cX^1(S),\cX^0(S))}\ra0 
\quad\text{as }S\ra0, 
\end{equation} 
uniformly with respect to 
\hb{\tau\in J}. Since 
$$ 
(\pl_t+\cA_\eta,\ga) 
=\bigl(\pl_t+\cA_\eta(\tau),\ga\bigr) 
+\bigl(\cA(\tau+\cdot)-\cA(\tau),0\bigr) 
\text{ on }M\times[\tau,T], 
$$ 
we infer from \eqref{T.k}, \,\eqref{T.0}, and Lemma~\ref{lem-N.P} that there 
exist 
\hb{S\in(0,T)} and 
\hb{k\in\thBN} such that 
$$ 
\bigl(\pl_t+\cA_\eta(jS+\cdot)\bigr)v=f(jS+\cdot) 
\qb \ga v=w 
$$ 
has for each 
\hb{w\in\cY} a~unique solution 
\hb{V_{\coV j}(w)\in\cX^1(S)} if 
\hb{0\leq j\leq k-1}, and a~unique solution 
\hb{V_k(w)\in\cX^1\bigl(\min\{S,\,T-kS\}\bigr)}. We set 
\hb{v_0:=V_0(u_0)} and 
\hb{v_i:=V_{\coV i}\bigl(v_{i-1}(S)\bigr)} for 
\hb{1\leq i\leq k}. For 
\hb{t=iS+s} we define~$u$ by 
\hb{u(t):=v_i(s)}, where 
\hb{0\leq i<k} and 
\hb{0\leq s\leq\min\{S,\,T-iS\}}. The trace theorem shows that 
$u$~belongs to 
\hb{\gF_q^{(s+r)/\vec r}(V\times J)} and is the unique solution of 
\hb{(\pl_t+\cA_\eta)u=f} on 
\hb{V\times J} satisfying 
\hb{\ga u=u_0}. 

\smallskip 
(3) 
Let 
\hb{0<\bs<1}. Set 
\hb{f^\eta:=e^{t\eta}f}. Then 
\hb{u\in\cX^1(T)} satisfies 
\hb{(\pl_t+\cA)u=f^\eta} and 
\hb{\ga u=u_0} iff 
\hb{u=e^{t\eta}v} and 
\hb{v\in\cX^1(T)} conforms to 
\hb{(\pl_t+\cA_\eta)v=f} and 
\hb{\ga v=u_0}. Since 
\hb{f\mt f^\eta} is an automorphism of~$\cX^0(T)$, we see from the preceding 
step that the theorem holds under the present additional hypothesis. 

\smallskip 
(4) 
We put 
\hb{\cY^s(V):=\gF_q^{s/\vec r}(V\times J)} and 
\hb{\cZ^s(V):=\gF_q^{s+r(1-1/q)}(V)}. Suppose 
\begin{equation}\label{T.sr} 
r\leq s<\bs<r+1 
\end{equation} 
and set 
\hb{s_0:=s-r}. Let $(f,u_0)$ belong to 
\hb{\cY^s(V)\times\cZ^s(V)}. Since 
$$ 
\cY^s(V)\times\cZ^s(V)\hr\cY^{s_0}(V)\times\cZ^{s_0}(V), 
$$ 
it follows from what we have already shown that there exists a unique 
\newline 
\hb{u\in\cY^{s_0+r}(V)=\cY^s(V)} satisfying 
\hb{(\pl_t+\cA)u=f} on 
\hb{V\times J} and 
\hb{\ga u=u_0}. Let 
\hb{1\leq i\leq r}. By applying~$\na^i$ we get 
$$ 
(\pl_t+\cA)\na^iu=\cA_iu+\na^if 
\text{ on }V_{\tau+i}^\sa\times J 
\qb \ga\na^iu=\na^iu_0, 
$$ 
where 
$$ 
\cA_iu:= 
-\sum_{\vph{k}j=0}^r\sum_{k=1}^i\bid ik\na^ka_j\btdot\na^{i-k}u. 
$$ 
Note 
$$ 
\na^{i-k}u\in\cY^{s-i+k}(V_{\coV\tau+i-k}^\sa) 
\hr\cY^{s_0}(V_{\coV\tau+i-k}^\sa) 
$$ 
and 
$$ 
\na^ka_j 
\in bc^{(\bs-k)/\vec r} 
\bigl(V_{\coV\tau+\sa+k}^{\sa+\tau+j}(\cL(F))\bigr) 
\hr bc^{(\bs-r)/\vec r} 
\bigl(V_{\coV\tau+\sa+k}^{\sa+\tau+j}(\cL(F))\bigr). 
$$ 
From this it follows, due to 
\hb{\bs-k\geq s_0} with 
\hb{\bs-k>s_0} if 
\hb{q<\iy}, that 
$$ 
\cA_ju+\na^if\in\cY^{s_0}(V_{\coV\tau+i}^\sa) 
\qb \na^iu_0\in\cZ^{s_0}(V_{\coV\tau+i}^\sa). 
$$ 
Hence the results of the preceding step guarantee that 
\begin{equation}\label{T.iu} 
\na^iu=(\pl_t+\cA,\ga)^{-1}(\cA_iu+\na^if,\na^iu_0) 
\in\cY^s(V_{\coV\tau+i}^\sa) 
\qa 1\leq i\leq r. 
\end{equation} 
Analogously, 
\begin{equation}\label{T.dtA} 
(\pl_t+\cA)\pl_tu=\cA_0u+\pl_tf 
\text{ on }V\times J 
\qb \ga\pl_tu=-\cA(0)u_0+f(0), 
\end{equation} 
where 
\begin{equation}\label{T.A0} 
\cA_0u:= 
-\sum_{j=0}^r\pl_ta_j\btdot\na^ju\in\cY^{s_0}(V), 
\end{equation} 
due to 
\hb{\pl_ta_j\in bc^{(\bs-r)/\vec r} 
   \bigl(V_{\coV\tau+\sa}^{\sa+\tau+j}(\cL(F))\bigr)}. It also follows from 
the trace theorem that 
$$ 
-\cA(0)u_0+f(0)\in\cZ^{s_0}(V). 
$$ 
Now we infer from \hbox{\eqref{T.dtA}, \,\eqref{T.A0}}, 
and the results of step~(3) that 
\begin{equation}\label{T.dt} 
\pl_tu= 
(\pl_t+\cA,\ga)^{-1} 
\bigl(\cA_0u+\pl_tf,-\cA(0)u_0+f(0)\bigr) 
\in\cY^s(V). 
\end{equation} 
It follows from\eqref{T.iu}, \,\eqref{T.dt}, 
and \eqref{I.WW}, \,\eqref{I.bcb} that 
\hb{u\in\cY^{s+r}(V)}. It is not difficult to check that the map 
\hb{(f,u^0)\mt u} is continuous from 
\hb{\cY^s(V)\times\cZ^s} onto~$\cY^{s+r}(V)$. 
This proves the theorem if \eqref{T.sr} is satisfied. 

\smallskip 
(5) 
Assume 
\hb{r<\bs<r+1} and 
\hb{1<s<r} with 
\hb{s\notin\BN}. Choose 
\hb{s_0\in(0,1)} and 
\hb{s_1\in(r,\bs)}. Then it follows from steps (3) and~(4) that 
\begin{equation}\label{T.sj} 
\kern-5pt 
(\pl_t+\cA,\ga)\in\Lis\bigl(\gF_q^{(s_j+r)/\vec r}(V\times J), 
 \gF_q^{s_j/\vec r}(V\times J)\times\gF_q^{s_j+r(1-1/q)}(V)\bigr) 
\end{equation} 
for 
\hb{j=0,1}. In \cite{Ama17a} it is shown that, setting 
\hb{\ta:=(s-s_0)/(s_1-s_0)}, 
$$ 
\bigl(\gF_q^{(s_0+r)/\vec r}(V\times J), 
\gF_q^{(s_1+r)/\vec r}(V\times J)\bigr)_{\ta,q}^0 
\doteq\gF_q^{(s+r)/\vec r}(V\times J) 
$$ 
and 
$$ 
\bigl(\gF_q^{s_0+r(1-1/q)}(V), 
\gF_q^{s_1+r(1-1/q)}(V)\bigr)_{\ta,q}^0 
\doteq\gF_q^{s+r(1-1/)}(V). 
$$ 
Thus we get the assertion in the present case from\eqref{T.sj} by 
interpolation. This proves the claim for 
\hb{0\leq s\leq\bs} with 
\hb{r<\bs<r+1}, provided $(s,q)$ is \hbox{$r$-admissible}. 
The general case follows now by induction.\nolinebreak\hfill\nolinebreak\qed
\end{proofTheoremI.P} 
\begin{proofTheoremI.S} 
\extraroman 
We modify the preceding proof by omitting~$t$ and all considerations with 
reference to it and invoke Theorem~\ref{thm-N.E} instead of~\ref{thm-N.A}. 
As for the analogue to step~(4), we use the fact that 
\hb{u\in\gF_q^{s+1}(V)} iff 
\hb{u\in\gF_q^s(V)} and 
\hb{\na u\in\gF_q^s(V_{\coV\tau+1}^\sa)}. Hence interpolation is not needed 
here. Then we get the existence of 
\hb{\eta\geq1} such that 
\hb{\lda+\cA\in\Lis\bigl(\gF_q^{s+r}(V),\gF_q^s(V)\bigr)} and 
$$ 
\|(\lda+\cA)^{-1}\| 
_{\cL(\gF_q^s(V),\gF_q^{s+jr}(V))}\leq c\big/(1+|\lda|)^{1-j} 
$$ 
for 
\hb{\Re\lda\geq\eta} and 
\hb{j=0,1}. This proves the claim due to the density of 
$\gF_q^{s+r}(V)$ in~$\gF_q^s(V)$.\nolinebreak\hfill\nolinebreak\qed
\end{proofTheoremI.S}

\def\cprime{$'$} \def\polhk#1{\setbox0=\hbox{#1}{\ooalign{\hidewidth
  \lower1.5ex\hbox{`}\hidewidth\crcr\unhbox0}}}

\let\vec\oldvec
\end{document}